\documentclass{amsart}

\usepackage{amstext,amssymb,amsmath,amsbsy,mathrsfs,mathtools,amsfonts,amscd,exscale}
\usepackage{graphics, epstopdf}
\usepackage{color}
\usepackage{amsmath, amssymb}
\usepackage{enumerate}
\usepackage{bigints}
\usepackage{verbatim}
\usepackage[poorman]{cleveref}

\usepackage{graphicx}
\usepackage{epstopdf}
\usepackage{float}
\usepackage{placeins}
\usepackage{float}
\usepackage[caption = false]{subfig}

\usepackage[normalem]{ulem}

\newcommand{\mbV}{\mathbb{V}}
\newcommand{\bE}{\mathbb{E}}

\newcommand{\eps}{\varepsilon}

\renewcommand{\sc}{\textsc}

\usepackage{cancel}

\newtheorem{Corollary}{Corollary}[section]
\newtheorem{Theorem}{Theorem}[section]
\newtheorem{Lemma}[Theorem]{Lemma}
\newtheorem{Proposition}[Theorem]{Proposition}

\theoremstyle{definition}
\newtheorem{Definition}[Theorem]{Definition}
\newtheorem{example}[Theorem]{Example}
\newtheorem{remark}[Theorem]{Remark}

\theoremstyle{remark}

\numberwithin{equation}{section}

\def\l({\left(}
\def\r){\right)}

\def\nd{\partial_{\mu}}

\def\LtG{L^2(\Gamma_h)}

\def\dt{ \tau }
\def\hb{{h}}

\def\Vhk{\mathbb{V}_h^k }

\def\Hhk{\mathbb{H}_h^k }

\def\lom{{L^2(\Omega)}}
\def\hT{{H^1(T)}}
\def\hom{{H^1(\Omega)}}

\def\htwom{{H^2(\Omega)}}

\def\Ye{y^{\sigma}}

\def\by{\boldsymbol{\delta Y}_h}
\def\bz{\boldsymbol{0}}
\def\bla{\boldsymbol{\Lambda}_h}
\def\bF{\boldsymbol{F}_n}

\def\npow{^{n+1}}
\def\mQ{\mathbb{Q}}

\def\mP{\mathbb{P}}

\def\calE{\mathcal{E}}
\def\Ei{\mathcal{E}_h^0}
\def\Eb{\mathcal{E}_h^b}

\def\isomh{(\nabla_h y_h)^T \nabla_h y_h}

\def\sumT{\sum_{T \in \mathcal{T}_h}}

\def\Ahk{\mathbb{A}_{h,\eps}^k} 

\renewcommand{\div}{\textrm{div}\,}

\def\hT{\widehat{T}}

\def\hx{\widehat{x}}
\def\hv{\widehat{v}}
\def\hw{\widehat{w}}
\def\hp{\widehat{p}}
\def\hI{\widehat{I}}

\definecolor{red}{rgb}{1,0,0}
\definecolor{blue}{rgb}{0,0,1}
\definecolor{aggiemaroon}{rgb}{0.8,0,0}
\definecolor{purple}{rgb}{1,0,1}

\newcommand{\rhn}[1]{{\color{black}#1}}

\begin{document}

	\title[Large Isometric Bending Plate Deformations]{DG Approach to Large Bending Plate Deformations with Isometry Constraint}

	\author[A. Bonito]{Andrea Bonito$^1$}
	\address{Department of Mathematics, Texas A\&M University, College Station, TX 77843}
	\email{bonito@math.tamu.edu}
	\thanks{$^1$ Partially supported by the NSF Grant DMS-1817691.}
	
	\author[R.H. Nochetto]{Ricardo H. Nochetto$^2$}
	\address{Department of Mathematics and Institute for Physical Science
		and Technology, University of Maryland, College Park, Maryland 20742}
	\email{rhn@math.umd.edu}
	\thanks{$^2$ Partially supported by the NSF Grants DMS-1411808 and DMS-1908267,
          the Institut Henri Poincar\'e (Paris) and the Hausdorff Institute (Bonn).}

	\author[D. Ntogkas]{Dimitrios Ntogkas$^3$}
	\address{Department of Mathematics, University of Maryland, College Park, Maryland 20742}
	\email{dimnt@math.umd.edu}
	\thanks{$^3$ Partially supported by the  NSF Grant DMS-1411808 and
		the 2016-2017 Patrick and Marguerite Sung Fellowship of the
		University of Maryland.}

\begin{abstract}
   We propose a new discontinuous Galerkin (dG) method for a geometrically
    nonlinear Kirchhoff plate model for large isometric bending deformations. The
    minimization problem is nonconvex due to the isometry constraint. We present a
    practical discrete gradient flow that decreases the energy and computes discrete
    minimizers that satisfy a prescribed discrete isometry defect. We prove
    $\Gamma$-convergence of the discrete energies and discrete global minimizers.
    We document the flexibility and accuracy of the dG method with several numerical
    experiments.
\end{abstract}

\keywords{Nonlinear elasticity; plate bending; isometry constraint; discontinuous Galerkin; iterative solution; $\Gamma$-convergence.}

\maketitle
  
	\section{Introduction} \label{S:Introduction}
	
	Large bending deformations of thin plates is a critical \rhn{feature} for many modern engineering 
	applications due to the extensive use of plate actuators in a variety of systems like thermostats, nano-tubes, micro-robots and micro-capsules \cite{BasAbLaGr,JaSmIn,KuLPL,SchmidtEb,SmInLu}. From the mathematical viewpoint, there is an increasing interest in the modeling and the numerical treatment of such plates.
	A rigorous analysis of large bending deformations of plates was conducted by Friesecke, James and M\"uller \cite{FJM}, who derived the
	geometrically non-linear Kirchhoff model from three dimensional hyperelasticity \rhn{via $\Gamma$-convergence.} Since then, there have been various other interesting models, such as the \rhn{models of prestrained plates} derived in \cite{Lewicka,Lewicka2}.
	Previous work on the numerical treatment of large bending deformations includes the single layer 
	problem by Bartels \cite{Bartels}, the bilayer problem by Bartels, Bonito and Nochetto \cite{BaBoNo} and the modeling and simulation of thermally actuated bilayer plates by Bartels, Bonito, Muliana and Nochetto \cite{BaBoMuNo}. In all three approaches \rhn{\cite{Bartels,BaBoMuNo,BaBoNo}} the model involves minimizing an energy functional that is dominated by the Hessian of the deformation $y: \Omega \to \mathbb{R}^3$ of the mid-plane \rhn{$\Omega\subset\mathbb{R}^2$ of the undeformed plate. Given functions $(g,\Phi)$, 
	the minimization takes place under Dirichlet boundary conditions for $y=g$ and $\nabla y=\Phi$ on part of the boundary $\partial_D\Omega$ of $\partial\Omega$} and the isometry constraint
        \begin{equation}\label{E:isometry}
	\nabla y^T \nabla y= I \quad \text{a.e. in }\Omega,
	\end{equation}
	where $I$ stands for the identity matrix in $\mathbb{R}^2$.  The authors \rhn{of \cite{Bartels,BaBoMuNo,BaBoNo}} employ Kirchhoff elements in order to impose the isometry constraint at the nodes of the triangulation and rely on an $H^2$- gradient flow that allows them to construct solutions of decreasing discrete energy. In \cite{Bartels,BaBoNo} it is also proved that the  discrete energy $\Gamma$-converges to the continuous one. Since for \rhn{fourth} order problems a conforming approach would be very costly, the Kirchhoff elements offer a natural non-conforming space for the model that allows the imposition of \eqref{E:isometry} nodewise. 
	
	\subsection{Our contribution}
	In this paper we focus on the single layer problem, \rhn{as in \cite{Bartels},} in order to investigate the applicability of a more flexible approach that hinges on a non-conforming space of discontinuous functions. \rhn{We use} interior penalty terms for the discrete energy along with a Nitsche's approach to enforce the boundary conditions in the limit. We start with the Dirichlet and forcing data
\begin{equation}\label{E:data}
  g\in [H^1(\Omega)]^3,\quad
  \Phi\in [H^1(\Omega)]^{3\times2}, \quad
  f \in [L^2(\Omega)]^3,
\end{equation}  
and the affine manifold of $[H^2(\Omega)]^3$
\begin{equation}\label{E:affine-manifold}
\mbV(g,\Phi) := \left\{ v \in  [H^{2}(\Omega)]^3: \, v=g, \ \nabla v = \Phi \  \text{on} \ \partial_D \Omega   \right\},
\end{equation}
where $\partial_D \Omega$ is an open \rhn{non-empty subset} of the boundary $\partial \Omega$.
We wish to approximate a minimizer $y: \Omega \to \mathbb{R}^3$ of the continuous energy
	\begin{equation} \label{E:CEnergy}
	E[y] = \frac{1}{2} \int_{\Omega}|D^2y|^2 - \int_{\Omega} f \cdot y
	\end{equation}
	in the nonconvex set of admissible functions
	\begin{equation} \label{E:ContSpace}
	\rhn{\mathbb{A}(g,\Phi) :=} \left\{ y \in \mbV(g,\Phi) : \ (\nabla y)^T \nabla y = I \ \text{a.e. in } \Omega \right\}, 
	\end{equation}
	where $|\cdot|$ denotes the Frobenius norm.
	To avoid the costly use of a conforming finite element subspace of $[H^2(\Omega)]^3$, we \rhn{resort to a space $[\mathbb{V}_h^k]^3$} of discontinuous piecewise polynomials of degree $k\geq 2$ \rhn{over a {\it shape-regular} but possibly {\it graded} mesh $\mathcal{T}_h$ (constructed either using the reference unit triangle or unit square). Since our estimates below are all local, hereafter $h$ stands for a mesh density function locally equivalent to the element size. However, to de-emphasize this aspect of our approach in favor of others (non-convexity, Hessian reconstruction, $\Gamma$-convergence) and to simplify notation, $h$ written as a parameter signifies the meshsize of $\mathcal{T}_h$ (i.e. the largest element size). We denote by $E_h:[\mathbb{V}_h^k]^3\to\mathbb{R}$ the discrete energy that approximates $E$ and accounts for the discontinuities of the functions $v_h\in[\mathbb{V}_h^k]^3$ and of their broken (i.e. piecewise) gradients $\nabla_h v_h$; see \eqref{E:DEn}.}

 It is important to notice that the energy $E$ in \eqref{E:CEnergy}
  is convex but the isometry constraint \eqref{E:isometry} is not.
  Therefore, we must approximate \eqref{E:isometry}, and thus the admissible set
  \rhn{$\mathbb{A}(g,\Phi)$} in a way
  amenable to computation, as well as construct an algorithm able to find \rhn{critical points.}
  We define the discrete admissible set \rhn{$\Ahk(g,\Phi)$} to be the set of functions
  \rhn{$v_h\in[\mathbb{V}_h^k]^3$ whose boundary jumps include $g$ and $\Phi$ (see \eqref{jumps} and \eqref{E:bd-jumps} below) and whose} {\it discrete isometry defect} $D_h[v_h]$ satisfies
	\begin{equation} \label{E:DIC} 
	D_h[v_h]:= \sum_{T \in \mathcal{T}_h } \left| \int_T (\nabla_h v_h)^T \nabla_h v_h  - I   \right|  \leq \eps, 
	\end{equation}
	where $\eps =\eps(\hb) \to 0$ as $\hb\to0$.
We then search for \rhn{$y_h \in \Ahk(g,\Phi)$} that minimizes the discrete energy
$E_h[y_h]$. To this end, we propose a discrete $H^2$-gradient flow \rhn{with fictitious time step $\tau$ that can be made arbitrarily small.} We show that it gives rise to \rhn{iterates $\left\{ y_h^n\right\}_{n=1}^\infty \subset \Ahk(g,\Phi)$} with decreasing discrete energy $E_h[y_h^{n+1}]<E_h[y_h^n]$, whenever $y_h^{n+1} \not= y_h^{n}$, and guarantees the discrete isometry defect \eqref{E:DIC} for all $n \geq 1$ provided the initial guess
\rhn{$y_h^0\in\Ahk(g,\Phi)$} is an approximate isometry \rhn{such that $D_h[y_h^0]\le\tau$. This is achieved by selecting $\eps$ proportional to $\tau$, depending on $E_h[y_h^0], g,\Phi,f$.}

We also prove $\Gamma$-convergence of the discrete energy $E_h$
to the continuous energy $E$ and that global minimizers \rhn{$y_h \in \mathbb{A}_{h,\eps}^k(g,\Phi)$ of $E_h$ converge in $\lom$ to global minimizers $y \in \mathbb{A}(g,\Phi)$ of $E$  as $\hb\to0$.
A key ingredient for $\Gamma$-convergence is reconstruction of a
suitable discrete Hessian $H_h[y_h]$ of $y_h$, which uses the broken Hessian $D^2_h y_h$ and the jumps $[y_h]$ and $[\nabla_h y_h]$ of $y_h$ and $\nabla_h y_h$ across interelement boundaries. Such $H_h[y_h]$ is an $L^2$-function in $\Omega$ that converges weakly to $D^2 y$ under suitable mesh conditions. A similar approach is employed in \cite{DiPietroErn} for second order problems and later in \cite{Pryer} for the $p-$biharmonic equation, where $p=2$. We refer to Section~\ref{S:DiscreteHessian} for a discussion of properties of $H_h[y_h]$ and critical differences with \cite{DiPietroErn,Pryer}.}

	Our approach is motivated by the flexibility of dG compared to Kirchhoff elements \rhn{\cite{Bartels,BartelsBook}.} First of all, Kirchhoff elements require polynomial degree $k=3$ and suffer from a complicated implementation that involves a discrete gradient that maps the gradient of the discrete deformation to another space. Since this is not implemented in standard finite element libraries, the above difficulty hinders the impact of the method in the engineering community. In contrast, the proposed dG approach works for $k\geq 2$,  does not require such map (or more precisely, the map is the trivial elementwise differentiation) and its implementation is standard. Moreover, quadrilateral \rhn{Kirchhoff elements are not amenable to adaptively refined partitions, at least in theory; the current dG theory, instead, allows for graded meshes $\mathcal{T}_h$.} It is worth mentioning as well that imposing Dirichlet boundary conditions weakly, instead of directly in the admissible set as in \rhn{\cite{Bartels,BartelsBook,BaBoNo,BaBoMuNo},} allows for more flexibility. 
	In particular, this alleviates the constraints on the construction of the recovery sequence necessary for  the $\Gamma$-convergence of $E_h$ towards $E$; see Section~\ref{ss:limsup}. Upon dropping the penalty term of jumps of $\nabla y_h$ across a prescribed curve, dG naturally accommodates configurations with \rhn{kinks,} as is the case of origami.
	Lastly, imposing the isometry constraint numerically with Kirchhoff elements at each vertex seems to be very rigid at the expense of approximation accuracy. In contrast, our experiments with dG indicate that this approach is more accurate and adjusts better to the geometry of the problem; see Section~\ref{ss:verticalLoad}.

	\subsection{Outline}
We first construct our discrete energy functional $E_h$ in \Cref{S:DGScheme},
\rhn{and prove key properties for functions $v_h\in[\mathbb{V}_h^k]^3$,} including the coercivity of $E_h$ with respect to an appropriate mesh-dependent energy norm.
          We introduce in \Cref{S:GFlow} the discrete $H^2$-gradient flow and show it is able to find discrete minimizers of $E_h$ that satisfy the desired discrete isometry defect \eqref{E:DIC}; this sets the stage for \rhn{relations} between $\eps, \hb$ and $\tau$ in
          \eqref{E:DIC}.
          In \Cref{S:DiscreteHessian} we define the discrete Hessian $H_h[y_h]$ and derive a bound in $L^2$ and its weak convergence in $L^2$.
        This leads to the proof of $\Gamma$-convergence of the discrete energy $E_h$ to the exact energy $E$, in \Cref{S:ConvergencePl}, \rhn{as well as} the convergence of global minimizers of $E_h$ to global minimizers of $E$.  In \Cref{S:NumExPl} we present experiments that corroborate numerically the excellent properties of our method and compare it with the Kirchhoff element approach. In \Cref{S:Implementation} we provide some details for our implementation and justify the choice of a discontinuous versus a continuous space of functions $\mathbb{V}_h^k$. \rhn{We draw conclusions in Section \ref{S:conclusions} and present some technical estimates for isoparametric maps in the appendix of Section \ref{S:A-quads}.} \looseness=-1

\section{Discrete Energy and Properties of dG} \label{S:DGScheme}

We start this section by providing intuition on the derivation of the discrete energy
$E_h$, but without presenting all the details.
We then introduce the discrete dG space $\Vhk$ \rhn{for scalar functions} along with an interpolation operator
$\Pi_h: \prod_{T\in \mathcal T_h} H^1(T) \to \Vhk\cap H^1(\Omega)$, and discuss its
properties. We finally prove coercivity of $E_h$.

\subsection{Continuous energy} \label{S:Connection}
Starting from a three-dimensional hyperelasticity model, dimension reduction as the thickness of the plate decreases to zero leads to the two-dimensional energy functional \rhn{\cite{Bartels,BartelsBook,FJM}}
\begin{equation}\label{E:second-form}
E[y] = \frac{1}{2} \int_{\Omega} |H|^2 - \int_{\Omega} f \cdot y,
\end{equation}
up to a multiplicative constant for the first term.
\rhn{Hereafter, $y\in\mathbb{A}(g,\Phi)$ is an isometric deformation from $\Omega\subset\mathbb{R}^2$ into $\mathbb{R}^3$, $\nu:= \partial_1 y \times \partial_2 y$ is the unit normal} and $H=(h_{ij})_{i,j=1}^2:= (\partial_{ij} y \cdot \nu)_{i,j=1}^2 $ is
the second fundamental form of the deformed plate $y(\Omega)$.
The connection between \eqref{E:second-form} and \eqref{E:CEnergy} follows from \rhn{\eqref{E:isometry}, namely}
$$
\partial_i y \cdot \partial_j y = \delta_{ij}, \quad i,j=1,2,
$$
\rhn{where $\delta_{ij}$ is the Kronecker delta.}
Differentiating with respect to \rhn{the cartesian coordinates} $x_1$ and $x_2$ and using simple algebraic manipulations, these relations imply
$$
\partial_k y \cdot \partial_{ij} y = 0
\quad\Longrightarrow\quad
\partial_{ij} y \parallel \nu = \partial_1 y \times \partial_2 y, \quad i,j,k=1,2.
$$
This, combined with the definition $h_{ij} = \partial_{ij} y \cdot \nu$,
leads in turn to 
$$
|h_{ij}|^2= |\partial_{ij} y|^2,  \quad \rhn{i,j=1,2.}
$$
Similarly, using that 
$$
 \partial_1 \left( \partial_{12} y \cdot \partial_2 y \right) = 0 \quad \textrm{and} \quad \partial_2 \left( \partial_{11} y\cdot \partial_2 y \right) = 0,
$$ 
we obtain
$$
|\partial_{12} y|^2 = \partial_{11} y \cdot \partial_{22} y.
$$
Therefore, the isometry property \eqref{E:isometry} yields the pointwise relations 
\begin{equation}\label{E:bilaplacian}
|H|^2 = |D^2y|^2 = |\Delta y|^2, 
\end{equation}
so that the \rhn{nonlinear} expression \eqref{E:second-form} of $E[y]$ coincides with \eqref{E:CEnergy}, namely
\begin{equation}\label{e:bilaplacian_energy}
  E[y] = \frac{1}{2} \int_\Omega |D^2 y|^2 - \int_{\Omega} f \cdot y.
\end{equation}
The Euler-Lagrange equation for a minimizer $y\in \mbV(g,\Phi)$ of \eqref{e:bilaplacian_energy} reads
\begin{equation} \label{E:Bilaplacian}
\int_\Omega D^2 y : D^2 v  = \int_{\Omega} f \cdot v  \quad \forall v \in \mbV(0,\bz),
\end{equation}
where $\mbV(g,\Phi)$ is defined in \eqref{E:affine-manifold}. \rhn{If $y=(y_n)_{n=1}^3$, then the strong form of \eqref{E:Bilaplacian} is
\begin{equation}\label{E:strong}
 \div\div D^2y_n = \Delta^2 y_n = \sum_{i,j=1}^2 \partial_{jiij} y_n = f_n \quad\text{in } \Omega, \quad n=1,2,3,
\end{equation}
whereas the natural boundary conditions imposed on $\partial\Omega\setminus\partial_D\Omega$ are
\begin{equation}\label{E:Neumann}
\begin{split}
  &\partial_\mu \nabla y_n = D^2y_n \, \mu = \sum_{i=1}^2 \partial_{ij} y_n \, \mu_i = 0, \qquad n=1,2,3, \\
  &\partial_\mu\Delta y_n = (\div D^2y_n) \mu = \sum_{i,j=1}^2 \partial_{iij} y_n \, \mu_j = 0, \qquad n=1,2,3.
\end{split}
\end{equation}
Hereafter, $\mu$ denotes the outward unit normal to $\Omega$.}

\subsection{Discrete energy}

\rhn{
We denote by $\mathbb{P}_k$ (resp. $\mathbb Q_k$) the space of polynomial functions of degree at most $k\geq 0$ (resp. at most $k$ on each variable). Moreover, $\widehat T$ stands for the reference element, which is either the unit triangle for $\mathbb{P}_k$ or the unit square for $\mathbb Q_k$. A generic element $T:=F_T(\widehat T)$ is given by a map $F_T \in \lbrack \mathbb P^k \rbrack^2$ (resp. $\lbrack\mathbb Q^k \rbrack^2$) provided $\widehat T$ is the unit triangle (resp. square). Notice that when $k=1$, $F_T$ is affine for triangles $T$ and bi-linear for quadrilaterals $T$. 
}

\rhn{
We consider  a sequence of meshes $\{ \mathcal T_h \}_{h>0}$ of $\Omega$  made of closed shape regular but possibly graded elements $T$ \cite{CiarletRaviart}. We assume that $\Omega$ can be exactly represented by such subdivisions, i.e. we do not account in the analysis below for variational crimes induced by the approximation of the boundary $\partial\Omega$. Hereafter, we denote by $h$ a mesh density function which is equivalent locally to the size $h_T$ of $T$ and $h_e$ of an edge $e$; $\hb$ written as a parameter also signifies the largest value of $h_T$. 
In order to handle hanging nodes (necessary for graded meshes based on quadrilaterals), we assume that all the elements within each domain of influence have comparable diameters (independently of $h$). We refer to Sections 2.2.4 and 6 of Bonito-Nochetto \cite{BoNo} for precise definitions and properties. At this point, we only point out that sequences of meshes made of quadrilaterals with at most one hanging node per side satisfy this assumption. We append a sub-index $h$ to differential operators to indicate that they are applied element-wise. For instance, $\nabla_h v$ is defined by $\nabla_h v |_T := \nabla v|_T$, $T\in \mathcal T_h$.
}

From now on $c$ and $C$ are generic constants independent of $h$ but \rhn{perhaps} depending on the shape-regularity constant of the mesh sequence $\{ \mathcal T_h \}_{h >0}$.
Also, we use the notation $A \lesssim B$ to indicate $A \leq c B$, where $c$ is a constant independent of $h$, $A$ and $B$.

\rhn{Let $\Vhk$ be the space of discontinuous functions over the mesh $\mathcal T_h$}
\begin{equation}\label{E:discrete-space}
  \Vhk := \left\{ v_h \in \lom \  : \quad \rhn{v_h \circ F_T \in \mathbb{P}_k \ (\text{resp. } \mathbb{Q}_k) \quad \forall T \in \mathcal T_h} \right\} 
\end{equation}
\rhn{of degree $k\ge 2$. We point out that the discrete functions $v_h\in\Vhk$ are not polynomials on the physical elements $T$ unless $F_T \in \lbrack\mathbb P^1\rbrack^2$. This is consistent with the implementation in the deal.ii library\cite{dealii85} used to obtain the numerical illustrations of Section~\ref{S:NumExPl}. Similarly, the Hessians of $v_h|_T$ and $v_h\circ F_T$ are not proportional because $D^2 F_T \ne \boldsymbol{0}$, except when $F_T \in \lbrack\mathbb P^1\rbrack^2$. We account for this delicate issue in Section \ref{S:A-quads}.}

We denote by $\Ei$ the collection of edges of $\mathcal T_h$ contained in $\Omega$ and by $\Eb$ those contained \rhn{in} $\partial_D\Omega$; hence $\calE_h=\Ei\cup\Eb$ is the set of active interelement boundaries. We further denote the interior skeleton and the boundary counterpart by
\begin{equation}\label{E:skeleton}
  \Gamma_h^0:=\cup\{e: e\in\Ei\},
  \quad
  \Gamma_h^b:=\cup\{e: e\in\Eb\},
\end{equation}
and by $\Gamma_h:=\Gamma_h^0\cup\Gamma_h^b$ the full skeleton.

\rhn{As is customary for dG methods, we need to introduce jumps and averages on edges.}
For $e \in \Ei$ we fix $\mu:=\mu_e$ to be one of the two unit normals to $e$ in $\Omega$; this choice is irrelevant for the discussion below. 
Given $v_h\in \Vhk$,
we denote its broken gradient by $\nabla_h v_h$ and \rhn{the {\it jumps} of $v_h$ and $\nabla_h v_h$ across interior edges by
\begin{equation}\label{jumps}
  [v_h]_e:= v_h^- -v_h^+,
  \quad
  [\nabla_h v_h]_e:= \nabla_h v_h^- -\nabla_h v_h^+, \quad\forall \, e\in \Ei,
\end{equation}
where $v_h^{\pm}(x) = \lim_{s \to 0^+} v_h(x\pm s~\mu_e)$ and $x\in e$. In order to deal with Dirichlet boundary data $(g,\Phi)$ we resort to a Nitsche's approach; hence we do not impose essential restrictions on the discrete space $[\Vhk]^3$. However, to simplify the notation later it turns out to be convenient to introduce the discrete sets $\Vhk(g,\Phi)$ and $\Vhk(0,\bz)$ which mimic the continuous counterparts $\mbV(g,\Phi)$ and $\mbV(0,\bz)$ but coincide with $[\Vhk]^3$. In fact, we say that $v_h\in[\Vhk]^3$ belongs to $\Vhk(g,\Phi)$ provided the boundary jumps of $v_h$ are defined to be
\begin{equation}\label{E:bd-jumps}
  [v_h]_e := v_h -  g,
  \quad
  [\nabla_h v_h]_e := \nabla_h v_h - \Phi,
  \quad \forall \, e\in\Eb.
\end{equation}
We stress that $\|[v_h]\|_{L^2(\Gamma_h^b)}$ and $\|[\nabla_h v_h]\|_{L^2(\Gamma_h^b)} \to 0$ imply $v_h\to g$ and $\nabla_h v_h \to\Phi$ in $L^2(\partial_D\Omega)$ as $\hb\to0$; hence the connection between $\Vhk(g,\Phi)$ and $\mbV(g,\Phi)$. Therefore, the sets $[\Vhk]^3$ and $\Vhk(g,\Phi)$ coincide but the latter is not a space because it carries the additional information of boundary jumps, namely
\begin{equation}\label{discrete-set}
  \Vhk (g,\Phi) := \Big\{ v_h\in [\Vhk]^3: \
  [v_h]_e, \, [\nabla_h v_h]_e \text{ given by \eqref{E:bd-jumps} for all } e\in\Eb \Big\}.
\end{equation}
In the same vein, we will deal with discrete {\it test} functions $v_h\in\Vhk (0,\bz)$ for which boundary jumps are given by
$[v_h]_e := v_h$ and $[\nabla_h v_h]_e := \nabla_h v_h$ for all $e \in \Eb$, which is consistent with \eqref{E:bd-jumps} for $g=0$, $\Phi=\bz$.
We observe again that the sets $\Vhk (0,\bz)$ and $[\Vhk]^3$ are the same.
We will not write the subscript $e$ whenever no confusion arises.
On the other hand, the definiton of {\it average} of $v_h\in [\Vhk]^3$ across an edge $e\in\calE_h$ is independent of Dirichlet conditions and is thus given by}
\begin{equation}\label{averages}
\! \{ v_h \}:=
\begin{cases}
  \frac 1 2 (v_h^+ + v_h^-) & e\in \Ei \\
  v_h^- & e\in \Eb
\end{cases},
\quad
\! \{ \nabla_h v_h \}:=
\begin{cases}
  \frac 1 2 (\nabla_h v_h^+ + \nabla_h v_h^-) & e\in \Ei \\
  \nabla_h v_h^- & e\in \Eb
\end{cases}.
\end{equation}  
\rhn{Definitions \eqref{jumps}, \eqref{E:bd-jumps} and \eqref{averages} extend to the
{\it broken energy space}
\begin{equation}\label{E:broken-energy}
  \rhn{\bE(\mathcal T_h):=\prod_{T\in \mathcal T_h} H^1(T)}.
\end{equation}

Before introducing the discrete energy $E_h$, we derive the corresponding bilinear form $a_h(\cdot,\cdot)$ in the usual manner. This is the first instance where the definitions of $\Vhk(g,\Phi)$ and $\Vhk(0,\bz)$ become critical. We integrate by parts twice the strong equation \eqref{E:strong} over elements $T\in\mathcal{T}_h$ against a test function $v_h\in\Vhk(0,\bz)$. We assume $y$ to be smooth, whence the jumps $[\partial_\mu \nabla y]_e$ and $[\partial_\mu \Delta y]_e$ vanish on edges $e\in\calE_h$, to arrive at
\begin{align*}
  \big( f, v_h \big)_{L^2(\Omega)} =
  \big(  D^2 y, D_h^2 v_h \big)_{L^2(\Omega)}
  - \big( \{\partial_\mu \nabla y\}, [\nabla v_h] \big)_{L^2(\Gamma_h)}
  + \big( \{\partial_\mu \Delta y\}, [v_h] \big)_{L^2(\Gamma_h)}.
\end{align*}
We next use that $[y]_e=0$ and $[\nabla y]_e=0$ for all edges $e\in\calE_h$: for interior edges $e\in\Ei$ this is because $y$ is smooth, whereas for boundary edges $e\in\Eb$ it is a consequence of $y=g$ and $\nabla y = \Phi$ on $\partial_D\Omega$ and definition \eqref{E:bd-jumps}. We can thus symmetrize the previous equality and add vanishing penalty terms
\begin{equation}\label{E:def-ah}
\begin{aligned}
  \big( f, v_h \big)_{L^2(\Omega)} &=
  \big(  D^2 y, D_h^2 v_h \big)_{L^2(\Omega)}
  \\
  & - \big( \{\partial_\mu \nabla y\}, [\nabla v_h] \big)_{L^2(\Gamma_h)}
  - \big( \{\partial_\mu \nabla_h v_h\}, [\nabla y] \big)_{L^2(\Gamma_h)}
  \\
  & + \big( \{\partial_\mu \Delta y\}, [v_h] \big)_{L^2(\Gamma_h)}
  + \big( \{\partial_\mu \Delta_h v_h\}, [y] \big)_{L^2(\Gamma_h)}
  \\
  & + \gamma_1 \big( h^{-1}\left[ \nabla y \right], \left[ \nabla_h v_h \right] \big)_{L^2(\Gamma_h)} + \gamma_0 \big(h^{-3/2} \left[ y \right], \left[ v_h \right] \big)_{L^2(\Gamma_h)}
  \\
  & =: a_h(y,v_h),
\end{aligned}
\end{equation}
with penalty parameters $\gamma_0,\gamma_1>0$ to be determined. 
Hereafter, we denote by $\left(\cdot,\cdot\right)_{L^2(\mathcal M)}$ the $L^2(\mathcal M)$ scalar product and vector versions of it.
Since \eqref{E:def-ah} reduces to \eqref{E:Bilaplacian} for $v_h\in\mbV(0,\bz)$, we resort to \eqref{E:def-ah} to define the discrete equation for $y_h$
\begin{equation}\label{E:bilinear}
\begin{aligned}
  y_h\in\Vhk(g,\Phi): \quad
  a_h(y_h,v_h) = (f,v_h)_{ \lom }  \quad\forall \, v_h\in\Vhk(0,\bz).
\end{aligned}
\end{equation}
We point out that the Dirichlet conditions in \eqref{E:affine-manifold} are enforced in the Nitsche's sense. Since $a_h$ is symmetric by construction, we define the discrete energy to be
\begin{equation}\label{E:DEn}
  E_h[y_h] := \frac12 \, a_h(y_h,y_h) - (f,y_h)_{ \lom }, 
\end{equation}
The first variation $\delta E_h[y_h;v_h]$ of  $E_h[y_h]$ in the direction $v_h\in\Vhk(0,\bz)$ yields \eqref{E:bilinear}.}

In order for \eqref{E:DEn} to be meaningful with respect to the original 
minimization problem in \eqref{E:CEnergy}-\eqref{E:ContSpace}, we define
the {\it discrete admissible set} $\Ahk(g,\Phi)$, a discrete analogue of
$\mathbb{A}(g,\Phi)$
in \eqref{E:ContSpace} that involves the discrete isometry defect $D_h$, to be
\begin{equation} \label{E:Ahk}
\! \Ahk(g,\Phi):= \left\{
\begin{aligned}
y_h \in \Vhk(g,\Phi): \ D_h[y_h] = 
  \sum_{T \in \mathcal{T}_h } \left| \int_T  (\nabla_h y_h)^T \nabla_h y_h  - I  \right|  \leq \eps
\end{aligned}
\right\} 
\end{equation}
with parameter \rhn{$0<\eps= \eps(\hb) \to 0$ as $\hb \to 0$. We will see in Section\nobreakspace \ref {S:GFlow} that the discrete gradient flow used to construct discrete solutions yields $\eps = C\hb$, where $C$ depends on $E_h[y_h^0],g,\Phi,f$, $y_h^0$ being the initial iterate,} and other geometric constants.
  
\subsection{Interpolation onto Continuous Piecewise Polynomials}\label{S:interpolation}
%
For several reasons we need to interpolate from the broken energy space \rhn{$\bE(\mathcal T_h)$ defined in \eqref{E:broken-energy}} onto the space $\mathring\mbV_h^k:=\mbV_h^k\cap H^1(\Omega)$ of continuous \rhn{functions which are polynomials of degree $\le k$ over the reference element $\widehat T$.}
We refer to \cite{BMN:02,BoNo,Brenner1,Brenner,Buffa} for such interpolation estimates.
We now construct a Cl\'ement type interpolation operator $\Pi_h:\bE(\mathcal T_h)\to\mathring\mbV_h^k$, thereby extending \cite{BMN:02,BoNo} to $\bE(\mathcal T_h)$, because
of its simplicity and fitness with our application of it.

We construct $\Pi_h$ in two steps. We first compute the local $L^2$-projection
$P_h:\bE(\mathcal T_h)\to\mbV_h^k$, which for every $v\in\bE(\mathcal T_h)$ and
$T\in\mathcal T_h$ reduces to the equation
\begin{equation}\label{e:proj}
P_h v \in \mbV_h^k(T): \quad \int_T (P_h v - v) w = 0
\quad\forall w\in \mbV_h^k(T),
\end{equation}
where $\mbV_h^k(D)$ (resp. $\mathring\mbV_h^k(D)$) stands for the restriction of functions in $\Vhk$ (resp. $\mathring\mbV_h^k$) to $D \subset \Omega$.
We next define the Cl\'ement interpolation operator $I_h:\mbV_h^k\to \mathring\mbV_h^k$
of \cite{BMN:02,BoNo} as follows. Given the canonical basis functions
$\{\phi_i\}_{i=1}^N$ of $\mathring\mbV_h^k$ with supports $\{\omega_i\}_{i=1}^N$
associated with nodes $\{x_i\}_{i=1}^N$, we compute the $L^2$-projection
of $v\in\mbV_h^k$ on stars $\omega_i$
\begin{equation}\label{E:def-vi}
v_i\in \mathring\mbV_h^k(\omega_i): \quad
\int_{\omega_i} (v-v_i) w = 0 \quad\forall w\in \mathring\mbV_h^k(\omega_i),
\end{equation}
and define $I_h v := \sum_{i=1}^N v_i(x_i) \phi_i \in \mathring\mbV_h^k$ and
\rhn{
\begin{equation}\label{E:inter-oper}
  \Pi_h:= I_h \circ P_h: \bE(\mathcal T_h)\to\mathring\mbV_h^k.
\end{equation}
If $\mathcal{T}_h$ has hanging nodes, the stars $\omega_i$ are related to the notion of domain of influence; we refer to Section 6 of \cite{BoNo} for details. We make the same assumptions as \cite{BoNo} on $\mathcal{T}_h$, which in turn imply that all elements of $\mathcal{T}_h$ within $\omega_i$ possess comparable size.}
\begin{Lemma}[interpolation]\label{L:interpolation}
The interpolation operator
\rhn{$\Pi_h$ defined in \eqref{E:inter-oper}} is invariant in the
space $\mathring\mbV_h^k$ and satisfies the following estimate
for all $v\in\bE(\mathcal{T}_h)$
\begin{gather}
\label{stability}
\|\nabla\Pi_h v\|_{L^2(\Omega)} + \|h^{-1}(v - \Pi_h v)\|_{L^2(\Omega)}
\lesssim \|\nabla_h v\|_{L^2(\Omega)} + \|h^{-1/2} [v]\|_{L^2(\Gamma_h^0)},
\end{gather}
\rhn{where we recall that $\Gamma_h^0$ stands for the interior skeleton of $\mathcal{T}_h$ defined in \eqref{E:skeleton}.}
\end{Lemma}
\begin{proof}
The operator $\Pi_h$ is invariant in $\mathring\mbV_h^k$ because so are $P_h$ and
$I_h$. We next examine properties of $P_h$, \rhn{$I_h$ and $\Pi_h$} separately.
\rhn{Recall that all $T\in\mathcal{T}_h$ are closed.}

\smallskip\noindent
{\it Step 1: Operator $P_h$.} Since $P_h$ is an elementwise $L^2$-projection,
    we easily deduce
$$
\|P_h v - v\|_{L^2(T)} + h_T^{1/2} \|P_h v -v\|_{L^2(\partial T)}
+ h_T \|\nabla P_h v\|_{L^2(T)} \lesssim h_T \|\nabla v\|_{L^2(T)}.
$$
Given \rhn{an interior edge $e\in\Ei$, let $T^\pm\in\mathcal{T}_h$ be the two adjacent
elements that
satisfy $e=T^+ \cap T^-$ and $v^\pm$ be} the restrictions of $v$ to $T^\pm$. If
$\omega(e):= T^+ \cup T^-$, a simple calculation now shows
\begin{align*}
  \|[P_hv]\|_{L^2(e)} & \le \|P_h v^+-v^+\|_{L^2(e)} + \|P_h v^--v^-\|_{L^2(e)}
  + \|v^+ - v^-\|_{L^2(e)}
  \\
  & \lesssim h_e^{1/2} \|\nabla_h v\|_{L^2(\omega(e))} + \|[v]\|_{L^2(e)}.
\end{align*}

\smallskip\noindent
{\it Step 2: Operator $I_h$.}
\rhn{Let $\omega(T)$ be the union of stars $\omega_i$ alluded to in \eqref{E:def-vi} that intersect $T$; note that all elements within $\omega(T)$ have comparable size \cite{BoNo}. Moreover, let $\gamma^0(\omega(T))$ and $\gamma_i^0=\gamma^0(\omega_i)$ be the skeletons of $\omega(T)$ and $\omega_i$ (interelement boundaries internal to these sets).
}
We recall (see e.g. \cite{BoNo}) that to prove the estimate
\[
\|\nabla I_h v\|_{L^2(T)} + \|h^{-1}(I_hv-v)\|_{L^2(T)}
\lesssim  \|\nabla_h v\|_{L^2(\omega(T))} + \|h^{-1/2}[v]\|_{L^2(\gamma^0(\omega(T)))}
\]
for all $v\in\mbV_h^k$, it suffices to derive the bounds
\[
\|\nabla v_i\|_{L^2(\omega_i)} + \|h^{-1}(v-v_i)\|_{L^2(\omega_i)}
\lesssim \|\nabla_h v\|_{L^2(\omega_i)} + \|h^{-1/2}[v]\|_{L^2(\gamma_i^0)},
\]
where $v_i$ is defined in \eqref{E:def-vi}.
Since the dimension of the space $\mbV_h^k(\omega_i)$ is finite and depends only
on shape regularity, all norms in $\mbV_h^k(\omega_i)$ are
equivalent and independent of meshsize \rhn{upon rescaling $\omega_i$ to unit size.}
We further observe that if the 
right hand side of the last inequality vanishes, then $v$ is constant in $\omega_i$.
The definition of $v_i$ thus gives $v_i=v$ and the left hand side vanishes,
thereby showing that the desired estimate is valid \rhn{in the rescaled $\omega_i$.}
The powers of meshsize result from a standard scaling argument \rhn{from unit size back
to actual size.}

\smallskip\noindent
{\it Step 3: Operator $\Pi_h$.} Combining the estimates for $P_h$ and $I_h$ yields
\[
\|\nabla\Pi_h v\|_{L^2(\Omega)} \lesssim \|\nabla P_h v\|_{L^2(\Omega)}
+ \|h^{-1/2}[P_h v]\|_{L^2(\Gamma_h^0)}
 \lesssim \|\nabla_h v\|_{L^2(\Omega)}
+ \|h^{-1/2}[v]\|_{L^2(\Gamma_h^0)}.
\]
A similar estimate is valid for $\|h^{-1}(v-\Pi_h v)\|_{L^2(\Omega)}$. This concludes
the proof.
\end{proof}

We have written \eqref{stability} in a convenient form
for our later application. Note that it only requires that $\mbV_h^k$ contains piecewise
constants and make no reference to the actual polynomial degree $k$. However, since
$\Pi_h$ is invariant on $\mathring\mbV_h^k$, we may apply \eqref{stability} to $v-p$ where
$p\in\mathring\mbV_h^k$ is the best $H^1$-approximation of $v\in H^{k+1}(\Omega)$
\rhn{to get}
\[
\|v-\Pi_h v\|_{L^2(\Omega)} = \|(v-p)-\Pi_h(v-p)\|_{L^2(\Omega)}
\lesssim h \|\nabla(v-p)\|_{L^2(\Omega)} \lesssim h^{k+1} \rhn{\|v\|_{H^{k+1}(\Omega)};}
\]
\rhn{see Lemma \ref{L:error-curved-quad} (error estimates for curved quadrilaterals).} \rhn{Another consequence of the proof of Lemma \ref{L:interpolation} is the {\it local}
interpolation estimate
\begin{equation*}
\|\nabla\Pi_h v\|_{L^2(T)} + \|h^{-1}(v - \Pi_h v)\|_{L^2(T)}
\lesssim \|\nabla_h v\|_{L^2(\omega(T))} + \|h^{-1/2} [v]\|_{L^2(\gamma^0(\omega(T)))}
\end{equation*}
for all $T\in\mathcal{T}_h$; recall that $\omega(T)$ and $\gamma^0(\omega(T))$ are
defined in Step 2 of Lemma \ref{L:interpolation}.
}

\begin{Corollary}[boundary error estimate]\label{C:boundary}
The following estimate is valid
\begin{equation}\label{E:boundary}
\|h^{-1/2}(v - \Pi_h v)\|_{L^2(\partial\Omega)}
\lesssim \|\nabla_h v\|_{L^2(\Omega)} + \|h^{-1/2} [v]\|_{L^2(\Gamma_h^0)}
\quad\forall \, v\in\bE(\mathcal{T}_h).
\end{equation}    
\end{Corollary}
\begin{proof}
Let $e$ be a generic boundary edge, not necessarily in $\Eb$, and
$T\in\mathcal{T}_h$ be an
adjacent element (i.e. $e\subset\partial T$). The scaled trace inequality reads
\[
\|h^{-1/2}(v-\Pi_h v)\|_{L^2(e)} \lesssim
\|h^{-1}(v-\Pi_h v)\|_{L^2(T)} + \|\nabla(v-\Pi_h v)\|_{L^2(T)}
\]
Adding over $e$, the desired estimate follows from \eqref{stability}.
\end{proof}
  
The following Friedrichs inequality is another straightforward
application of \eqref{stability};
similar estimates are proved in \cite{Brenner1,Brenner}.
We observe that the jumps $[v]$ of $v\in\bE(\mathcal T_h)$ in \eqref{stability}
are computed on the interior skeleton $\Gamma_h^0$, but the
desired estimate requires control of the trace of $v$ on $\partial_D\Omega$. 
\rhn{
\begin{Corollary}[discrete Friedrichs inequality]\label{C:Friedrichs}
There exists a constant $C_F>0$, depending on $\Omega$ and $\partial_D\Omega$,
such that for all
$v\in\bE(\mathcal{T}_h)$ there holds
\begin{equation}\label{E:Friedrichs}
\|v\|_{L^2(\Omega)} \le C_F \Big( \|\nabla_h v\|_{L^2(\Omega)}
+ \|h^{-1/2}[v]\|_{L^2(\Gamma_h^0)} + \|v\|_{L^2(\partial_D\Omega)} \Big).
\end{equation}
\end{Corollary}
}
\begin{proof}
\rhn{In view of \eqref{stability}, it suffices to prove \eqref{E:Friedrichs} for $\Pi_h v$. The standard form of the Friedrichs inequality is valid for $\Pi_hv\in H^1(\Omega)$, namely}
\[
\|\Pi_h v\|_{L^2(\Omega)} \lesssim \|\nabla \Pi_h v\|_{L^2(\Omega)}
+ \|\Pi_h v\|_{L^2(\partial_D\Omega)},
\]
where the hidden constant depends only on $\Omega$ \rhn{and $\partial_D\Omega$.
Let $e\in\Eb$ be a boundary edge in $\partial_D\Omega$ and $T\in\mathcal{T}_h$ be
an element so that $e\subset\partial T$. The scaled trace inequality
\[
\|w\|_{L^2(e)} \lesssim h_T^{-1/2}\|w\|_{L^2(T)} + h_T^{1/2}\|\nabla w\|_{L^2(T)}
\]
for $w=v-\Pi_h v$, in conjunction with \eqref{stability}, yields
\[
\|\nabla \Pi_h v\|_{L^2(\Omega)} +
\|\Pi_h v\|_{L^2(\partial_D\Omega)} \lesssim \|\nabla_h v\|_{L^2(\Omega)}
+ \|h^{-1/2}[v]\|_{L^2(\Gamma_h^0)} + \|v\|_{L^2(\partial_D\Omega)},
\]
and shows \eqref{E:Friedrichs} as desired.}
\end{proof}

\subsection{Coercivity of the Discrete Energy} \label{S:Coercivity}

We now prove that the discrete energy \eqref{E:DEn} is coercive with respect to a suitable dG \rhn{quantity that substitutes the $H^2$-norm.} As motivation, we start with a similar coercivity estimate for the continuous case.
\begin{Lemma}[coercivity of $E$] \label{L:ECoercivity}
	Let data $(g,\Phi,f)$ satisfy \eqref{E:data}. 
	For any $y \in \mbV(g,\Phi)$ there holds
	$$
	\|y\|_\htwom^2 \lesssim E[y] + \|g\|_\hom^2 + \|\Phi\|_\hom^2 + \|f\|_\lom^2,
	$$
        where $E[y]$ is defined in \eqref{E:CEnergy}.
\end{Lemma}
\begin{proof}
\rhn{Since $y=g$ and $\nabla y = \Phi$ on $\partial_D\Omega$, the Friedrichs inequality implies
\[
\|y-g\|_\lom \lesssim \|\nabla(y-g)\|_\lom,
\quad
\|\nabla y-\Phi\|_\lom \lesssim \|D^2 y- \nabla\Phi)\|_\lom,
\]    
whence
\begin{equation}\label{e:H1norm_estim}
	\|y\|_\hom^2 \leq C \left( \|D^2y\|_\lom^2 + \|g\|_\hom^2 +  \|\Phi\|_\hom^2 \right).\end{equation}
Consequently, we use \eqref{E:CEnergy} to find that for any $\rho >0$}
	$$
	\begin{aligned}
	\|D^2y\|_\lom^2 &= 2 E[y] + 2\int_\Omega f \cdot y \leq 2E[y] + 2\|f\|_\lom \|y\|_\lom 	\\
	&\leq 2E[y] + \frac{1}{\rho} \|f\|_\lom^2  + \rho \|y\|_\lom^2 \\
	& \leq 2E[y] + \frac{1}{\rho} \|f\|_\lom^2  + C \rho (\|D^2y\|_\lom^2 + \|g\|_\hom^2 +  \|\Phi\|_\hom^2).
	\end{aligned}
	$$
	Setting $\rho = \frac{1}{2C}$, we deduce that
	$$
	\|D^2y\|_\lom^2 \lesssim E[y] + \|g\|_\hom^2 +  \|\Phi\|_\hom^2 + \|f\|_\lom^2,
	$$
	which combined with \eqref{e:H1norm_estim} gives the asserted estimate. 
	\end{proof}

We observe that the Friedrichs inequality plays a crucial role in the previous proof
to control $y-g$ and $\nabla y -\Phi$, which vanish on $\partial_D \Omega$.
At the discrete level we face two difficulties: the lack of regularity that leads to interior jumps for $y_h$ and the Nitsche's approach that \rhn{circumvents the explicit imposition of Dirichlet boundary conditions. To deal with the first issue, we introduce a discrete $H^2$-scalar product
\begin{equation}
\begin{aligned}
  \big(w_h,v_h\big)_{H_h^2(\Omega)} := & \big(D_h^2 w_h,D_h^2 v_h \big)_{L^2(\Omega)}
  \\ &
+ \big(h^{-1}[ \nabla_h w_h],[\nabla_h v_h]\big)_{L^2(\Gamma_h)}
+ \big(h^{-3}[ w_h],[ v_h]\big)_{L^2(\Gamma_h)}
\end{aligned}      
\end{equation}
and corresponding discrete $H^2$-norm
$\| v_h \|_{H^2_h(\Omega)} := (v_h,v_h)_{H_h^2(\Omega)}^{1/2}$
for all $v_h,w_h\in\Vhk(0,\bz)$. We point out that the jumps are computed in the full
skeleton $\Gamma_h$ and for boundary edges $e\in\Eb$ we have $[v_h]=v_h$ and $[\nabla_h v_h]=\nabla_h v_h$ according to the definition of $\Vhk(0,\bz)$ in \eqref{discrete-set} for $g=0$ and $\Phi=\bz$.

In order to account for the second issue, the Nitsche's approach, we need to allow $v_h\in\Vhk(g,\Phi)$ and thus utilize the convention that boundary jumps in
\begin{equation} \label{E:energynormN}
\| v_h \|_{H_h^2(\Omega)}^2 :=   \| D_h^2 v_h\|_{L^2(\Omega)}^2
+ \| h^{-1/2} [ \nabla_h v_h] \|_{L^2(\Gamma_h)}^2 
+ \|h^{-3/2} [ v_h]\|_{L^2(\Gamma_h)}^2
\end{equation}
are defined as in \eqref{discrete-set} for $v_h\in\Vhk(g,\Phi)$. This convention will simplify notation and will not lead to confusion because we will always specify the membership of $v_h$ in the sequel. Notice that $\| . \|_{H_h^2(\Omega)}$ is not a norm on $\Vhk(g,\Phi)$.

In addition, we can derive a Friedrichs-type estimate for any $v_h\in\Vhk(g,\Phi)$.
We simply apply \eqref{E:Friedrichs} to $v_h$ and $\nabla_h v_h$, and add and subtract $g$ and $\Phi$ to the boundary terms $\|v_h\|_{L^2(\partial_D\Omega)}$ and $\|\nabla_h v_h\|_{L^2(\partial_D\Omega)}$, to obtain
        \begin{equation}\label{e:H2typePoincare}
        \|v_h\|_{L^2(\Omega)} + \|\nabla_h v_h\|_{L^2(\Omega)} \le C_F
        \Big(\|v_h\|_{H_h^2(\Omega)} + \|g\|_{H^1(\Omega)} + \|\Phi\|_{H^1(\Omega)}  \Big),
        \end{equation}
where $C_F$ (not relabeled) is proportional to the constant in \eqref{E:Friedrichs}.

\begin{Lemma}[coercivity of $E_h$] \label{L:Coercivity}
Let data $(g,\Phi,f)$ satisfy \eqref{E:data} and set
\begin{equation}\label{E:R}
  R(g,\Phi,f) := \|g\|_\hom^2 + \|\Phi\|_\hom^2 + \|f\|_\lom^2.
\end{equation}
If the penalty parameters $\gamma_0$, $\gamma_1$ in the discrete energy $E_h[y_h]$
defined in \eqref{E:DEn} are sufficiently large,
then there exists a constant $C_F^*$ depending on $C_F$ in \eqref{e:H2typePoincare} such that
$$
\|y_h\|_{H_h^2(\Omega)}^2 \lesssim  E_h[y_h] + C_F^* R(g,\Phi,f) \quad \forall  y_h\in\Vhk(g,\Phi).
$$
\end{Lemma}
\begin{proof}
We examine the various terms in $a_h(y_h,y_h)$ defined in \eqref{E:def-ah} with $y_h\in\Vhk(g,\Phi)$. We start with $\big(\{ \nd \nabla_h y_h\}, \left[ \nabla_h y_h \right]\big)_{\LtG}$, recall the definition \eqref{averages} of averages, and use the inverse estimate \eqref{E:inverse} for isoparametric elements to obtain
\[
\|\{ \nd \nabla_h y_h\}\|_{L^2(e)} \le \|\{D^2_h y_h\}\|_{L^2(e)}
\lesssim h_e^{-1/2} \|D^2_h y_h\|_{L^2(\omega(e))}
\]
for all $e\in\calE_h$, where $\omega(e)$ is either the union of two elements containing $e$ for $e\in\Ei$ or a single element for $e\in\Eb$. Consequently, Young's inequality yields for any $\rho>0$
	\begin{equation*}
	\begin{aligned}
	  \Big|\big(\{ \nd \nabla_h y_h\}, \left[ \nabla_h y_h \right]\big)_{\LtG} \Big|
          \le \rho \| D^2_h y_h  \|_{L^2(\Omega)}^2
          + \frac{C_1}{\rho} \| h^{-1/2} \left[ \nabla_h y_h \right] \|_{\LtG}^2.
	\end{aligned}
	\end{equation*}
We next consider $\big(\{ \nd \Delta_h y_h\}, \left[ y_h \right]\big)_{\LtG}$. The inverse inequality \eqref{E:inverse} again implies
$$
	\begin{aligned}
	\|\{ \nd \Delta_h y_h\}\|_{L^2(e)}
        \le \|\{ D_h^3 y_h\}\|_{L^2(e)}
	\lesssim h_e^{-3/2}\|D_h^2 y_h\|_{L^2(\omega(e))}
        \quad\forall \, e\in\calE_h,
	\end{aligned}
$$
	whence Young's inequality gives
	\begin{align*}
	\Big|\big(\{ \nd \Delta_h y_h\}, \left[ y_h \right]\big)_{\LtG} \Big|
	\lesssim \rho \|D^2_h y_h\|_{L^2(\Omega)}^2
        + \frac{C_0}{\rho} \|h^{-3/2}\left[ y_h \right]\|_{\LtG}^2.
	\end{align*}
        In light of definition \eqref{E:energynormN} of $\|\cdot\|_{H^2_h(\Omega)}$, we choose $\rho=1/8$ and $\gamma_1,\gamma_0$ sufficiently large, depending on the geometric constants $C_1,C_0$, to absorb the above terms into $\frac12 \|D_h^2 y_h\|_{L^2(\Omega)}^2$, $\frac{\gamma_1}{2} \|h^{-1/2}[\nabla_h y_h]\|_{L^2(\Gamma_h)}^2$ and $\frac{\gamma_0}{2} \|h^{-3/2}[y_h]\|_{L^2(\Gamma_h)}^2$ of $a_h(y_h,y_h)$. We thus deduce coercivity of $a_h$ and, according to \eqref{E:DEn}, we find
        \[
        \|y_h\|_{H_h^2(\Omega)}^2 \lesssim \frac12 a_h(y_h,y_h)
        = E_h[y_h] + \big(f,y_h\big)_{L^2(\Omega)}.
        \]
        We apply the Friedrichs-type estimate \eqref{e:H2typePoincare} to deal with the forcing term, namely
	\begin{align*}
	  \Big|\int_{\Omega} f \cdot y_h \Big| \le \|f\|_\lom \|y_h\|_\lom  \le C_F
          \|f\|_\lom \Big(\|y_h\|_{H_h^2(\Omega)} +  \|g\|_{\hom} + \|\Phi\|_{\hom} \Big).
	\end{align*}
        We finally see that $\|y_h\|_{H_h^2(\Omega)}$ can be hidden on the left-hand side of the previous coercivity estimate via Young's inequality, and conclude the proof.
\end{proof}

Lemma \ref{L:Coercivity} and its proof yield the following bounds on $E_h[y_h]$
\begin{equation}\label{E:lower-bound}
  - C_F^* R(g,\Phi,f) \le E_h[y_h] \lesssim \|y_h\|_{H_h^2(\Omega)}^2 + R(g,\Phi,f).
\end{equation}
%
  
\section{Discrete Gradient Flow} \label{S:GFlow}

We now design a discrete $H^2$-gradient flow to construct discrete minimizers.
A key aspect of this flow is to guarantee the discrete isometry defect \eqref{E:DIC}, namely
  \[
   D_h[y_h] = \sum_{T \in \mathcal{T}_h } \left| \int_T (\nabla_h v_h)^T \nabla_h v_h  - I   \right| \le \eps.
  \]
We start with the first variation of the isometry constraint \eqref{E:isometry} evaluated at $y\in\mbV(g,\Phi)$ in the direction $v$, the so-called {\it linearized isometry constraint}, which reads
}
\begin{equation}\label{E:linear-iso}
\rhn{L[y;v]} = (\nabla v)^T \nabla y + (\nabla y)^T \nabla v = 0
\quad\forall \, v\in \mbV(0,\bz).
\end{equation}
This serves to describe the tangent manifold to \eqref{E:isometry} at $y\in\mbV(g,\Phi)$
\[
\mathcal{F}[y] := \big\lbrace v \in \mbV(0,\bz): \quad \rhn{L[y;v]=0} \big\rbrace.
\]
\rhn{In view of \eqref{E:linear-iso}, let} the {\it discrete linearized isometry constraint} at $y_h\in\Vhk(g,\Phi)$ be
\begin{equation}\label{E:LinDIC}
\rhn{L_T[y_h;v_h]} := \int_T  (\nabla_h v_h)^T \nabla_h y_h 
+ (\nabla_h y_h)^T \nabla_h v_h =0
\quad\forall \, v_h\in \Vhk(0,\bz),
\end{equation}
which imposes the pointwise equality \eqref{E:linear-iso} on average
over each element $T\in\mathcal{T}_h$. This
in turn leads to the subspace $\mathcal{F}_{h}[y_h]$ of $\Vhk(0,\bz)$ defined as
$$
\mathcal{F}_{h}[y_h]
:= \left\lbrace v_h \in \Vhk(0,\bz) : \quad  \rhn{L_T[y_h;v_h]} =0 \quad \forall \,
T \in \mathcal{T}_h\right\rbrace.
$$

We are now in a position to describe the {\it discrete $H^2$-gradient flow}. Let
$y_h^0\in \Vhk(g,\Phi)$ be a suitable initial guess with energy $E_h[y_h^0]$ as
small as possible and isometry defect $D_h[y_h^0]\le \tau$, where $\tau>0$ is a
fictituous time-step to be determined later. Given an iterate $y_h^n\in \Vhk(g,\Phi)$
for $n\ge0$, we seek to minimize the functional
\[
\mathcal{F}_{h}[y_h^n] \ni w_h \mapsto \frac{1}{2\dt} \|w_h\|_{H_h^2(\Omega)}^2
+ E_h[y_h^n+w_h],
\]
where the discrete $H^2$-norm $\|w_h\|_{H_h^2(\Omega)}$ is defined in \eqref{E:energynormN}.
We thus minimize $E_h[y_h^n+w_h]$ but penalizing the deviation of $w_h$ from zero in
the norm $\|w_h\|_{H_h^2(\Omega)}$. If $\delta y_h^{n+1}\in\mathcal{F}_{h}[y_h^n]$ is a minimizer, then it satisfies the optimality condition
\begin{equation}\label{E:GFlow-energy}
  \tau^{-1} \big(\delta y_h^{n+1},v_h\big)_{H_h^2(\Omega)}
  + \delta E_h[y_h^n+\delta y_h^{n+1};v_h] = 0
  \quad \, \forall v_h\in\mathcal{F}_{h}[y_h^n],
\end{equation}
where $\delta E_h[y_h^{n+1};v_h]$ is the variational derivative
of $E_h$ at \rhn{$y_h^{n+1}:=y_h^n + \delta y_h^{n+1}$} in the direction of $v_h$.
In view of \eqref{E:bilinear} and \eqref{E:DEn},
the Euler-Lagrange equation \eqref{E:GFlow-energy} is equivalent to \looseness=-1
\begin{equation} \label{E:GFlow}
  \begin{aligned}
    \dt^{-1} (\delta y_h^{n+1} &,v_h)_{H_h^2(\Omega)} +  a_h(\delta y_h^{n+1},v_h)
    \\ &=
  \rhn{-} a_h(y_h^n,v_h) + (f, v_h)_{L^2(\Omega)} \rhn{=: F^n(y_h^n;v_h)} \quad\forall v_h \in \mathcal{F}_{h}[y_h^n].
  \end{aligned}
\end{equation}
 
\begin{remark}[solvability of \eqref{E:GFlow}]\label{R:solvability}
Note that problem \eqref{E:GFlow} is linear in $\delta y_h^{n+1}$
and Lemma~\ref{L:Coercivity} (coercivity of $E_h$) with vanishing data $(g,\Phi,f)$
yields coercivity of $a_h$ \looseness=-1
\begin{equation}\label{e:coercivity_a}
  \quad a_h(v_h,v_h) \geq \alpha \rhn{\| v_h \|^2_{H^2_h(\Omega)}}
  \quad\forall \, v_h \in \Vhk(0,\bz)
\end{equation}
for some constant $\alpha$ independent of $h$.
Since $0 \in \mathcal{F}_{h}[y_h^n]$ we infer that 
$ \mathcal{F}_{h}[y_h^n] \not = \emptyset$ and 
the Lax-Milgram theorem implies existence and uniqueness of a solution 
$\delta y_h^{n+1}\in\mathcal{F}_{h}[y_h^n]$ to each step of the discrete
gradient flow \eqref{E:GFlow-energy}.
\end{remark}

\begin{remark}[saddle point formulation]\label{R:saddle}
\rhn{
  Instead of solving \eqref{E:GFlow} within the manifold $\mathcal{F}_{h}[y_h^n] \subset
  \Vhk(0,\bz)$ we could pose the problem in the entire space $\Vhk(0,\bz)$ provided we
  append the constraint $\delta y_h^{n+1}, v_h \in \mathcal{F}_{h}[y_h^n]$ via a Lagrange
  multiplier. To this end, consider the bilinear form
  $\ell_h(y_h^n;\cdot,\cdot):[\Vhk]^3\times[\mathbb{V}_h^0]^{2\times2}$ given by
  \begin{equation}\label{E:ell}
  \ell_h (y_h^n; v_h,\mu_h) := \sum_{T\in\mathcal{T}_h}
  \int_T \mu_n : \Big( (\nabla_h y_h^n)^T \nabla_h v_h + (\nabla_h v_h)^T \nabla_h y_h^n \Big).
  \end{equation}
  The saddle point system equivalent to \eqref{E:GFlow} reads: find
  $(\delta y_h^{n+1},\lambda_h^{n+1}) \in[\Vhk]^3\times[\mathbb{V}_h^0]^{2\times2}$
  such that for all $(v_h,\mu_h)\in[\Vhk]^3\times[\mathbb{V}_h^0]^{2\times2}$ there holds
  \begin{equation}\label{E:saddle}
    \begin{aligned}
     \dt^{-1} (\delta y_h^{n+1},v_h)_{H_h^2(\Omega)} +  a_h(\delta y_h^{n+1},v_h)
     + \ell_h(y_h^n;v_h,\lambda_h^{n+1}) &= F^n(y_h^n;v_h)
     \\
     \ell_h(y_h^n;\delta y_h^{n+1},\mu_h) &= 0.
    \end{aligned}
  \end{equation}
  It is an open question whether a uniform inf-sup condition is
  valid for \eqref{E:saddle}.
  We will get back to \eqref{E:saddle} in Section \ref{S:Implementation} as a
  practical procedure to find $\delta y_h^{n+1}$.
}
\end{remark}

We now show that the discrete $H^2$-gradient flow \eqref{E:GFlow-energy}
reduces the energy $E_h$.

\begin{Lemma} [energy decay] \label{L:energydecrease}
Let $y_h^n\in\Vhk(g,\Phi)$ be the $n$-th iterate of the discrete $H^2$-gradient flow \eqref{E:GFlow-energy}
with data $(g,\Phi,f)$ obeying \eqref{E:data}. If $\delta y_h^{n+1}\ne 0$ is the solution of \eqref{E:GFlow-energy}, then \rhn{the next iterate
$y_h^{n+1} = y_h^n + \delta y_h^{n+1}$} satisfies
$$
	E_h[y_h^{n+1}] < E_h[y_h^n].
$$
Moreover, if $\alpha$ is the coercivity constant in \eqref{e:coercivity_a} and
$\tau$ is the time step, then
\begin{equation}\label{E:en_decay}
  \left( \frac{\alpha}{2} + \frac{1}{ \dt}	\right)
  \sum_{n=0}^{N-1} \|\delta y_h^{n+1}\|_{H_h^2(\Omega)}^2
	+ E_h[y_h^{N}] 
	\leq E_h[y_h^0].
\end{equation}
\end{Lemma}

\begin{proof}
We set $v_h= \delta y_h^{n+1}$ in \eqref{E:GFlow-energy} and use the fact that
$E_h$ is quadratic to obtain
\[
\rhn{\delta E_h[y_h^n + \delta y_h^{n+1}, \delta y_h^{n+1}]
= E_h[y_h^{n+1}] - E_h[y_h^n]} + \frac{1}{2} a_h(\delta y_h^{n+1},\delta y_h^{n+1}).
\]
Invoking the coercivity property \eqref{e:coercivity_a} we thus get
\[
\Big(\frac{\alpha}{2} + \frac{1}{\tau}  \Big)
\rhn{\| \delta y_h^{n+1} \|_{H_h^2(\Omega)}^2 + E_h[y_h^{n+1}] \le E_h[y_h^n],}
\]
whence \rhn{$E_h[y_h^{n+1}] < E_h[y_h^n]$} if $\delta y_h^{n+1}\ne0$.
We finally sum over n to deduce \eqref{E:en_decay}.
\end{proof}

We now show that the discrete $H^2$-gradient flow guarantees the discrete
isometry defect \eqref{E:DIC} for any $\eps>0$ provided the time step
$\dt$ \rhn{in \eqref{E:GFlow-energy}} is suitably chosen. \rhn{This is due
to the presence of the {\it dissipative} leftmost term in \eqref{E:en_decay}.}

\begin{Lemma} [discrete isometry defect] \label{L:GFlowDIC}
Let the initial guess $y_h^0\in\Vhk(g,\Phi)$ for \eqref{E:GFlow} be an
approximate isometry in the sense that
\rhn{
\begin{equation}\label{E:approx-isometry}
  D_h[y_h^0] = \sumT \Big| \int_T  (\nabla_h  y_h^0)^T  \nabla_h y_h^0  - I \Big|
  \le \tau,
\end{equation}
let $R(g,\Phi,f)$ and $C_F^*$ be given in Lemma \ref{L:Coercivity} (coercivity of $E_h$), 
\rhn{let $C_F$ be the Friedrichs constant in \eqref{e:H2typePoincare},} and set}
\begin{equation}\label{E:delta0}
\rhn{\eps_0 := \Big(1 + C_F^2 \big( E_h[y_h^0] + C_F^* R(g,\Phi,f) \big) \Big) \tau,}
\end{equation}
Then every iterate $y_h^n$ of \eqref{E:GFlow-energy} for $n\ge1$
satisfies the discrete isometry defect \eqref{E:DIC}
\begin{equation}\label{e:Dh}
  D_h[y_h^n] = \sumT\left|  \int_T (\nabla_h y_h^{n})^T \nabla_h y_h^{n} - I
  \right| \le \eps \quad \rhn{\forall \, \eps \ge \eps_0.}
	\end{equation}
\end{Lemma}
\begin{proof}
  We start by quantifying the increase of the discrete isometry defect in
  each iteration of the discrete gradient flow. \rhn{Combining $y_h^{n+1}=y_h^n+\delta y_h^{n+1}$ with the property $L_T[y_h^n;\delta y_h^{n+1}]=0$, which is \eqref{E:LinDIC} for $v_h=\delta y_h^{n+1}$, we get the identity}
\begin{equation*}
	\int_T (\nabla_h y_h^{n+1})^T \nabla_h y_h^{n+1}
	= \int_T \left(\nabla_h y_h^{n}\right)^T\nabla_hy_h^{n} + \int_T\left(\nabla_h\delta y_h^{n+1}\right)^T\nabla_h \delta y_h^{n+1}.
\end{equation*}
Summing for $n=0$ to $n=N-1$, and exploiting telescopic cancellation,
we obtain
\[
\int_T (\nabla_h y_h^N)^T \nabla_h y_h^N =
\int_T (\nabla_h y_h^0)^T \nabla_h y_h^0 +
\sum_{n=0}^{N-1} \|\nabla_h \delta y_h^{n+1}\|_{L^2(T)}^2.
\]
Consequently, adding over $T\in\mathcal{T}_h$ and employing \eqref{E:approx-isometry},
we deduce
$$
D_h[y_h^N] = \sumT\left| \int_T  (\nabla_h y_h^{N})^T \nabla_h y_h^{N} - I
\right| \leq \tau + \sum_{n=0}^{N-1} \|\nabla_h \delta y_h^{n+1}\|_{L^2(\Omega)}^2.
$$
We finally combine the Friedrichs-type inequality \eqref{e:H2typePoincare} for \rhn{$\delta y_h^{n+1}\in\Vhk(0,\bz)$ (i.e. corresponding to Dirichlet boundary data $g=0$ and $\Phi=\bz$) with the energy decay relation \eqref{E:en_decay} and the lower bound in \eqref{E:lower-bound} to get}
\[
\sum_{n=0}^{N-1} \|\nabla_h \delta y_h^{n+1}\|_{L^2(\Omega)}^2 \le
\rhn{C_F^2 \sum_{n=0}^{N-1} \| \delta y_h^{n+1} \|_{H_h^2(\Omega)}^2 < C_F^2 \Big(E_h[y_h^0]
+ C_F^* R(g,\Phi,f) \Big)\tau.}
\]
Inserting this in the bound for $D_h[y_h^N]$ yields the asserted estimate
\eqref{e:Dh}.
\end{proof}

\begin{remark} [choice of initial guess $y_h^0$]
  We realize from \rhn{Lemma \ref{L:GFlowDIC}} (discrete isometry defect) that choosing
  $y_h^0$ might be tricky, unless we can make the discrete isometry defect
  $D_h[y_h^0]$ small without increasing much the discrete energy $E_h[y_h^0]$
  \rhn{due to the boundary penalty terms.}
We will revisit this issue in Section \ref{S:NumExPl}.
\end{remark}

\section{Discrete Hessian} \label{S:DiscreteHessian}

In this section we provide a suitable definition of discrete Hessian \rhn{$H_h[y_h]\in [L^2(\Omega)]^{3\times 2\times 2}$ for $y_h \in [\Vhk]^3$.} Central to this concept is the following question: if the discontinuous function $y_h$ converges strongly in $[L^2(\Omega)]^3$ to $y\in [H^2(\Omega)]^3$, under what conditions could $H_h[y_h]$ converge weakly in $[L^2(\Omega)]^{3 \times 2 \times 2}$ to $D^2y$, namely
\begin{equation}\label{E:weak-convergence}
H_h[y_h] \rightharpoonup D^2y \quad \text{in} \quad  [\lom]^{3 \times 2 \times 2} \, ?
\end{equation}
It is apparent that the information contained in the broken \rhn{Hessian $D^2_hy_h\in [L^2(\Omega)]^{3\times2\times2}$} is insufficient for this purpose, and that the jumps $[y_h]\in [L^2(\Gamma_h)]^3$ and $[\nabla_hy_h]\in[L^2(\Gamma_h)]^{3\times 2}$ do not provide directly the missing information because they are singular with respect to the Lebesgue measure. To bridge this gap, in Section \ref{S:Hessian-def} we introduce lifting operators of $[y_h]$ and $[\nabla_hy_h]$ \cite{} and use them to construct $H_h[y_h]$. Another important property of $H_h[y_h]$ critical for the lim-inf argument is
\begin{equation} \label{E:HhIneq}
\frac{1}{2} \|H_h[y_h]\|_\lom^2 - (f_h,y_h)_\lom \leq E_h[y_h],
\end{equation}
\rhn{with constant $1$ on the right-hand side.}
We will prove \eqref{E:weak-convergence} and
\eqref{E:HhIneq} in Section \ref{S:PropertiesHh}.

\rhn{
Our approach is similar in spirit to that developed by Pryer\cite{Pryer}, who in turn is inspired by Di Pietro and Ern\cite{DiPietroErn} for second order problems. However, an important distinction is worth pointing out.
The finite element spaces considered in \cite{DiPietroErn,Pryer} consist of polynomial functions (of total degree $\leq k$) on the physical elements $T\in\mathcal{T}_h$.
Instead, the dG method proposed in this work uses polynomial functions on the reference element $\widehat{T}$ \rhn{(see \eqref{E:discrete-space}), which is consistent with the implementation within the deal.ii library\cite{dealii85,dealiipaper} responsible for the numerical experiments of Section~\ref{S:NumExPl}. Consequently, several challenges arise when dealing with isoparametric elements (in particular quadrilaterals elements): (i) the mappings $F_T:\widehat T \rightarrow T$ between $\widehat{T}$ and $T\in\mathcal{T}_h$ are non-affine and do not produce polynomial functions on $T$;} (ii)  $\nabla_h \Vhk$ does not embed in $\lbrack \mathbb V_h^{k'}\rbrack^2$ for some $k'$ as required, for example, for the analyses in \cite{DiPietroErn,Brezzi,Cockburn}; (iii) the Hessians of $D^2 v_h|_T$ and $D^2 (v_h\circ F_T)$ are not proportional because $D^2 F_T \ne \boldsymbol{0}$.
As a consequence, the space of Hessians
\begin{equation}\label{e:hessian_space}
\mathbb H_h^k:= \{ \tau_h \in  L^2(\Omega)^{2\times 2} \ : \ \tau_h|_{T} = D^2_h w_h, \quad w_h \in  \Vhk(T), \ T\in \mathcal T_h\},
\end{equation}
where the reconstructed Hessian $H_h[v_h]$ of a function $v_h\in\Vhk$ belongs to, is not a subspace of $\lbrack\mathbb V_h^{k}\rbrack^{2\times 2}$ in general; $\tau_h\in\mathbb H_h^k$ might not be piecewise polynomial.
}

\subsection{Lifting operators and definition of $H_h[y_h]$}\label{S:Hessian-def}
\rhn{Given a scalar-valued function $v_h\in\Vhk$,} we define two lifting operators $R_h([\nabla_h v_h])$ and $B_h([v_h])$ that extend the jumps $[\nabla_h v_h]$ and $[v_h]$ of \rhn{$\nabla_h v_h$ and $v_h$ from the full} skeleton $\Gamma_h$ to the bulk $\Omega$, \rhn{as in Brezzi et al.\cite{Brezzi}.} It turns out that of the many ways to achieve this, there is only one that leads to \eqref{E:weak-convergence} and \eqref{E:HhIneq}. We describe this next.

\rhn{Given an edge $e \in \calE_h$, let $\omega(e)$ be the patch associated with $e$ (i.e. the union of two elements sharing $e$ for interior edges $e\in\Ei$ or just one single element for boundary edges $e\in\Eb$), and let $\Hhk(\omega(e))$ stand for the restriction of functions in $\Hhk$ to $\omega(e)$. The first lifting operator
$R_h := \sum_{e \in \calE_h} r_e:\rhn{[L^2(\Gamma_h)]^2} \to \mathbb H_h^k$ hinges on the
local liftings $r_e: [L^2(e)]^2 \to \mathbb H_h^k$ which, for all $e \in \calE_h$ and $\phi \in [L^2(e)]^2$, are defined by}
\begin{equation}\label{E:lift1}
\rhn{r_e(\phi)\in \Hhk(\omega(e)):} \
\int_{\omega(e)} r_e(\phi) : \tau_h = 
\int_e  \phi\cdot \left\{\tau_h \right\}  \mu_e
\quad\forall \, \tau_h \in \rhn{\Hhk(\omega(e)),}
\end{equation}
and vanish outside $\omega(e)$. \rhn{Notice that upon taking $\tau_h = D_h^2 w_h$, we get}
\begin{equation}\label{E:lift-1}
\big(R_h([\nabla_h v_h]), \rhn{D_h^2 w_h} \big)_{L^2(\Omega)} 
 = \big([\nabla_h v_h], \left\{  \nd(\nabla_h w_h)  \right\} \big)_{L^2(\Gamma_h)}
\quad\forall \, w_h\in\Vhk.
\end{equation}

The second lifting operator
$B_h := \sum_{e\in \calE_h} b_e:L^2(e)\to \mathbb H_h^k$
relies on the local liftings $b_e:L^2(e)\to \mathbb H_h^k$ which,
for all $e\in\calE_h$ and $\phi\in L^2(e)$, are given by
\begin{equation}\label{E:lift2}
\rhn{b_e(\phi) \in \Hhk(\omega(e)):} \  
\int_{\omega(e)} b_e(\phi) : \tau_h = 
\int_e  \phi \ \left\{ {\rm div} \ \tau_h \right\} \cdot \mu_e 
\quad \forall \, \tau_h \in \rhn{\Hhk(\omega(e)),}
\end{equation}
and vanish outside $\omega(e)$. 
In this case, taking $\tau_h = D_h^2 w_h$ again and observing that
$
{\rm div} ( D_h^2 v_h ) = \nabla_h(\Delta_h v_h)
$
elementwise, we obtain
\begin{equation}\label{E:lift-2}
\big(B_h([v_h]), \rhn{D^2_h w_h}  \big)_{\rhn{L^2(\Omega)}} = \big([v_h],\left\{ \nd(\Delta_h w_h) \right\}\big)_{L^2(\Gamma_h)} \quad\forall\, w_h \in \Vhk.
\end{equation}

\rhn{To deal with vector-valued functions, we apply \eqref{E:lift1} and \eqref{E:lift2} componentwise.}

\begin{Definition}[discrete Hessian]\label{D:Hh}
We define $H_h: \Vhk(g,\Phi) \to \lbrack \mathbb H_h^k \rbrack^3$ to be
	\begin{equation} \label{E:Hh}
	\begin{aligned}
	  H_h[y_h]&:= D_h^2 y_h - R_h([\nabla_h y_h]) + B_h([y_h]),
	\end{aligned}
	\end{equation}
where \rhn{$R_h=\sum_{e\in\calE_h} r_e$ and $B_h=\sum_{e\in\calE_h} b_e$} are defined
in \eqref{E:lift1} and \eqref{E:lift2}.          
\end{Definition}

It is worth realizing that, in view of \eqref{E:lift-1} and \eqref{E:lift-2}, we
readily find
\begin{equation}\label{e:hessian_match}
\begin{aligned}
\!\! \big( H_h[y_h], \tau_h  \big)_{L^2(\Omega)} &= \big( D^2_h y_h, \rhn{D_h^2 w_h} \big)_{L^2(\Omega)}
\\
& - \big( [\nabla_h y_h],\left\{  \nd(\nabla_h w_h)  \right\} \big)_{L^2(\Gamma_h)}
+ \big([y_h], \rhn{\left\{ \nd(\Delta_h w_h) \right\}}\big)_{L^2(\Gamma_h)}
\end{aligned}
\end{equation}
for all \rhn{$\tau_h = D_h^2 w_h\in \lbrack \Hhk \rbrack^3$ and $y_h\in\Vhk(g,\Phi)$,} and that the boundary jumps  on $\Gamma_h^b$
are given by $[y_h]=y_h - g$ and $[\nabla_h y_h] = \nabla_h y_h - \Phi$, according to
\rhn{\eqref{E:bd-jumps} and \eqref{discrete-set}.
Relation \eqref{e:hessian_match} is a key tool in the analysis of our dG method and justifies using the Hessian space \eqref{e:hessian_space} for the definition of reconstructed Hessian $H_h[y_h]$. In contrast, Pryer \cite{Pryer} assumes that $y_h$ is piecewise polynomial of total degree $\leq k$ in $\Omega$ and $H_h[y_h]$ is piecewise polynomial of total degree $k-2$ in $\Omega$.}

\subsection{Properties of $H_h[y_h]$}\label{S:PropertiesHh}
We start with $L^2$ a priori bounds for $R_h, B_h$ and $H_h$.

\begin{Lemma}[$L^2$-bounds of lifting operators] \label{L:LiftBounds}
  Let $y_h\in\Vhk(g,\Phi)$ with data $(g,\Phi)$ satisfying \eqref{E:data}.
  Then, for all $e \in \Gamma_h$, there holds
$$
\| r_e([\nabla_h y_h])\|_{L^2(\omega(e))}
\lesssim \big\|h^{-1/2} [\nabla_h y_h]\big\|_{L^2(e)}
$$
	and
$$
\| b_e([y_h])\|_{L^2(\omega(e))} \lesssim 
\big\|h^{-3/2} [y_h]\big\|_{L^2(e)}. 	
$$
\end{Lemma}
\begin{proof}
We argue as in \cite{Brezzi} but with emphasis on boundary edges $e\in\Eb$
because they contain the Dirichlet data $(g,\Phi)$ \rhn{in view of \eqref{E:bd-jumps};}
see also \cite{BoNo,DiPietroErn,Pryer}. We prove the first bound because the other
one is identical. Definition \eqref{E:lift1} of $r_e$ yields
	$$
	\| r_e(\nabla_h y_h-\Phi)\|_{L^2(\omega(e))}^2 = \rhn{\int_e (\nabla_h y_h-\Phi) : \left\{ r_e(\nabla_h y_h-\Phi) \right\}  \mu_e.} 
	$$
\rhn{Since $\left\{ r_e(\nabla_h y_h-\Phi) \right\}_e = r_e(\nabla_h y_h-\Phi)$ for all $e\in\Eb$, according to \eqref{averages}, invoking the inverse estimate \eqref{E:inverse} for $D^2 v_h = r_e(\nabla_h y_h-\Phi)$ implies
	\begin{align*}
	  \| r_e(\nabla_h y_h-\Phi)\|_{L^2(\omega(e))}^2
        &\le \| h^{-1/2} (\nabla_h y_h-\Phi) \|_{L^2(e)} \, \| h^{1/2}r_e(\nabla_h y_h-\Phi)\|_{L^2(e)} \\
        & \lesssim \| h^{-1/2} (\nabla_h y_h-\Phi) \|_{L^2(e)} \, \| r_e(\nabla_h y_h-\Phi)\|_{L^2(\omega(e))} ,
	\end{align*}
and yields the desired bound for $e\in\Eb$. For interior edges $e\in\Ei$, we replace $\nabla y_h-\Phi$ with $[\nabla_h y_h]_e$ given by \eqref{jumps}, and apply the same argument.
}
\end{proof}

\rhn{Since each element $T \in\mathcal{T}_h$ intersects at most
$4$ sets $\omega(e)$, we immediately get}
\begin{equation}\label{E:boundsRB}
\begin{aligned}
\|R_h([\nabla_h y_h])\|_{L^2(\Omega)} & \lesssim
\|h^{-1/2} [\nabla_h y_h]\|_{L^2(\Gamma_h)},
\\
\|B_h([y_h])\|_{L^2(\Omega)} &\lesssim
\|h^{-3/2} [y_h]\|_{L^2(\Gamma_h)},
\end{aligned}
\end{equation}
as well as the following corollary.

\begin{Corollary}[$L^2$-bound of discrete Hessian]\label{C:HhBound}
  If $y_h\in\Vhk(g,\Phi)$ with data $(g,\Phi)$ satisfying \eqref{E:data},
  then the following bound holds
	$$
	\| H_h[y_h]\|_{\lom} \lesssim \rhn{\|y_h\|_{H_h^2(\Omega)}.}
	$$
\end{Corollary}
\begin{proof}
  Combine \MakeUppercase {D}efinition\nobreakspace \ref {D:Hh} (discrete Hessian)
  with \rhn{\eqref{E:boundsRB} and \eqref{E:energynormN}.} 
\end{proof}

While there is some flexibility in the definitions \eqref{E:lift1} and
\eqref{E:lift2} of \rhn{local} lifts $r_e$ and $b_e$ for Lemma \ref{L:LiftBounds}
($L^2$-bounds
of lifting operators) to be valid, the following result reveals
that these definitions are just right for the weak convergence of $H_h[y_h]$ \rhn{provided the following condition holds on the sequence of meshes $\{ \mathcal T_h \}_{h>0}$
\begin{equation}\label{E:assumptions-F}
  \frac{\|D^m F_T^{-1}\|_{L^\infty(T)}}{\|DF_T^{-1}\|_{L^\infty(T)}^m} \le \xi(h) \rightarrow 0
  \quad\text{as } h_T\to0,
\end{equation}
for all $T=F_T(\widehat{T})\in\mathcal{T}_h$ and $m=2,3$.
Note that \eqref{E:assumptions-F} is valid for affine equivalent elements because $F_T^{-1}$ is affine and $\xi(h)=0$. However, \eqref{E:assumptions-F} is more restrictive for non-affine mappings $F_T$ (i.e. $F_T \in \lbrack \mathbb P^k \rbrack^2$ for $k\geq 2$ or $F_T \in \lbrack \mathbb Q^k \rbrack^2$ for $k\geq 1$).}

\begin{Proposition}[{weak convergence of $H_h[y_h]$}] \label{P:WeakConv}
\rhn{Let \eqref{E:assumptions-F} hold and} $y_h\in\Vhk(g,\Phi)$ with data $(g,\Phi)$ satisfying \eqref{E:data}.
If \rhn{$\|y_h\|_{H_h^2(\Omega)} \leq \Lambda$, with $\Lambda>0$ independent of $h$,} and
$y_h$ converges to a function $y \in [H^2(\Omega)]^3$ in $[\lom]^3$ as $\hb\to0$, then
\looseness=-1
$$
H_h[y_h] \rightharpoonup D^2y \quad \text{in} \quad  [\lom]^{3 \times 2 \times 2}.
$$
\end{Proposition}
\begin{proof}
\rhn{Since we know $D^2 y\in [L^2(\Omega)]^{3\times 2 \times 2}$ by assumption,}
we need to show that
\[
\int_\Omega H_h[y_h] : \tau \rightarrow \int_\Omega D^2 y : \tau
\quad\text{ as $\hb\to0$}
\]
for all $\tau \in [C_0^\infty(\Omega)]^{3\times 2 \times 2}$. \rhn{First
we show this property relying on a suitable approximation $\tau_h\in[\mathbb{H}_h^k]^3$ of $\tau$, and next we construct such $\tau_h$ using \eqref{E:assumptions-F}.}

\medskip\noindent
{\it Step 1: Weak convergence.} We argue with each component of $y_h$,
integrate by parts elementwise and utilize the definition \eqref{E:Hh}
of $H_h[y_h]$ \rhn{and elementary calculations} to deduce
\begin{align*}
  \int_\Omega H_h[y_h] : \tau &=
  \sum_{T\in\mathcal{T}_h} \int_T y_h \cdot \div\div \tau
  \\
  & - \sum_{e\in\calE_h} \int_{\omega(e)} \rhn{r_e\big([\nabla_h y_h]\big)} : \big(\tau - \tau_h\big)
  + \sum_{e\in\calE_h} \int_{\omega(e)} \rhn{b_e \big([y_h]\big)} : \big(\tau - \tau_h\big)
  \\
  & + \sum_{e\in\Ei} \int_e \rhn{[\nabla_h y_h]:\big\{ \tau - \tau_h  \big\}\mu_e}
  -  \sum_{e\in\Ei} \int_e \rhn{[y_h] \cdot \big\{ \div_{\! h}\big(\tau - \tau_h\big) \big\}\mu_e}
  \\
  & \rhn{-\sum_{e\in\Eb} \int_e (\nabla_h y_h-\Phi): \tau_h \, \mu_e}+ \sum_{e\in\Eb} \int_e \rhn{(y_h-g) \cdot  \div_{\! h} \tau_h \, \mu_e} \, ,
\end{align*}
\rhn{provided $\tau_h = D_h^2 w_h\in [\mathbb{H}_h^k]^3$ is a discrete Hessian.
We point out that
there is a contribution of boundary terms $e\in\Eb$ because they appear in
Definition \ref{D:Hh} (discrete Hessian). Combining the bound
$\|y_h\|_{H_h^2(\Omega)} \leq \Lambda$ with \eqref{E:energynormN} and \eqref{E:boundsRB} gives \looseness=-1
\begin{equation}\label{E:Rh-Bh}
\begin{split}
  \|R_h([\nabla_h y_h])\|_{L^2(\Omega)}+\|B_h([y_h])\|_{L^2(\Omega)} &\lesssim\Lambda,
  \\
  \|h^{-1/2} [\nabla_h y_h]\|_{L^2(\Gamma_h)}
  +\|h^{-3/2} [y_h]\|_{L^2(\Gamma_h)} & \lesssim \Lambda.
\end{split}
\end{equation}
Suppose that there exists $w_h\in[\Vhk]^3$ such that $\tau_h=D_h^2 w_h$ satisfies
\begin{equation}\label{e:needed}
  \|\tau - \tau_h\|_{L^2(\Omega)} + \|h\nabla_h (\tau-\tau_h)\|_{L^2(\Omega)}
  + \|h^2D_h^2 (\tau-\tau_h)\|_{L^2(\Omega)} \to 0 \quad \textrm{as } h\to0.
\end{equation}
The scaled trace inequalities
\begin{gather*}
  \|h^{1/2} (\tau-\tau_h)\|_{L^2(\Gamma_h)} \lesssim
  \|\tau-\tau_h\|_{L^2(\Omega)} +  \|h\nabla_h(\tau-\tau_h)\|_{L^2(\Omega)}
  \\
  \|h^{3/2} \div_{\! h}(\tau-\tau_h)\|_{L^2(\Gamma_h)} \lesssim
  \|h \nabla_h(\tau-\tau_h)\|_{L^2(\Omega)} +  \|h^2 D_h^2(\tau-\tau_h)\|_{L^2(\Omega)},
\end{gather*}
together with \eqref{E:Rh-Bh} and \eqref{e:needed}, readily yield
\begin{align*}
\Big| \int_\Omega H_h[y_h]:\tau &- \int_\Omega y_h \cdot \div\div \tau \Big|
\lesssim \Lambda \|\tau - \tau_h\|_{L^2(\Omega)}
\\
& + \Lambda \Big( \|h\nabla_h(\tau-\tau_h)\|_{L^2(\Omega)}
+ \|h^2 D^2_h (\tau-\tau)h) \|_{L^2(\Omega)}\Big) \to 0.
\end{align*}
Since $y_h\to y$ in $[L^2(\Omega)]^3$, we obtain
\begin{equation}\label{e:full_weak}
\int_\Omega H_h[y_h]:\tau \rightarrow \int_\Omega y \cdot \div\div \tau
= \int_\Omega D^2 y : \tau \quad\textrm{as } h\to0,
\end{equation}
because $y\in [H^2(\Omega)]^3$ and $\tau\in [C_0^\infty(\Omega)]^{3\times2\times2}$.
This is the asserted limit.}

\medskip\noindent
\rhn{
{\it Step 2: Proof of \eqref{e:needed}.}
We argue componentwise and construct a scalar $w_h \in \Vhk$ such that $\tau_h = D_h^2 w_h\in \mathbb{H}_h^k$ satisfies \eqref{e:needed} for $\tau\in [C_0^\infty(\Omega)]^{2\times2}$. To this end, we choose an arbitrary element $T\in\mathcal{T}_h$ with isoparametric map $F=F_T:\hT\to T$ from the reference element $\hT$; see Section \ref{S:A-quads}. We assume that $\hT$ contains the origin, and let $x_T:=F(0)\in T$. If quadratic polynomials were contained in $\Vhk(T)$, the obvious idea would be to construct a quadratic approximation of $\tau$ in $T$ by Taylor expansion. Since this is not true for isoparametric maps $F$, we build a quadratic approximation in $\hT$ instead. We denote by
$\overline{\tau}_T$ the meanvalue of $\tau$ within $T$, set
\[
\widehat{\tau} := DF^T(0) \, \overline{\tau}_T \, DF(0) \in \mathbb{R}^{2\times2},
\]
and note that $\|\widehat{\tau}\|_{L^2(\hT)} \lesssim h_T \|\overline{\tau}_T\|_{L^2(T)}$.
Let $w_h:=\hw\circ F^{-1}\in\Vhk(T)$ where $\hw$ is the quadratic function $\hw (\hx) = \frac12 \hx^T \widehat{\tau} \hx$ in $\hT$; hence $D^2\hw(\hx)=\widehat{\tau}$ for all $\hx\in\hT$. According to \eqref{E:point-second}, the Hessian of $w_h$ within $T$ reads
\[
\partial_{ij} w_h(x) = \sum_{m,n} \, \partial_{mn} \hw(\hx) \, \partial_j F_n^{-1}(x)\partial_i F_m^{-1}(x) + \sum_m \partial_m \hw(\hx) \, \partial_{ij} F_m^{-1}(x).
\]
Since $\nabla \hw(0)=0$, we infer that
\[
D^2 w_h(x_T) = DF^{-T}(x_T) \, \widehat\tau \, DF^{-1}(x_T) = \overline{\tau}_T,
\]
whence Poincar\'e inequality implies
\[
\|\tau - D^2 w_h(x_T)\|_{L^2(T)} \lesssim h_T \|D\tau\|_{L^2(T)}.
\]
We need to control the deviation of $D^2 w_h$ from constant due to $F$;
for affine equivalent elements, $D^2 w_h=D^2 w_h(x_T)$ turns out to be constant.
It is here that we exploit \eqref{E:assumptions-F} to deal with $F$.
Differentiating $\partial_{ij} w_h$ and using $D^3\hw=0$, we discover
\begin{align*}
  \|h D^3 w_h\|_{L^2(T)} &\lesssim h_T^2 \|DF^{-1}\|_{L^\infty(T)}^3 \|D^3\hw\|_{L^2(\hT)}
  \\
  & + h_T^2 \|D^2F^{-1}\|_{L^\infty(T)}\|DF^{-1}\|_{L^\infty(T)}
  \|D^2\hw\|_{L^2(\hT)}
  \\
  & + h_T^2 \|D^3F^{-1}\|_{L^\infty(T)}\|\nabla\hw\|_{L^2(\hT)}
   \lesssim \xi(h) h_T^{-1} \|\widehat\tau\|_{L^2(\hT)}
  \lesssim \xi(h) \|\overline{\tau}_T\|_{L^2(T)},
\end{align*}
whence $\delta D^2 w_h = D^2 w_h - D^2 w_h(x_T)$ satisfies
\begin{align*}
  \|\delta D^2 w_h\|_{L^2(T)} &\lesssim h_T\|\delta D^2 w_h\|_{L^\infty(T)}
  \\
  & \lesssim h_T^2 \|D^3 w_h\|_{L^\infty(T)} \lesssim h_T \|D^3 w_h\|_{L^2(T)}
  \lesssim \xi(h) \|\overline{\tau}_T\|_{L^2(T)}
\end{align*}
and $\|\tau-\tau_h\|_{L^2(\Omega)}\to0$ as $h\to0$ because
$\sum_{T\in\mathcal{T}_h} \| \overline{\tau}_T\|_{L^2(T)}^2 \le \|\tau\|_{L^2(\Omega)}^2$.
Moreover, we also see that
\begin{equation*}
  \|h\nabla_h(\tau-\tau_h)\|_{L^2(\Omega)} \lesssim
  \|h \nabla\tau\|_{L^2(\Omega)} + \|h D_h^3 w_h\|_{L^2(\Omega)} \to 0
  \quad\textrm{as } h\to0.
\end{equation*}
Arguing with $D_h^4 w_h$ similarly to $D_h^3 w_h$, we easily derive
\begin{gather*}
  \|h^2 D^4 w_h\|_{L^2(T)}\lesssim \xi(h) \|\overline{\tau}_T\|_{L^2(T)},
\end{gather*}
as well as $\|h^2 D_h^2(\tau-\tau_h)\|_{L^2(\Omega)} \to 0$ as $h\to0$.
This shows \eqref{e:needed} as desired. 
}
\end{proof}

\rhn{
We reiterate that condition \eqref{E:assumptions-F} is innocuous for affine equivalent elements because $D^2F_T^{-1}=0$, but is rather restrictive otherwise (e.g quadrilateral meshes). In order to avoid \eqref{E:assumptions-F}, we content ourselves with a weaker statement yet sufficient for $\Gamma$-convergence of $E_h$ towards $E$. We state this next.}

\rhn{
\begin{Proposition}[lim-inf property of the reconstructed Hessian]\label{p:liminf_H}
Let $y_h\in\Vhk(g,\Phi)$ and $y \in \mbV(g,\Phi)$
with data $(g,\Phi)$ satisfying \eqref{E:data}.
If $\|y_h\|_{H_h^2(\Omega)} \leq \Lambda$, with $\Lambda>0$ independent of $h$, and
$y_h\to y$ in $[\lom]^3$, and $y_h\to y$, $\nabla_h y_h \to \nabla y$ in $L^2(\partial\Omega)$ as $\hb\to0$, then
there exists $H \in \lbrack L^2(\Omega)\rbrack^{3\times 2 \times 2}$ such that $H_h[y_h]  \rightharpoonup H$ in $\lbrack L^2(\Omega)\rbrack^{3\times 2\times 2}$ and
\begin{equation}\label{e:ortho_H}
 \int_\Omega (H-D^2 y) : D^2 w = 0, \qquad \forall w \in \lbrack H^2(\Omega)\rbrack^{3}.
\end{equation}
Furthermore, the following estimate holds
\begin{equation}\label{e:liminf_H}
\| D^2 y \|_{L^2(\Omega)} \leq \liminf_{h \to 0} \| H_h[y_h] \|_{L^2(\Omega)}.
\end{equation}
\end{Proposition}
\begin{proof}
We assume that the reference element $\widehat T$ is the unit square and omit the somewhat simpler case of the unit triangle for which interpolations estimates such as \eqref{E:error-quad-curved} are standard.
We split the proof into two steps.

\medskip\noindent
{\it Step 1: Proof of \eqref{e:ortho_H}.}
Thanks to Corollary~\ref{C:HhBound} ($L^2$-bound on the Hessian) and the boundedness assumption $\|y_h\|_{H_h^2(\Omega)} \leq \Lambda$, there exists $H \in L^2(\Omega)^{3\times 2 \times 2}$ such that
\begin{equation}\label{e:H_H_h}
H_h[y_h]  \rightharpoonup H  \quad  \textrm{in }\lbrack L^2(\Omega)\rbrack^{3\times 2\times 2}  \quad \Rightarrow \quad
\| H \|_{L^2(\Omega)} \leq \liminf_{h \to 0} \| H_h[y_h] \|_{L^2(\Omega)}.
\end{equation}
Proceeding as in Step 1 of the proof of Proposition~\ref{P:WeakConv} with $\tau\in [C^\infty(\overline{\Omega})]^{2 \times 2}$ we get boundary terms involving both $\tau$ and $\div\tau$ which result in the last line of the expression for $\int_\Omega H_h[y_h] : \tau$ being replaced by $I + II + III$, where
\begin{align*}
 I =  &- \int_{\partial_D\Omega} \big( \nabla_h y_h - \Phi \big) : \big(\tau_h-\tau\big) \mu
    + \int_{\partial_D\Omega} \big(y_h - g \big) \cdot \div_{\! h} \big(\tau_h-\tau\big) \mu,
   \\
 II = &  \int_{\partial\Omega\setminus\partial_D\Omega} \big( \nabla_h y_h - \nabla y \big) : \tau \mu - \int_{\partial\Omega\setminus\partial_D\Omega} \big(y_h - y \big) \cdot \div\tau \mu,
   \\
 III = &  \int_{\partial\Omega} \nabla y : \tau \mu - \int_{\partial\Omega} y \cdot \div\tau \mu.
\end{align*}
Let $\tau = D^2  w$ for $w \in \lbrack C^\infty(\overline\Omega)\rbrack^{3}$ and $\tau_h = D^2_h w_h$ where $w_h$ is chosen as follows: given $T\in\mathcal{T}_h$, let $\hw_T=w|_T\circ F$, $\widehat{P}\hw_T$ be the $L^2$-projection of $\hw_T$ onto the space $\mathbb{Q}_k$, and $w_h|_T = \widehat{P}\hw_T \circ F^{-1}$. We intend to prove the estimates
\begin{equation}\label{E:est-tau}
  \| \tau - \tau_h \|_{L^2(\Omega)} 
  + \| h \nabla_h (\tau - \tau_h) \|_{L^2(\Omega)}
  + \| h^2 D_h^2 (\tau - \tau_h) \|_{L^2(\Omega)} \lesssim h,
\end{equation}
which are refinements of \eqref{e:needed} for $w \in \lbrack C^\infty(\overline\Omega)\rbrack^{3}$.
In view of \eqref{E:est-tau}, employing the same argument as in Step 1 of Proposition
\ref{P:WeakConv}, together with the assumption that $y_h\to y$ and $\nabla_h y_h\to\nabla y$ in $L^2(\partial\Omega)$, we infer that $|I|, |II| \to 0$ as well as
\begin{equation*}
\int_\Omega H_h[y_h]:D^2 w \rightarrow \int_\Omega y \cdot  \div\div D^2 w - \int_{\partial \Omega} y \cdot \div D^2 w \, \mu + \int_{\partial \Omega} \nabla y : D^2 w \, \mu,
\end{equation*}
as $h\to 0$ for all $w \in \lbrack C^\infty(\overline\Omega)\rbrack^{3}$. 
Integrating twice back by parts, we deduce
\begin{equation}\label{e:partial_wc}
\int_\Omega H_h[y_h]:D^2 w \rightarrow \int_\Omega D^2 y :  D^2 w \quad \textrm{as }h  \to 0.
\end{equation}
Upon approximating any $H^2(\Omega)$ by functions in $[C^\infty(\overline\Omega)]$, relation \eqref{e:partial_wc} is also valid for any $w \in [H^2(\Omega)]^3$. This combined with \eqref{e:H_H_h} yields the orthogonality property \eqref{e:ortho_H}. The latter in turn implies $\|D^2 y\|_{L^2(\Omega)} \le \| H \|_{L^2(\Omega)}$ as well as \eqref{e:liminf_H}.

\medskip\noindent
{\it Step 2: Proof of \eqref{E:est-tau}.}    
We recall that $F:\hT\to T$ is the isoparametric map between $\hT$
and $T\in\mathcal{T}_h$.
The first estimate in \eqref{E:est-tau} is similar to \eqref{E:error-quad-curved}:
\[
\|\tau-\tau_h\|_{L^2(T)} = \|D^2 (w - w_h)\|_{L^2(T)} \lesssim
h_T \|w\|_{H^3(T)}
\]
upon replacing the Lagrange interpolant $I_h w$ by $w_h$ defined in Step 1.
To prove the middle estimate in \eqref{E:est-tau} we observe that if $k\ge 3$,
then \eqref{E:error-quad-curved} yields
\[
\|\nabla(\tau-\tau_h)\|_{L^2(T)} = \|D^3 (w - w_h)\|_{L^2(T)} \lesssim
h_T \|w\|_{H^4(T)}.
\]
However, this estimate is inadequate for $k=2$ for which we revisit the proof
of Lemma \ref{L:error-quad} (error estimate for quadrilaterals) and add and subtract
the $L^2$-projection $\widehat{P}_3\hw$ onto the space of polynomials of degree
$\le 3$ to arrive at
\begin{equation*}
\|D^3 (w - w_h)\|_{L^2(T)} \lesssim h_T^{-2}
\sum_{j=1}^3 \left( \|D^j(\hw - \widehat{P}_3\hw)\|_{L^2(\hT)}
+ \|D^j(\widehat{P}_3\hw - \widehat{P}\hw)\|_{L^2(\hT)}\right).
\end{equation*}
For the first term, \rhn{we use the Bramble-Hilbert estimate in $\mathbb{Q}_3$ to deduce
\[
\sum_{j=1}^3 \|D^j(\hw - \widehat{P}_3\hw)\|_{L^2(\hT)} \lesssim \|[D^4\hw]\|_{L^2(\hT)} \lesssim h_T^3 \|w\|_{H^4(T)}
\]
in view of \eqref{E:Dk+1-iso}.
For the second term, instead, we use an inverse estimate
\begin{equation*}
  \|D^3(\widehat{P}_3\hw - \widehat{P}\hw)\|_{L^2(\hT)}
  \lesssim \|D^2(\widehat{P}_3\hw - \widehat{P}\hw)\|_{L^2(\hT)},
\end{equation*}
and next add and subtract $\hw$. We first realize that
$\|D^2(\hw - \widehat{P}_3\hw)\|_{L^2(\hT)}$ may be tackled as in the case $k\ge3$
whereas
\[
\sum_{j=1}^2 \|D^j(\hw - \widehat{P}\hw)\|_{L^2(\hT)}\lesssim \|[D^3 \hw]\|_{L^2(\hT)}\lesssim h_T^2 \|w\|_{H^3(T)}
\]
according to \eqref{E:Dk+1-iso} again.} Altogether, we obtain the bound
\[
\|D^3 (w - w_h)\|_{L^2(T)} \lesssim
h_T \|w\|_{H^4(T)} + \|w\|_{H^3(T)} \lesssim \|w\|_{H^4(T)},
\]
which implies the
desired middle estimate in \eqref{E:est-tau}. Exactly the same procedure
leads to the remaining estimate for $\|D^4 (w - w_h)\|_{L^2(T)}$ provided
$w\in H^5(T)$, except that
this time we invoke the $L^2$-projection $\widehat{P}_4$ onto polynomials of
degree $\le 4$ and an inverse estimate to relate $\| D^4 \widehat v \|_{L^2(\widehat T)}$ to $\| D^2 \widehat v \|_{L^2(\widehat T)}$ for $\widehat v$ polynomial; we thus deduce
\rhn{$h_T \|D^4 (w-w_h)\|_{L^2(T)} \lesssim \|w\|_{H^5(T)}$.}
This concludes the proof.
\end{proof}
}

Another consequence of Definition \ref{D:Hh} (discrete Hessian) and Lemma \ref{L:LiftBounds} ($L^2$-bounds of lifting operators) is \eqref{E:HhIneq} with constant $1$. This is crucial for the rest.

\begin{Proposition} [{relation between $E_h[y_h]$ and $H_h[y_h]$}] \label{P:HhIneq}
Let data $(g,\Phi,f)$ satisfy \eqref{E:data} and $y_h\in\Vhk(g,\Phi)$.
Then the discrete energy \rhn{$E_h[y_h]$  defined in \eqref{E:DEn}
and the discrete Hessian $H_h[y_h]$} defined in \eqref{E:Hh} satisfy
	\begin{equation*}
	\frac{1}{2} \|H_h[y_h]\|_\lom^2 - (f,y_h)_\lom \leq E_h[y_h],
	\end{equation*}
provided the penalty parameters $\gamma_0$ and $\gamma_1$ are chosen sufficiently large.
\end{Proposition}
\begin{proof}
We expand the expression for $H_h[y_h]$ and utilize
\eqref{E:lift-1} and \eqref{E:lift-2} to obtain
\begin{equation}\label{e:for_later}
  \frac{1}{2} \int_{\Omega} |H_h[y_h]|^2 - \rhn{\int_{\Omega} f \cdot y_h = E_h[y_h]
  - J_h[y_h],}
\end{equation}
where
\begin{align*}
  \rhn{J_h[y_h]} := &-\frac{1}{2} \big\| B_h([y_h])-R_h([\nabla_h y_h]) \big\|_{\lom}^2
  \\
  & + \gamma_0 \big\| h^{-3/2} [y_h] \big\|_{L^2(\Gamma_h)}^2
    + \gamma_1 \big\| h^{-1/2} [\nabla_h y_h] \big\|_{L^2(\Gamma_h)}^2.
\end{align*}
In view of \eqref{E:boundsRB} there is a constant  
$C$ independent of $h$ such that
\begin{equation}\label{e:jump_control}
\rhn{J_h[y_h] \geq (\gamma_0-C) \big\|h^{-3/2} [y_h]\big\|_{L^2(\Gamma_h)}^2
+ (\gamma_1-C) \big\|h^{-1/2} [\nabla_h y_h]\big\|_{L^2(\Gamma_h)}^2 \ge 0}
\end{equation}
for $\gamma_0$ and $\gamma_1$ sufficiently large. This is what we intended to prove.
\end{proof}

It is convenient to point out that $\gamma_0,\gamma_1$ sufficiently large in 
\eqref{e:jump_control} also yields
\begin{equation}\label{E:sign-J}
\rhn{  \big\|h^{-3/2} [y_h]\big\|_{L^2(\Gamma_h)}^2 + \big\|h^{-1/2} [
   \nabla_h y_h]\big\|_{L^2(\Gamma_h)}^2
  \lesssim J_h[y_h].
}  
\end{equation}

\section{$\Gamma$-Convergence of $E_h$} \label{S:ConvergencePl}

In this section we prove the $\Gamma$-convergence of $E_h$ to $E$. \rhn{This consists
of three parts. We first show equi-coercivity and compactness for $y_h \in \Ahk(g,\Phi)$
in Section \ref{S:equi-coercive}. We next prove a lim-inf inequality for $E_h[y_h]$ in
Section \ref{S:lim-inf} and a lim-sup property in Section \ref{ss:limsup}; the former
uses lower semicontinuity of the $L^2$-norm whereas the latter hinges on constructing a
recovery sequence for any $y \in \mathbb A(g,\Phi)$ via regularization and interpolation.
These three results are responsible for $\Gamma$-convergence of $E_h$
as well as convergence (up to a subsequence) of global minimizers $y_h$ of
$E_h$ in $\Ahk(g,\Phi)$ to global minimizers of $E$ in $\mathbb A(g,\Phi)$; this is
the topic of Section \ref{S:G-convergence}.
Lastly, we combine these results in Section \ref{S:strong-conv}
to show that the stabilization terms}
$$
\big\|h^{-3/2}[y _h]\big\|_{L^2(\Gamma_h)} \rightarrow 0, \quad 
\big\|h^{-1/2}[\nabla_h y _h]\big\|_{L^2(\Gamma_h)} \rightarrow 0 \quad \textrm{as }\hb\to 0,
$$
for such \rhn{$y_h$ and that $H_h[y_h]$ in Proposition \ref{p:liminf_H} (lim-inf property of the reconstructed Hessian) as well as $D_h^2 y_h$ converge to $D^2y$ strongly in $[\lom]^{3\times 2 \times 2}$.}

\subsection{Equi-coercivity and compactness}\label{S:equi-coercive}

If data $(g,\Phi,f)$ satisfies \eqref{E:data} and $y_h \in \Vhk(g,\Phi)$ possesses
a uniform bound $E_h[y_h] \leq \Lambda$ for all $h>0$, then \rhn{Lemma \ref{L:Coercivity}} (coercivity of $E_h$) guarantees {\it equi-coercivity}
\begin{equation}\label{E:equi-coerc}
  \rhn{\|y_h\|_{H_h^2(\Omega)}^2} \lesssim E_h[y_h]  + \rhn{C_F^* R(g,\Phi,f)\le C,}
\end{equation}
\rhn{where $C>0$ is independent of $h$.} This, together with \rhn{Friedrichs-type estimate} \eqref{e:H2typePoincare}, leads to weakly converging subsequences.

We now establish the $L^2$-compactness property.

\begin{Proposition}[compactness in $\lom$] \label{P:L2conv}
  \rhn{Let data $(g,\Phi,f)$ satisfy \eqref{E:data} and}
  let the sequence $\{y_h\}_{h>0} \subset \rhn{\Ahk(g,\Phi)}$ satisfy the uniform bound 
	$E_h[y_h] \leq \Lambda$
	for all $h>0$ with $\Lambda$ independent of $h$.
	Then there exists a subsequence (not relabeled) and a function
        \rhn{$y \in \mathbb{A}(g,\Phi)$ such that 
	$y_h \rightarrow y$ in $[L^2(\Omega)]^3$,
        $\nabla_h y_h\to\nabla y$ in $[L^2(\Omega)]^{3\times2}$,
        $y_h\to y$ in $[L^2(\partial\Omega)]^3$ and
        $\nabla_h y_h \to \nabla y$ in $[L^2(\partial\Omega)]^{2\times3}$
        as $\hb,\eps\to0$.
        }
\end{Proposition}
\begin{proof}
We proceed in several steps.

\smallskip\noindent
{\it Step 1: Weak convergence in $L^2$.}
We first use \eqref{E:equi-coerc} to deduce that
\rhn{$\|y_h\|_{H_h^2(\Omega)} \le C$}
for all $h$. We next employ the \rhn{
Friedrichs-type} inequality \eqref{e:H2typePoincare}
        to obtain \rhn{$\|y_h\|_\lom \le C$}
	for all $h$. Consequently, there exists a subsequence of $\{y_h\}_{h>0}$
        (not relabeled)	converging weakly in $[L^2(\Omega)]^3$ to some $y \in [\lom]^3$.
        We must show that $y\in\mbV(g,\Phi)$.

\medskip\noindent
{\it Step 2: $H^1$-regularity of $y$ and $L^2$-strong convergence.}
To prove that $y\in[\hom]^3$ we need to regularize $y_h$. To this end, we
employ the smoothing interpolation operator $\Pi_h$ defined in \Cref{S:interpolation}
and let \rhn{$z_h:=\Pi_hy_h\in[\mathring\mbV_h^k]^3 = [\mbV_h^k \cap H^1_0(\Omega)]^3$.}
In view of the stability bound \eqref{stability} and the Friedrichs-type estimate \eqref{e:H2typePoincare}, we find that
\begin{align*}
\| z_h \|_{\lom}+	\| \nabla z_h \|_{\lom}   &\lesssim \| y_h \|_{\lom} + \| \nabla_h y_h \|_{\lom} + \| h^{-1/2} [y_h]\|_{L^2(\Gamma_h^0)} \\
&\lesssim \rhn{\| y_h \|_{H_h^2(\Omega)}} + \| g\|_{H^1(\Omega)}  + \| \Phi \|_{H^1(\Omega)}
\rhn{\le C,}
\end{align*}
\rhn{because of equi-coercivity property \eqref{E:equi-coerc}. Since
$z_h$ is uniformly bounded in $H^1$, (a subsequence of)}
$z_h$ converges weakly to some 
$z \in [H^1(\Omega)]^3$ and strongly to $z$ in $[L^2(\Omega)]^3$ and
\rhn{$[L^2(\partial\Omega)]^3$. Properties \eqref{stability} and
\eqref{E:boundary} of $\Pi_h$ yield
\begin{equation*}
  \| y_h -  \Pi_h y_h\|_{\lom} + \| y_h -  \Pi_h y_h\|_{L^2(\partial\Omega)}\lesssim
  h^{\frac12} \left( \|  \nabla_h y_h \|_{\lom}
  + \| h^{-\frac12}[y_h]\|_{L^2(\Gamma_h^0)}\right)
  \lesssim h^{\frac12},
\end{equation*}
whence as $\hb\to 0$
\begin{gather*}
  \| y_h -  z\|_{\lom} \leq \| y_h -  \Pi_h y_h\|_{\lom} + \|z_h - z\|_{\lom}
  \to 0,
  \\
  \| y_h -  z\|_{L^2(\partial\Omega)} \leq \| y_h -  \Pi_h y_h\|_{L^2(\partial\Omega)}
  + \|z_h - z\|_{L^2(\partial\Omega)}
  \to 0.
\end{gather*}
}
The uniqueness of the weak $\lom$ limit
implies that $y=z\in[H^1(\Omega)]^3$
and that $y_h$ converges to $y$ strongly in $[L^2(\Omega)]^3$.
Regarding the \rhn{trace}, we first observe that
$$
\| h^{-3/2} (y_h-g) \|_{L^2(\partial_D\Omega)} = \|h^{-3/2}[y_h]\|_{L^2(\Gamma_h^b)}
\leq  \rhn{\|y_h\|_{H_h^2(\Omega)} \le C,}
$$
\rhn{according to \eqref{E:energynormN} and \eqref{E:equi-coerc}; hence
$\|y_h- g\|_{L^2(\partial_D \Omega)}\to0$ as $\hb\to0$. Since}
\[
z - g = (z-z_h) +  (\Pi_h y_h - y_h) + (y_h-g),
\]
and $\|\Pi_h y_h - y_h\|_{L^2(\partial_D\Omega)}\to0$ as $\hb\to0$,
in view of \Cref{C:boundary} (boundary error estimate) and \eqref{E:equi-coerc},
we infer that $y=z=g$ on $\partial_D\Omega$.

\medskip\noindent
{\it Step 3: $H^2$-regularity of $y$ and $H^1$-strong convergence.}
\rhn{This entails repeating} Step 2 with $\nabla_h y_h$. We thus define $Z_h:=\rhn{\Pi_h (\nabla_h y_h)} \in
[\mathring\mbV_h^k]^{3\times2}$. Applying again the stability bound \eqref{stability}
in conjunction with \rhn{\eqref{E:energynormN} and} \eqref{E:equi-coerc} gives the uniform bound
\[
\| Z_h \|_{\lom}+\|\nabla Z_h\|_{L^2(\Omega)} \lesssim \| \nabla_h y_h \|_{\lom}+ \|D^2_h y_h\|_{L^2(\Omega)}
+ \|h^{-1/2} [\nabla_h y_h]\|_{L^2(\Gamma_h^0)} \le C.
\]
Hence, \rhn{(a subsequence of)} $Z_h$ converges weakly in $[H^1(\Omega)]^{3\times2}$ and strongly in $[L^2(\Omega)]^{3\times2}$ \rhn{and $[L^2(\partial\Omega)]^{3\times2}$}
to a function $Z \in [H^1(\Omega)]^{3\times2}$. 
Moreover, an argument similar to Step 2, again relying on \Cref{C:boundary}, yields $Z=\Phi$ on $\partial_D\Omega$ and $\|\nabla_h y_h - Z\|_{L^2(\Omega)},
\rhn{\|\nabla_h y_h - Z\|_{L^2(\partial\Omega)}} \to0$ as $\hb\to0$.

It remains to show that $Z=\nabla y$, whence $y\in [H^2(\Omega)]^3$. For any test function $\phi\in [C_0^\infty(\Omega)]^{3\times 2}$, elementwise integration by parts leads to
\[
\int_{\Omega} \nabla_h y_h : \phi = - \int_\Omega y_h \cdot \textrm{div}\phi
+ \int_{\Gamma_h^0}  [y_h]\cdot \phi \, \mu.
\]
We next show that the last term tends to $0$ as $\hb\to0$. \rhn{Note that
\[
\left| \int_{\Gamma_h^0}[y_h]\cdot \phi \, \mu \right| \lesssim  \| h^{-1/2} [y_h]\|_{L^2(\Gamma_h^0)} \ \|h^{1/2} \phi\|_{L^2(\Gamma_h^0)} \to 0 \qquad \textrm{as }\hb\to0,
\]
because $\| h^{-1/2} [y_h]\|_{\LtG} \le \hb \|y_h\|_{H_h^2(\Omega)}\to0$ while $\|h^{1/2} \phi\|_{L^2(\Gamma_h^0)} \lesssim  |\Omega| \|\phi\|_{L^\infty(\Omega)} \leq C$.
This in turn implies that as $\hb\to0$
\[
\int_{\Omega} Z : \phi = - \int_\Omega y \cdot \textrm{div}\phi,
\]
or equivalently $Z=\nabla y$.
}

\medskip\noindent
{\it Step 4: Isometry constraint}. To show that $y$ satisfies \eqref{E:isometry}, we combine the discrete isometry defect \eqref{E:DIC} \rhn{with a bound for the broken Hessian $D_h^2 y_h$. The former controls the isometry constraint mean over elements, whereas the latter controls oscillations of $\nabla_h y_h$.} In fact, we prove
\begin{equation}\label{e:limit_constraint}
	\sum_{ T \in \mathcal{T}_h} \big \|(\nabla_h y_h)^T \nabla_h y_h - I \big\|_{L^1(T)} \lesssim \hb + \eps.
\end{equation}
Let \rhn{$M_T[y_h]\in\mathbb{R}^{2 \times 2}$ denote the {\it mean}} of the isometry constraint over $T\in\mathcal{T}_h$
\[
\rhn{M_T[y_h]} := |T|^{-1} \int_T \Big(\isomh - I \Big),
\]
and write
\[
\big\| \isomh -I \big\|_{L^1(T)} \leq \big\| (\isomh -I ) - \rhn{M_T[y_h]}\big\|_{L^1(T)}
+ \big\|  \rhn{M_T[y_h]} \big\|_{L^1(T)}.
\]
For the second term, \rhn{we use that $y_h\in\Ahk(g,\Phi)$, defined in \eqref{E:Ahk}),
to deduce that}
\[
\rhn{\big\|  M_T[y_h] \big\|_{L^1(T)} = |T| \, |M_T[y_h]|
\quad\Rightarrow\quad
\sum_{T\in\mathcal{T}_h} |T| \, |M_T[y_h]| = D_h[y_h] \le \eps.
}
\]
For the first term, instead, we apply \rhn{the Poincar\'e} inequality in $L^1(T)$ to obtain 
\begin{align*}
  \big\| (\isomh -I) - \rhn{M_T[y_h]} \big\|_{L^1(T)} &\lesssim \rhn{h_T} \ \big\|\nabla_h\big(\isomh \big) \big\|_{L^1(T)}
  \\
  & \lesssim \rhn{h_T} \ \|D_h^2y_h\|_{L^2(T)} \|\nabla_h y_h\|_{L^2(T)}.
\end{align*}
Summing over $T\in\mathcal{T}_h$ and using the Friedrichs-type inequality \eqref{e:H2typePoincare}, together with Lemma~\ref{L:Coercivity} (coercivity of $E_h[y_h]$) and the uniform bound on \rhn{$E_h[y_h]\le\Lambda$,} we get
\begin{align*}
\sumT \big\| (\isomh -I ) - \rhn{M_T[y_h]} \big\|_{L^1(T)} 
&\lesssim \rhn{\hb} \  \rhn{\|D_h^2y_h\|_{L^2(\Omega)} \|\nabla_h y_h\|_{L^2(\Omega)}
\lesssim \hb,}
\end{align*}
\rhn{where the hidden constant depends on $\Lambda$ and data $(g,\Phi,f)$. This
  shows \eqref{e:limit_constraint}.}

With this at hand, we now show that $\|\nabla y ^T \nabla y - I \|_{L^1(\Omega)}=0$.
We observe
\[
\nabla y ^T \nabla y - I = \nabla y ^T (\nabla y - \nabla_h y_h)
+ (\nabla y^T - (\nabla_h y_h) ^T ) \nabla_h y_h + (\nabla_h y_h)^T \nabla_h y_h - I,
\]
which implies
\[
\|\nabla y ^T \nabla y - I \|_{L^1(\Omega)} \lesssim
\Big(\|\nabla y \|_\lom + \|\nabla_h y_h \|_\lom \Big) \|\nabla_h y_h - \nabla y\|_\lom
+ \rhn{\hb} + \eps \to 0
\]
as $\hb, \eps \to0$,
upon recalling the strong convergence of $\nabla_h y_h$ to $\nabla y$ in $[\lom]^{3 \times 2 \times 2}$ from Step 3 and the uniform bound of \rhn{$\|\nabla_hy_h\|_{L^2(\Omega)}$ from \eqref{e:H2typePoincare}.}
%
The proof is thus complete.
\end{proof}

\subsection{Lim-inf property of $E_h$}\label{S:lim-inf}

The lim-inf property easily follows from the preceding results, but we prove it here \rhn{for completeness.} Instead of looking at a general sequence $\{ y_h \}_{h>0}$ with $y_h \to y$ in $[\lom]^3$, we assume $E_h[y_h] \leq \Lambda$ for all $h$ and uniform constant $\Lambda>0$, since otherwise \rhn{$\liminf_{\hb\to0} E_h[y_h]=\infty$ and} the lim-inf inequality is trivial.

\begin{Proposition}[lim-inf property] \label{P:LimInf}
\rhn{Let data $(g,\Phi,f)$ satisfy \eqref{E:data}.}  
Let the penalty parameters $\gamma_0$ and $\gamma_1$ in \eqref{E:DEn}
be chosen sufficiently large, and let the discrete isometry defect 
parameter \rhn{$\eps = \eps(\hb) \to0$ as $\hb\to 0$.}
Let \rhn{$y_h \in \Ahk(g,\Phi)$} and
$E_h[y_h] \leq \Lambda$ for all $h$ with $\Lambda$ independent of $h$. Then,
there exists \rhn{$y \in \mathbb{A}(g,\Phi)$} such that $y_h \to y$ in $[\lom]^3$,
\rhn{$\nabla_h y_h\to\nabla y$ in $[L^2(\Omega)]^{3\times2}$,
$y_h\to y$ in $[L^2(\partial\Omega)]^3$ and
$\nabla_h y_h \to \nabla y$ in $[L^2(\partial\Omega)]^{2\times3}$ as $\hb\to0$} and 
	$$
	E[y] \leq \liminf_{\rhn{\hb \rightarrow 0} } E_h[y_h].
	$$
\end{Proposition}
\begin{proof}
\rhn{Since $y_h \in \Ahk(g,\Phi)$ and $E_h[y_h] \leq \Lambda$, we invoke \protect \MakeUppercase {P}roposition\nobreakspace \ref {P:L2conv} (compactness in $\lom$)
to get $y\in \mathbb{A}(g,\Phi)$, namely $y$ satisfies the Dirichlet
boundary conditions on $\partial_D\Omega$ and the isometry constraint
\eqref{E:isometry}, whereas $y_h$ satisfies
$y_h \rightarrow y$ in $[L^2(\Omega)]^3$,
$\nabla_h y_h\to\nabla y$ in $[L^2(\Omega)]^{3\times2}$,
$y_h\to y$ in $[L^2(\partial\Omega)]^3$ and
$\nabla_h y_h \to \nabla y$ in $[L^2(\partial\Omega)]^{2\times3}$
as $\hb,\eps\to0$.
	Furthermore, the equi-coercivity bound \eqref{E:equi-coerc} in conjunction with Proposition~\ref{p:liminf_H} ($\liminf$ for the reconstructed Hessian) guarantees that 
	$$
	\| D^2 y \|_{L^2(\Omega)} \leq \liminf_{h\to 0} \| H_h[y_h] \|_{L^2(\Omega)}.
	$$
	This, together with $(f,y_h)_\lom \to (f,y)_\lom$ as \rhn{$\hb\to0$}, yields
	$$
	\frac{1}{2} \int_{\Omega} |D^2y|^2  
	- \int_{\Omega} f \cdot y \leq \frac{1}{2} \liminf_{\rhn{\hb \rightarrow 0}} 
	\left(\int_{\Omega} |H_h( y_h)|^2 - \int_{\Omega} f \cdot y_h\right).
	$$
	}
	We next employ Proposition \ref{P:HhIneq} \rhn{(relation between $E_h[y_h]$ and $H_h[y_h]$),} i.e.  
	$$
	\frac{1}{2} \int_{\Omega} |H_h( y_h)|^2 - \int_{\Omega} f \cdot y_h \leq E_h[y_h],
	$$
	and combine the \rhn{last} two inequalities to deduce the asserted estimate.
\end{proof}

\subsection{Lim-sup property of $E_h$}\label{ss:limsup}

Since we are interested in minimizers of $E$ within the admissible set
\rhn{$\mathbb{A}(g,\Phi)$, to show the lim-sup property we construct a recovery sequence
in $[\lom]^3$ for a function $y \in\mathbb{A}(g,\Phi)$} via regularization and
Lagrange interpolation.
Since the isometry and Dirichlet boundary contraints are relaxed
via \eqref{E:DIC} and the Nitsche's approach, we do not need to preserve them
in the regularization and interpolation procedures. 
This extra flexibility is an improvement over \cite{Bartels,BaBoNo},
where the lim-sup property
is proven under the assumption that both procedures preserve those
constraints \rhn{at the nodes} using an intricate approximation argument by P. Hornung
\cite{Hornung}.

\begin{Proposition}[lim-sup property] \label{P:LimSup} \rhn{If the
parameters $\gamma_0$ and $\gamma_1$ in \rhn{\eqref{E:def-ah}} are sufficiently large,
then for any $y \in \mathbb{A}(g,\Phi)$,} there 
exists a recovery sequence $\left\{ y_h\right\}_h \subset \rhn{\Ahk(g,\Phi)}\cap[H^1(\Omega)]^{3\times2}$
with  \rhn{$\|y_h\|_{H_h^2(\Omega)} \lesssim \|y\|_{H^2(\Omega)}$ uniformly in} $h$ such that \looseness=-1
	$$
y_h \to y \quad \text{in} \quad [\lom]^3,
\qquad
\rhn{\nabla y_h \to \nabla y \quad \text{in} \quad [L^2(\Omega)]^{3\times2}
\qquad \text{as } h\to0,}  
	$$
	and
	$$
	\limsup_{\rhn{\hb, \eps \rightarrow 0}} E_h[y_h] \le E[y],
	$$
        provided \rhn{$\hb\le1$ and $\eps:=\eps(h) \ge \eps_1$,} where
\begin{equation}\label{E:delta1}
  \rhn{\eps_1 := C \hb \|y\|_{H^2(\Omega)}^2}
\end{equation}
and $C$ is an interpolation constant that depends on the shape regularity
of $\{\mathcal{T}_h\}_h$.
\end{Proposition}

\begin{proof}
We proceed in several steps.

\smallskip\noindent
\textit{Step 1: Recovery sequence.} Let $\Lambda:=\|y\|_\htwom$ and
let $y_h:=I_h y \in [\mathring\mbV_h^k]^3=[\Vhk \cap H^1(\Omega)]^3$ be the standard Lagrange interpolant of $y$, which is well-defined in $\mathbb{R}^2$ because $H^2(\Omega)\subset C^0(\overline{\Omega})$. \rhn{We first resort to the local estimate \eqref{E:H2-stab}, namely
\begin{equation}\label{E:D2y}
  \|D_h^2 y_h\|_{L^2(T)} \lesssim |y|_{H^2(T)} \quad\forall \, T\in\mathcal{T}_h,
\end{equation}
which is valid for isoparametric maps.
Moreover, $\|[y_h]\|_{L^2(e)} = 0$ for all $e\in\Ei$ because
$y_h\in C^0(\overline\Omega)$, whereas combining a trace
inequality with \eqref{E:BH-T} yields
\[
\|[y_h]\|_{L^2(e)} = \|I_h y-y\|_{L^2(e)}
\lesssim h_e^{3/2} |y|_{H^2(\omega(e))}
\quad\forall \, e\in\Eb,
\]
because $y=g$ on $\partial_D\Omega$ and \eqref{E:bd-jumps}. Similarly, for all
$e\in\Ei$} we have
\[
\big\|[\nabla y_h]\big\|_{L^2(e)} \le \big\|\nabla (I_h y^+ - y)] \big\|_{L^2(e)} +
 \big \|\nabla (I_h y^- - y)] \big\|_{L^2(e)} \lesssim h_e^{1/2} \rhn{|y|_{H^2(\omega(e))},}
\]
\rhn{where $y^\pm=y|_{T^\pm}$ and $T^\pm$ are adjacent elements such that
  $e=T^+\cap T^-$ and $\omega(e)=T^+ \cup T^-$. Moreover, for $e \in \Eb$ we have
  $[\nabla y_h]_e = \nabla (I_h y - y)$ according to \eqref{E:bd-jumps}, because
  $\nabla y = \Phi$ on $\partial_D\Omega$, and we proceed as before.
Therefore, we infer that
\[
\|y_h\|_{H_h^2(\Omega)} \lesssim |y|_{H^2(\Omega)}
  \quad\Rightarrow\quad
E_h[y_h] \lesssim \Lambda^2 + R(g,\Phi,f),
\]
where the last inequality is a consequence of the upper bound in \eqref{E:lower-bound}.
}
  
\medskip\noindent
\textit{Step 2: Discrete isometry defect}.
We claim that $y_h$ satisfies \eqref{E:DIC} for $\eps\ge\eps_1$. Since 
 $$
	\begin{aligned}
	\|\nabla y_h\|_{L^2(T)} &\leq \| \nabla y\|_{L^2(T)} + \| \nabla (y_h-y)\|_{L^2(T)} 
	\leq \|\nabla y\|_{L^2(T)} + C \rhn{h_T  |y|_{H^2(T)},}
	\end{aligned}
$$
	\rhn{in light of \eqref{E:BH-T}, adding and substracting $y$ and recalling
          $\nabla y^T \nabla y = I$ we obtain
        \[
        \nabla y_h^T \nabla y_h - I = \nabla (y_h-y)^T \nabla y_h + \nabla y^T \nabla (y_h-y),
        \]
        whence $D_h[y_h]$ defined in \eqref{E:DIC} satisfies for $\hb\le 1$}	
 \begin{align*}
   D_h[y_h] & \le  \sum_{ T \in \mathcal{T}_h} 
    \Big( \|\nabla y_h\|_{L^2(T)} + \|\nabla y\|_{L^2(T)} \Big) \|\nabla (y-y_h)\|_{L^2(T)}
   \\
&\lesssim \sum_{ T \in \mathcal{T}_h} h_T \|\nabla y\|_{L^2(T)} \, \rhn{|y|_{H^2(T)}} + \sum_{ T \in \mathcal{T}_h} h_T^2 \rhn{|y|_{H^2(T)}^2} \\
   & \leq \rhn{\hb \| \nabla y \|_{\lom} | y |_{H^2(\Omega)} + \hb^2 | y |_{H^2(\Omega)}^2 \leq
   2\Lambda^2 \hb.}
\end{align*}
 Therefore, if \rhn{$\eps_1 = C \hb \Lambda^2 =  C \hb \|y\|_{H^2(\Omega)}^2$,} we
 deduce that $y_h \in \rhn{\Ahk(g,\Phi)}$ for $\eps\ge\eps_1$.

 \medskip\noindent
 \textit{Step 3: Convergence of $E_h[y_h]$}. It remains to show that
 as \rhn{$\hb,\eps\to0$}
	$$
	E_h[y_h]  \to E[y],
	$$
 \rhn{with $\eps\ge\eps_1$.}
 We focus on two critical terms in $E_h[y_h]$ in \eqref{E:DEn}, namely
 \[
 \|D_h^2y_h\|_\lom \to \| D^2 y \|_{\lom},
 \quad
 \| h^{-1/2} \left[ \nabla_h y_h \right] \|_{L^2(\Gamma_h)} \to 0,
 \]
 because similar arguments yield
 $\| h^{-3/2}  \left[ y_h \right] \|_{L^2(\Gamma_h)} \to 0$
 as well as convergence of the remaining terms in $E_h[y_h]$
 \rhn{upon proceeding as in Lemma \ref {L:Coercivity}} (coercivity of $E_h$).
	
 Since $y$ is merely in $[H^2(\Omega)]^3$, we argue by density. Let $\{ \Ye \}_{\sigma>0} \subset [C^\infty(\mathbb R^2)]^3$ be a sequence of regularizations of $y$ such that $\Ye \to y$ in $[H^2(\Omega)]^3$ as $\sigma \to 0$, and let 
 $\rhn{\Ye_h = I_h \Ye} \in [\mathring\mbV_h^k]^3$ be the Lagrange interpolant of $\Ye$.
 \rhn{Estimate \eqref{E:D2y} implies
\[
\| D^2( y_h -  \Ye_h )\|_{L^2(T)} \lesssim |y-\Ye|_{H^2(T)},
\]
whence writing $y-y_h=(y-\Ye) - I_h(y-\Ye) + (\Ye - \Ye_h)$ gives
\begin{align*}
  h_T^{-1} \|\nabla(y-y_h)\|_{L^2(T)} + \|D^2(y-y_h)\|_{L^2(T)}
  \lesssim  | y -\Ye |_{H^2(T)} +  h_T  \|\Ye \|_{H^3(T)},
\end{align*}
because the polynomial degree is $k\ge2$.
This can be made arbitrarily small upon choosing first the coarse scale
$\sigma$ and next $\hb$, and 
implies $\|D^2_h y_h\|_{L^2(\Omega)} \to \|D^2 y\|_{L^2(\Omega)}$ as $\hb\to0$
by the triangle inequality.}
	
We finally examine the stabilization term $\| h^{-1/2} \left[ \nabla_h y_h \right] \|_{L^2(\Gamma_h)}$. We recall that $[\nabla y]=0$ for all interior edges $e\in\mathcal{E}_h^0$ and
$\nabla y = \Phi$ for boundary edges $e\in\mathcal{E}_h^b$. We next utilize the scaled trace inequality \rhn{in conjunction with \eqref{E:bd-jumps} to write
\begin{align*}
  \big\| h^{-1/2} [\nabla_h y_h] \big\|_{L^2(e)} &= \big\| h^{-1/2} [\nabla_h (y - y_h)] \big\|_{L^2(e)}
  \\
& \lesssim h_e^{-1} \|\nabla_h(y-y_h)\|_{L^2(\omega(e))} + \| D^2_h (y - y_h) \|_{L^2(\omega(e))},
\end{align*}
whence $\| h^{-1/2} \left[ \nabla_h y_h \right] \|_{L^2(\Gamma_h)}\to0$ as $\hb\to0$. This concludes the proof.}
\end{proof}

\rhn{We point out the occurrence of the $H^2$-seminorm on the right-hand side of \eqref{E:D2y} for both affine equivalent and isoparametric elements (including subdivisions made of quadrilaterals). We present a proof in Section \ref{S:A-quads} and refer to \cite{CiarletRaviart} and Chapter 13 of \cite{ErnGNew} for a comprehensive discussion of isoparametric elements.}

\subsection{Convergence of global minimizers}\label{S:G-convergence}

 We now show that cluster points of global minimizers of $E_h$ are global minimizers of $E$, without assuming the existence of the latter. 
 The proof combines \rhn{Propositions \ref{P:L2conv} (compactness in $\lom$),
   \ref{P:LimInf} (lim-inf property) and \ref{P:LimSup} (lim-sup property).
 
We now collect properties of the {\it nonconvex} discrete admissible set
$\Ahk(g,\Phi)$ of \eqref{E:Ahk}, the discrete counterpart of the nonconvex
 set $\mathbb{A}(g,\Phi)$ of
\eqref{E:ContSpace}. Given an initial guess $y_h^0\in\Vhk(g,\Phi)$ with
isometry defect $D_h[y_h^0]\le\tau$ and energy $E_h[y_h^0]\le\Lambda_0$
independently of $h$, Lemma \ref{L:GFlowDIC} (discrete isometry defect) guarantees
that the discrete $H^2$-gradient flow
produces a sequence $\{y_h^n\}_n\subset\Ahk(g,\Phi)$ with $\eps\ge\eps_0$
and
 \[
\eps_0 = \Big(1 + C_F^2 \big( \Lambda_0 + C_F^* R(g,\Phi,f) \big) \Big) \tau.
\]
This shows that $\Ahk(g,\Phi)$ is non-empty provided $\eps\ge\eps_0$ and that
$\eps_0$ can be made arbitrarily small as $\tau\to0$.
We could take $\tau$ proportional to $\hb$, which is the choice in
Section \ref{S:NumExPl}. Moreover, Lemma \ref{L:energydecrease} (energy decay)
shows that $E_h[y_h]\le \Lambda_0$ whereas Lemma \ref{L:Coercivity} (coercivity of $E_h$) gives
\[
\|y_h\|_{H_h^2(\Omega)}^2 \le \Lambda_0 + C_F^* R(g,\Phi,f).
\]
This is the setting for $\Gamma$-convergence and is discussed next.}

\begin{Theorem}[convergence of global minimizers] \label{T:Gcon}
Let data $(g,\Phi,f)$ satisfy \eqref{E:data}.
Let the penalty parameters $\gamma_0$ and $\gamma_1$ in \rhn{\eqref{E:def-ah}}
be chosen sufficiently large. Let \rhn{$y_h\in\Ahk(g,\Phi)$} be a sequence of
functions such that $E_h[y_h] \leq \Lambda_0$, for a constant
$\Lambda_0$ independent of $h$, \rhn{and let} the discrete isometry
defect parameter $\eps$ satisfy
\[
\eps \ge \eps_1 := C \big(\Lambda_0+R(g,\Phi,f) \big) \hb,
\]
where $C$ is an interpolation
constant depending on the shape regularity of $\{\mathcal{T}_h\}$.
If \rhn{$y_h \in \Ahk(g,\Phi)$} is an almost global minimizer of $E_h$
\rhn{in the sense that}
$$
	E_h[y_h] \leq \inf_{\rhn{w_h \in \mathbb{A}_{h,\eps}^k(g,\Phi)}} E_h[w_h] + \sigma,
$$
where $\sigma, \eps \rightarrow 0$ as $\hb \rightarrow 0$,
then $\left\{ y_h \right\}_h$ is precompact in $[\lom]^3$ 
and every cluster point $y$ of $y_h$ belongs to \rhn{$\mathbb{A}(g,\Phi)$}
and is a global minimizer of $E$, namely 
$
	E[y] = \inf_{w \in \rhn{\mathbb A(g,\Phi)}}E[w].
$
Moreover, up to a subsequence (not relabeled), the energies converge
$$
	\lim_{\hb \rightarrow 0}E_h[y_h]=E[y].
$$
\end{Theorem}
\begin{proof}
This proof is standard and given for completeness.  
Since  \rhn{$E_h[y_h]\le \Lambda_0$, we invoke Proposition \ref{P:LimInf}} (lim-inf property) to deduce the existence of 
	\rhn{$y \in \mathbb A(g,\Phi)$} such that a (non-relabeled) subsequence of $y_h \to y$ in $[\lom]^3$ and
	$$
	E[y] \leq \liminf_{\rhn{\hb \rightarrow 0} } E_h[y_h].
	$$
	This means that \rhn{$\mathbb{A}(g,\Phi)$} is non-empty and that 
	$$
	\rhn{\inf_{z \in \mathbb{A}(g,\Phi)} E[z] \leq E[y] \leq \Lambda_0.}
	$$
	To show that $y$ is a global minimizer of $E$, we let $0<\eta <1$ and
        \rhn{$w \in \mathbb{A}(g,\Phi)$ satisfy 
	$$
	E[w] \leq  \inf_{z \in \mathbb{A}(g,\Phi)} E[z] + \eta \le \eta+ \Lambda_0.
	$$
	Lemma \ref{L:ECoercivity} (coercivity of $E$) yields
        $\|w\|_\htwom^2 \lesssim \eta+\Lambda_0+R(g,\Phi,f)=:\Lambda_1$
        whence, if $\eps\ge\eps_1=C\hb\Lambda_1$,
        Proposition \ref{P:LimSup} (lim-sup property) gives a recovery
	sequence $w_h \in \mathbb{A}_{h,\eps}^k(g,\Phi)$ of $w$ such that
	$w_h \to w$ in $[\lom]^3$ as $\hb,\eps\to0$ and}
	$$
	\limsup_{\rhn{\hb,\eps\to0}} E_h[w_h] \leq E[w].
	$$
        We next utilize that $E_h[y_h]\le E_h[w_h]+\sigma$, because $y_h$ is an
        almost global minimizer of $E_h$ within \rhn{$\Ahk(g,\Phi)$, to deduce
	$$ 
	\begin{aligned}
	E[y] \leq \liminf_{\hb \rightarrow 0 } E_h[y_h] 
	\leq \limsup_{\hb \rightarrow 0} \, \big (E_h[w_h]  + \sigma \big)
	\leq   E[w]
	\leq  \inf_{z \in \mathbb{A}(g,\Phi)} E[z] + \eta.
	\end{aligned}
	$$
        }
	Taking $\eta \to 0$ implies that $y$ is an global minimizer of $E$ and
	$$
	\lim_{\hb \rightarrow 0}E_h[y_h]=E[y].
	$$
	This concludes the proof.
\end{proof}

\subsection{Strong Convergence of $H_h[y_h]$ and Scaled Jumps}\label{S:strong-conv}
We now exploit Theorem \ref{T:Gcon} (convergence of global minimizers) to
strengthen the \rhn{assertions} of \Cref{P:L2conv} (compactness in
$L^2(\Omega)$), in the
  spirit of \cite{DiPietroErn,Pryer}.
In fact, we show strong convergence of the scaled jump terms to zero and
of the discrete Hessian $H_h[y_h]$ \rhn{as well as broken Hessian $D^2_h y_h$}
of $y_h$ to $D^2 y$ as $\hb\to0$.

\begin{Corollary}[\textrm{strong convergence of Hessian and scaled jumps}] \label{C:StrongConv} 
Let $(g,\Phi,f)$ satisfy \eqref{E:data} and the penalty parameters $\gamma_0$ and $\gamma_1$ in \rhn{\eqref{E:def-ah} be chosen sufficiently large so that \eqref{E:sign-J} holds. Let $y_h\in\Ahk(g,\Phi)$} be a subsequence of almost global minimizers of $E_h$
converging to a global minimizer \rhn{$y\in\mathbb{A}(g,\Phi)$} of $E$, as established in
\Cref{T:Gcon}. Then, the
following statements are valid as $\hb\to0$
	\begin{enumerate}[(i)]
		\item $H_h[y_h] \to D^2y$ strongly in $[\lom]^{3\times 2 \times 2}$;
		\item $\| h^{-1/2} [ \nabla_h y_h ]\|_{L^2(\Gamma_h)} + \| h^{-3/2} [  y_h ]\|_{L^2(\Gamma_h)} \to 0$;
		\item $D_h^2y_h \to D^2y$ strongly in $[\lom]^{3\times 2 \times 2}$.
	\end{enumerate}
\end{Corollary}
\begin{proof} We first observe that
Theorem \ref{T:Gcon} (convergence of global minimizers) yields
\[
E_h[y_h] \rightarrow E[y] =  \frac{1}{2} \int_{\Omega} |D^2 y|^2 - \int_{\Omega} f\cdot y \qquad \textrm{as }\hb\to 0.
\]
We apply Propositions \ref{P:HhIneq} (relation between $E_h[y_h]$ and \rhn{$H_h[y_h]$})
to deduce
\begin{align*}
	\frac{1}{2} \limsup_{\rhn{\hb\rightarrow 0}} \|H_h[y_h]\|_\lom^2
	&\leq \limsup_{\rhn{\hb \rightarrow 0}}\left( E_h[y_h]+\int_{\Omega} f \cdot y_h \right)
        = \frac{1}{2} \|D^2 y\|_\lom^2.
\end{align*}
\rhn{We now resort to Proposition~\ref{p:liminf_H} (lim-inf property of the reconstructed Hessian). In fact, combining \eqref{e:liminf_H} with the
previous inequality gives the convergence in norm
$$
		\|H_h[y_h]\|_\lom \rightarrow \|D^2y\|_\lom.
$$
In addition, Proposition~\ref{p:liminf_H} guarantees that $H_h[y_h] \rightharpoonup H$ in $\lbrack L^2(\Omega)\rbrack^{3\times 2\times 2}$ where $H$ satisfies the orthogonality property \eqref{e:ortho_H}. This leads to 
$$
 \big( H_h[y_h], D^2 y \big)_{L^2(\Omega)} \rightarrow  \big( H, D^2 y \big)_{L^2(\Omega)} = \| D^2 y \|_{L^2(\Omega)}^2, \quad \textrm{as }h \to 0.
$$
The above property and the convergence of norms imply strong convergence, \rhn{namely
\[
\big\|H_h[y_h] - D^2 y \big\|_{L^2(\Omega)}^2 = \big\|H_h[y_h]\big\|_{L^2(\Omega)}^2
+ \| D^2 y\|_{L^2(\Omega)}^2 - 2 \big( H_h[y_h], D^2 y \big)_{L^2(\Omega)}\to0.
\] 
} 
This in turn yields (i).} To prove (ii) we make use of (i) to infer that as \rhn{$\hb\to0$}
$$
 	J_h[y_h] = E_h[y_h] - \frac{1}{2} \int_\Omega |H_h[y_h]|^2
	+ \int_{\Omega} f \cdot y_h \, \rightarrow \,
        E[y] -\frac{1}{2} \int_\Omega |D^2y|^2 + \int_{\Omega} f \cdot y =0,
$$    
\rhn{where we have used the definition \eqref{e:for_later} of $J_h[y_h]$.   
In view of \eqref{E:sign-J}, we thus get} 
$$ 
\|h^{-3/2} [y_h]\|_{L^2(\Gamma_h)} +  \|h^{-1/2} [\nabla_h y_h]\|_{L^2(\Gamma_h)} \to 0 \qquad \textrm{as } \rhn{\hb \to 0.}
$$
This not only establishes (ii), but combined with (i), the definition
\eqref{E:Hh} of $H_h[y_h]$, and the bounds \eqref{E:boundsRB} for
\rhn{$R_h([\nabla_h y_h])$ and $B_h([y_h])$,}
directly implies (iii).
\end{proof}

\section{Numerical Experiments} \label{S:NumExPl}

In this section we explore and compare the performance of our method with that
of \rhn{Kirchhoff elements \cite{Bartels,BartelsBook,BaBoMuNo,BaBoNo}.} We are interested in the speed and accuracy of the method,
as well as its ability to capture the \rhn{essential physics and geometry}
of the problems appropriately.
We observe that our dG approach seems to be more 
flexible with comparable or better speed. \rhn{We present below} specific examples
computed within the platform deal.ii with polynomial degree $k=2$
\rhn{and uniform meshes $\mathcal{T}_h$ made of rectangles}
\cite{dealii85,dealiipaper};
hence we use $\mathbb{Q}_2(T)$ for all $T\in\mathcal{T}_h$.
Moreover, one might notice that we consistently use
rather large penalty parameters $\gamma_0, \gamma_1$
relative to the second-order case \rhn{\cite{dealii85,ABCM,dealiipaper,BoNo,Brezzi,Riv}.}
\rhn{These choices hinge mostly on experiments performed for the vertical load
example of Subsection \ref{ss:verticalLoad}
and are not dictated by stability considerations exclusively. In fact, they are a compromise between the discrete initial energy $E_h[y_h^0]$ and the fictitious time step $\tau$ of the gradient flow, which obey the relations \eqref{E:approx-isometry} and \eqref{E:delta0} and control the discrete isometry defect $D_h[y_h]$ according to \eqref{e:Dh}. We point out that enforcing Dirichlet conditions via the Nitsche's approach depends
on the magnitude of $\gamma_0$ and $\gamma_1$, which in turn affects $E_h[y_h^0]$. We examined a very wide range of
$\gamma_0$ and $\gamma_1$ and compared $D_h[y_h]$ with}
the rate at which it decreases for each
tested pair. We selected those 
values that lead to the smallest defect and the largest rate of convergence.
We made similar choices for all subsequent examples without exhaustive testing
for each of them. We note that our theory does not explicitly predict 
why the best convergence behavior manifests for such large values
\rhn{of $\gamma_0$ and $\gamma_1$.}
We do not provide our computational study leading to $\gamma_0$ and $\gamma_1$,
for the sake of brevity.

\subsection{Vertical load on a square domain} \label{ss:verticalLoad}

This is a simple example of bending due to a vertical load. We use the same configuration 
as in \cite{Bartels} in order to \rhn{present a faithful} comparison of the two methods. We deal with a square domain with two of the non-parallel sides being clamped.
\begin{example} \label{Eg:vertical_load} Let $\Omega = (0,4) \times (0,4)$ and $\partial_D \Omega = \left\{ 0 \right\} \times [0,4] 	\cup [0,4] \times \left\{ 0 \right\}$ be the part of the boundary where we enforce the \rhn{Dirichlet} boundary conditions
  \begin{equation}\label{E:clamped}
	g = (x_1,x_2,0) \quad \text{and} \quad \Phi = [I_2,0]^T.
  \end{equation}
	We apply a vertical force \rhn{$f=(0,0,F)$ of magnitude $F=2.5 \times 10^{-2}$.}
\end{example}
\FloatBarrier
\begin{table}[htb]
	\begin{center}
		\begin{tabular}[t]{ | c | c | c | c | c | c |}
			\hline
			No. Cells & DoFs   & \rhn{$\dt=\hb$} & $E_h[y_h]$  &   $D_h[y_h]$ & Iterations \\
			\hline\hline
			256 & 7680 & $\sqrt{2}~2^{-2}$ &	  -7.53e-3  &  	4.02e-3 	& 13 \\ 
			\hline
			1024 & 30,720  & $\sqrt{2}~2^{-3}$ &	  -5.76e-3 	&  	1.63e-3 	& 28  \\ 
			\hline 
			4096 & 122,880  & $\sqrt{2}~2^{-4}$ &	 -4.26e-3 	&  	6.07e-4		& 76 \\ 
			\hline 
			16384 & 491,520   & $\sqrt{2}~2^{-5}$&	  -3.30e-3	&  2.28e-4		& 140\\ 
			\hline 
		\end{tabular}
	\end{center}
	\FloatBarrier
	\vskip0.2cm
	\caption[Single-Layer: Vertical Load Energies]{\small \rhn{Example \ref{Eg:vertical_load}:} Number of cells, degrees of freedom, discrete energy $E_h[y_h]$, isometry defect $D_h[y_h]$ and number of gradient flow iterations for \rhn{the square plate $\Omega=[0,4]^2$, clamped in two adjacent} sides with vertical forcing. We observe super-linear rates for the isometry defect, while theory predicts linear rates for the case \rhn{$\tau = \hb$.} }\label{Ta:DG_convergence}
\end{table}
\FloatBarrier

\FloatBarrier
\begin{figure}[htb]
	\begin{center}
		\includegraphics[scale=0.3]{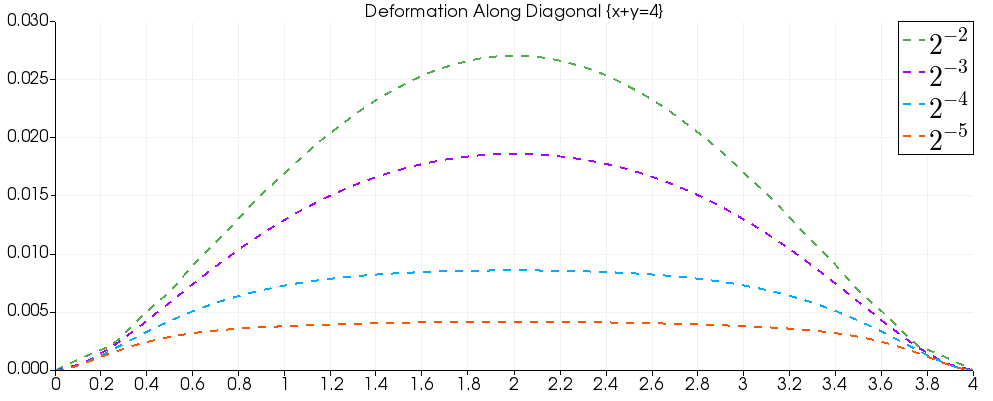}
	\end{center}
		\caption[Single-Layer: Diagonal Deformation - dG]{\small \rhn{Example \ref{Eg:vertical_load}:} Deformation along the diagonal $x_1+x_2=4$. We observe smaller deformation $y$ than with the Kirchhoff elements, up to one order of magnitude. For example, \rhn{$\hb=2^{-3}$} yields $|y| \approx 0.018$ for dG, while \rhn{$\hb=2^{-6}$} gives $|y| \approx0.02$ for Kirchhoff elements.} \label{F:DGDiag} 
\end{figure} 
\FloatBarrier

We first illustrate the convergence of the energy $E_h[y_h]$ and the isometry defect $D_h[y_h]$ as \rhn{$\hb$} decreases. We set $\gamma_0 = 5,000$, $\gamma_1 = 1100$ and choose \rhn{$\dt =\hb$} and observe that the isometry defect decays super-linearly with $\dt$, which is better than the linear convergence predicted in Lemma \ref{L:GFlowDIC} \rhn{(discrete isometry defect).} Compared to \cite{Bartels}, we obtain a more clear rate, while the defect \rhn{itself is up to one order of magnitude smaller.} The number of gradient flow iterations is similar for both methods.

We now explore the geometric behavior of our method. More precisely, it was observed in \cite{Bartels} that there was an artificial displacement along the diagonal $x_1+x_2=4$, which does not correspond to the actual physics of the problem, \rhn{namely} $y=0$ for $x_1+x_2 \leq 4$. This displacement decreases with smaller mesh size $\hb$. Our method introduces the same artificial deformation. However, we can see in \rhn{Figure \ref{F:DGDiag}} that \rhn{this displacement is smaller now, up to one order of magnitude for the last refinement.}

\subsection{Obstacle Problem}
\rhn{In Example \ref{Eg:vertical_load} we document the flexibility of
the dG approach to capture the correct bending} behavior of plates. To explore this 
further, we introduce an extra element to the deformation: \rhn{a rigid obstacle.} We use a square plate clamped on one side and we exert a vertical force. We require that the plate does not penetrate the obstacle. This example is motivated by \cite{BaBoMuNo}, where the deformation of the plate is the result of thermal actuation of bilayer hinges connected to the plate.

\begin{example} \label{Eg:obstacle} Let $\Omega = (-1,1) \times (-1,1)$ \rhn{and} $\partial_D \Omega = \left\{ -1 \right\} \times [-1,1]$ be the \rhn{Dirichlet boundary where we enforce clamped boundary conditions \eqref{E:clamped}.} We apply the vertical force $f=(0,0,1)$. We choose a simple case where the obstacle is a rigid flat 
plate at height $z = 0.2$. From a mathematical viewpoint, this obstacle problem can be 
treated by introducing the convex set of \rhn{admissible deformations}
$$
K = \left\{ y \in [\lom]^3 \ : \quad y_3 \leq 0.2	 \right\},
$$
along with splitting of variables. We introduce another deformation $s$, $ s \approx y$ such that $s \in K$ always and penalize the $\lom$-distance between $y$ and $s$. If $\sigma$ is an obstacle penalty parameter, we add the following extra term to the energy $E[y]$
$$
\frac{1}{\sigma} \|y-s\|_\lom^2.
$$

At the discrete level, this affects the gradient flow at each step, where we use as $s_h$ the $L^2$-projection of the previous solution $y_h^n$ in $K$. We refer to \cite{BaBoMuNo} for more details about the variable 
splitting and the projection. 

\rhn{The plate configurations before contact are unaffected by the obstacle presence
  and dictated
  by pure bending. Contact occurs with minor penetration of the obstacle because we do
  not force the deformation $y_h$ to belong to the set $K$, but rather penalize its
  $L^2$-distance from its projection to $K$. As time evolves, the plate bends backwards
  trying to decrease the obstacle crossing because it is energetically costly.
  This behavior can be enhanced by further mesh refinement and by choosing a smaller} obstacle parameter $\sigma$. To illustrate the \rhn{obstacle-plate interaction,} we provide six different steps of the gradient flow in Figure \ref{F:Obstacle_time_steps}. The illustration corresponds to 1024 cells, $\sigma = 3 \times 10^{-4}$ and time step $\dt = 5 \times 10^{-4}$. \rhn{The penalty parameters are  $\gamma_0=\gamma_1=5000$, and the} isometry defect at the end of this simulation is $9 \times 10^{-5}$.
\end{example}

\FloatBarrier
\begin{figure}
	\begin{center}
		\subfloat[t=500ts]{\includegraphics[width = 2.4in]{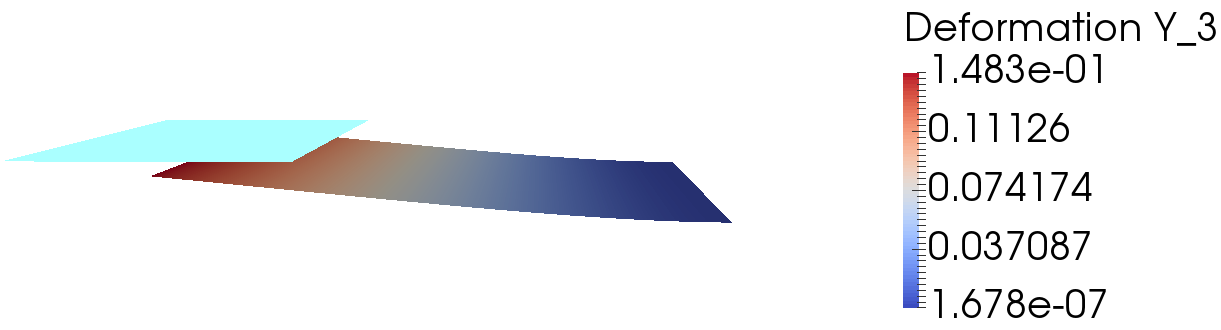}} \quad
		\subfloat[t=700ts]{\includegraphics[width = 2.4in]{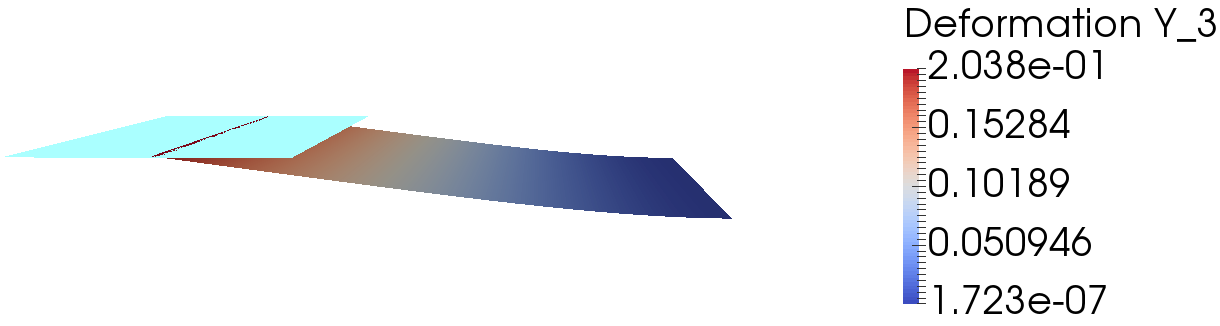}}\\
		\subfloat[t=1300ts]{\includegraphics[width = 2.4in]{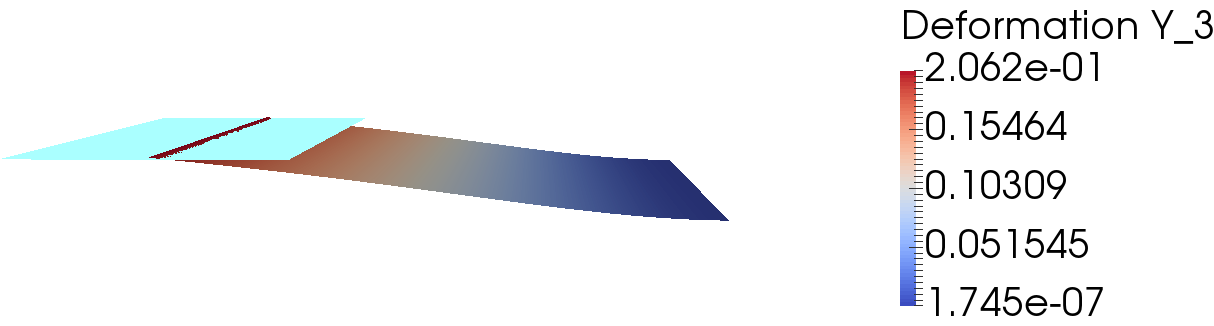}}\quad
		\subfloat[t=3000ts]{\includegraphics[width = 2.4in]{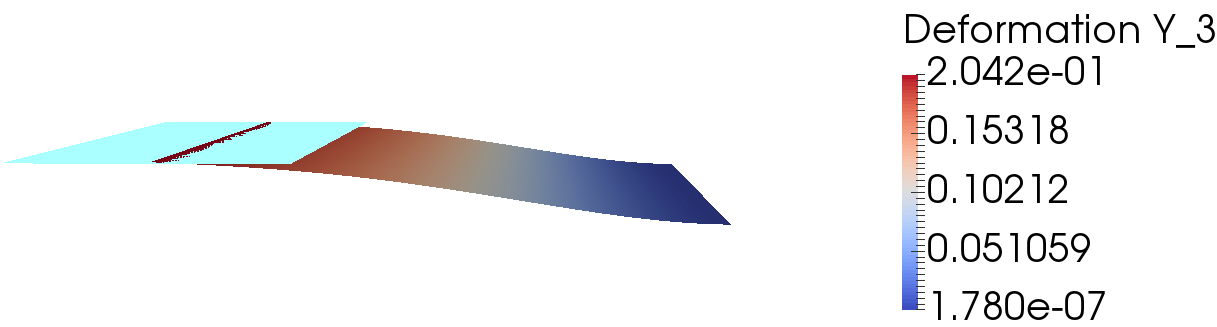}} \\
		\subfloat[t=5000ts]{\includegraphics[width = 2.4in]{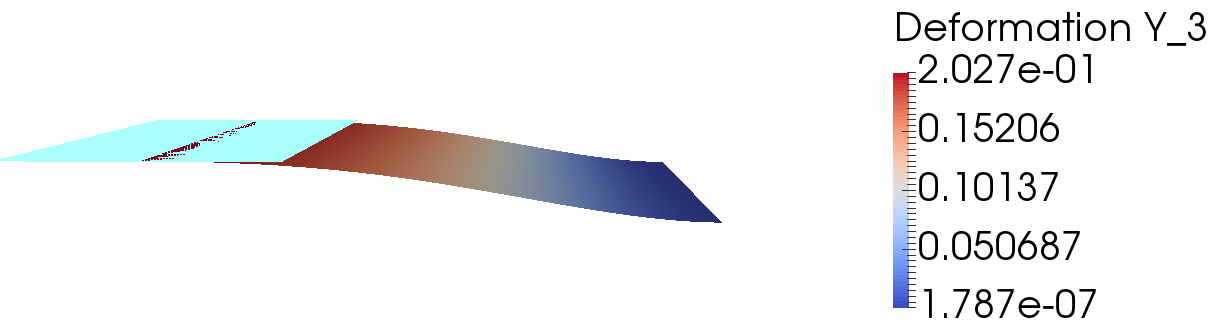}} \quad
		\subfloat[t=13382ts]{\includegraphics[width = 2.4in]{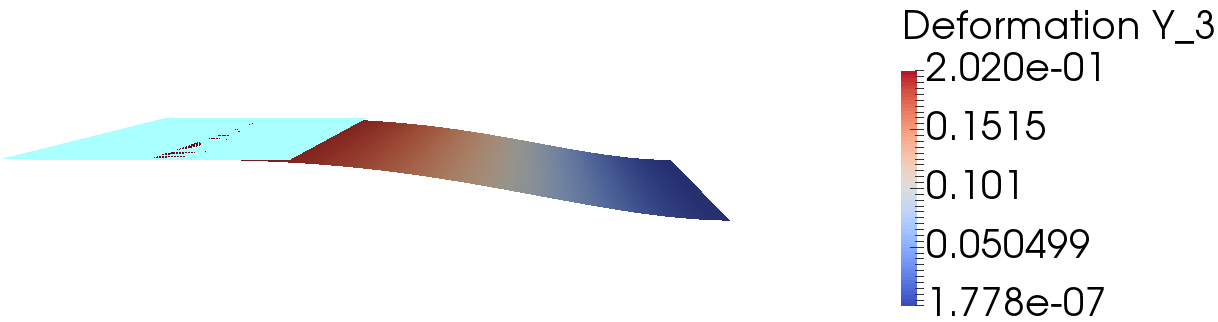}} \quad
	\end{center}
		\caption[Single-Layer: Obstacle Deformation]{\small \rhn{Example \ref{Eg:obstacle} (deformation of the plate with obstacle dictated only by forcing): (a) Configuration before contact,} (b) First contact with the obstacle, (c) Penetration of the obstacle, (d) Plate starts adjusting shape to decrease crossing, (e) Stronger bending \rhn{with smaller crossing, (f) Final stage with obtacle crossing of about 0.002.}}
			\label{F:Obstacle_time_steps} 
\end{figure}
\FloatBarrier

\subsection{Compressive Case: \rhn{Buckling}}
\rhn{We finally investigate the geometrically and physically interesting case of buckling.} We use a rectangular plate with a small vertical load, \rhn{that induces a bias for bending,} and we impose compressive boundary conditions on two opposite sides. \rhn{The latter reveals the delicate interplay between the Nitsche's approach for enforcement of Dirichlet boundary conditions and the choice of suitable initial configurations, for the performance of the gradient flow.} This example is also motivated by \cite{Bartels}.

\begin{example} \label{Eg:compression}
	Let $\Omega = (-2,2) \times (0,1)$ \rhn{and} $\partial_D \Omega = \left\{ -2,2 \right\} \times [0,1]$
	be the two sides where we impose the compressive boundary conditions
	$$
	g = (x_1 \pm 1.4,x_2,0) \quad \text{and} \quad \Phi = [I_2,0]^T.
	$$
	We apply a vertical force $f$ of magnitude $10^{-2}$.
	Given the compressive nature of the problem, the plate could bend either upwards or downwards (buckling), resulting in two deformations which are the reflection of each other with respect to the $x_1-x_2$ plane and have the same \rhn{minimal} energy. This is why we apply a small force \rhn{upwards to select the deformation with positive third component. Since our initial configuration is flat, there is a boundary mismatch that leads to large boundary penalty terms in the Nitsche's approach and correspondingly large initial energy $E_h[y_h^0]$. This in turn gives rise to either large isometry defects, because of \eqref{E:delta0} and \eqref{e:Dh}, or tiny time steps $\tau$ and very slow evolution. This might explain why we need a stronger force than the one in \cite{Bartels} (i.e. $10^{-5}$) because $E_h[y_h^0]$ is dominated by the boundary terms.
Moreover, to prevent a large $E_h[y_h^0]$ from creating very abrupt and non-physical deformations during the gradient flow for moderate $\tau$, we employ a quasi-static approach: we enforce the boundary conditions gradually thereby avoiding a large mismatch (parameter continuation).} We use a parameter $\alpha$ that starts from zero and increases throughout the flow until it reaches the value 1. In order to achieve a gradual adjustment to the boundary conditions, we let
	$$
	\varphi(\alpha) := (1- \alpha) \ id + \alpha \ g, \qquad 0 \le \alpha \le 1,
	$$
        and $\Phi(\alpha)=\Phi$ be the Dirichlet boundary conditions for $y$ and $\nabla y$, where $id$ stands for the identity function. \rhn{As the mesh $\mathcal{T}_h$ becomes finer, the growth of parameter $\alpha$ must be slower to compensate for the larger initial energy associated with mesh-dependent boundary terms as well as to allow for smooth flow evolutions.}

\FloatBarrier
\begin{figure}[htb]
	\begin{center}
		\subfloat[$\alpha=0.05$]{\includegraphics[width = 2.1in]{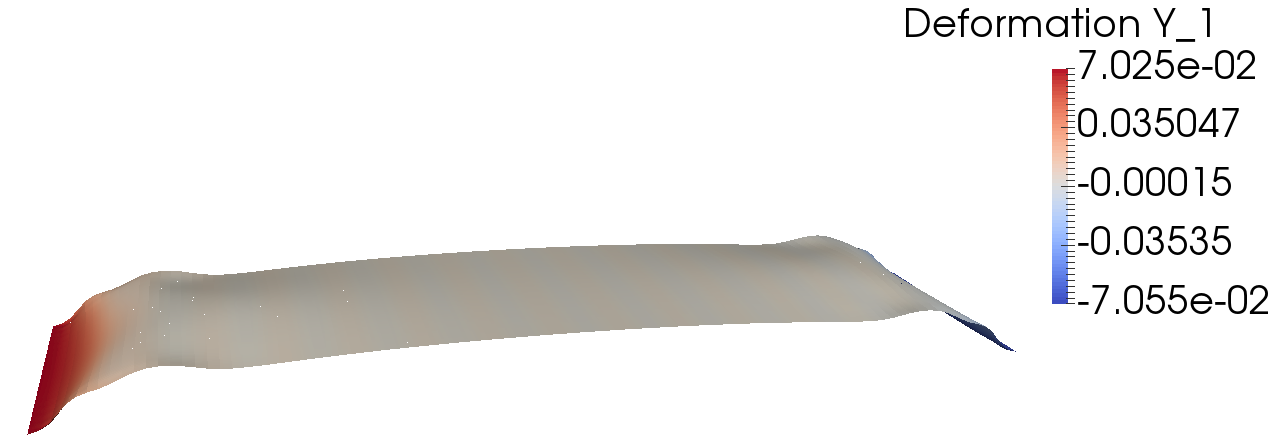}} \quad
		\subfloat[$\alpha=0.1$]{\includegraphics[width = 2.1in]{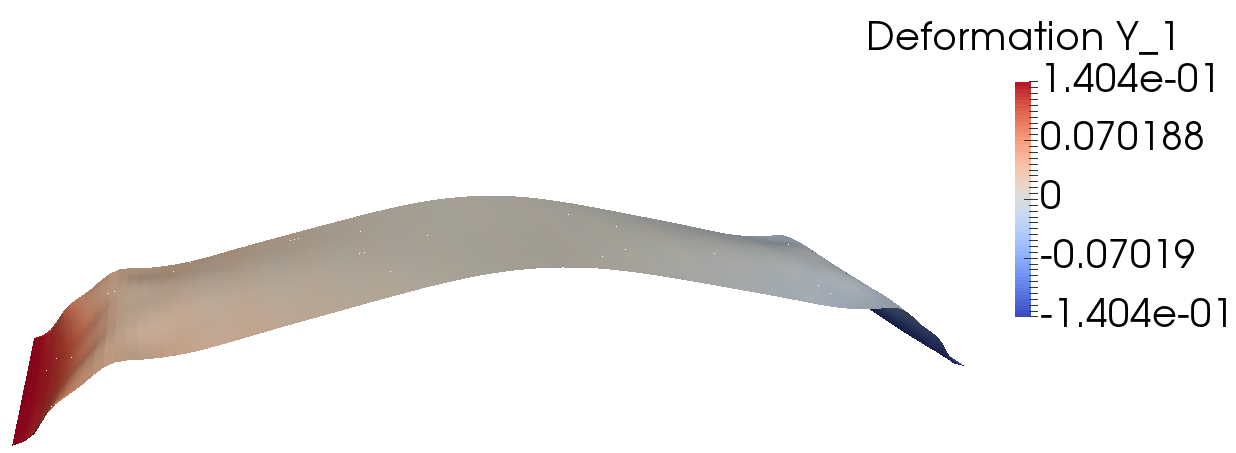}}\\
		\subfloat[$\alpha=0.25$]{\includegraphics[width = 2.1in]{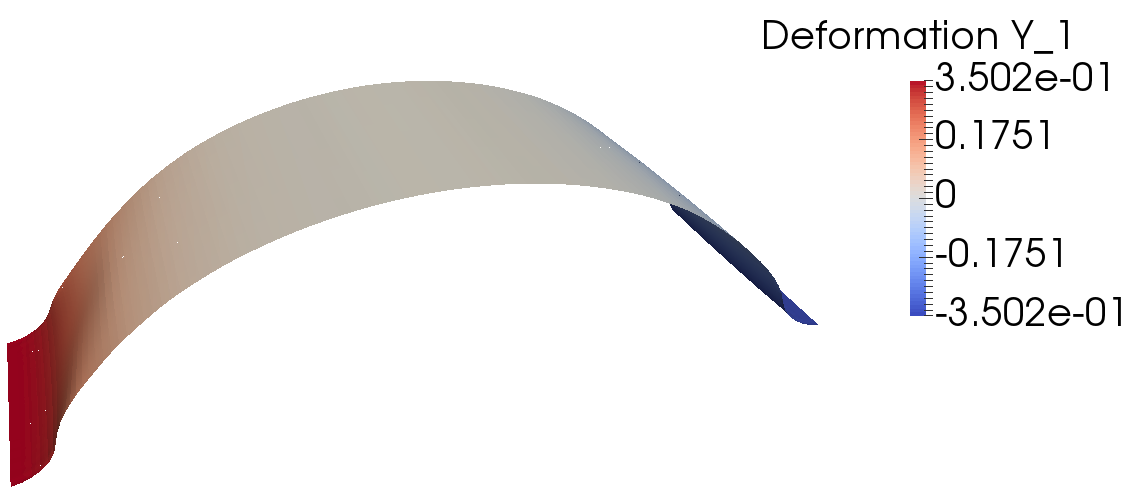}} \quad 
		\subfloat[$\alpha=0.6$]{\includegraphics[width = 2.1in]{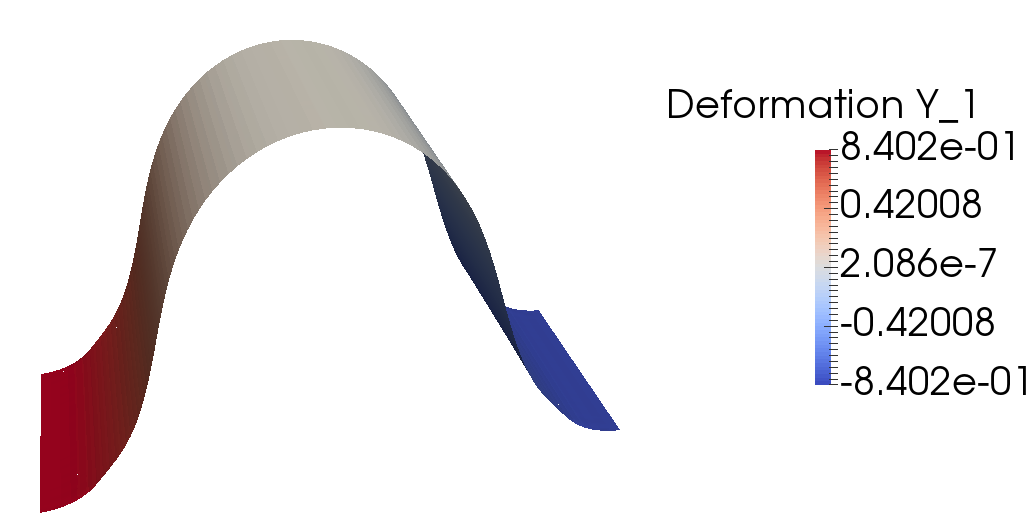}}\\
		\subfloat[$\alpha=0.65$]{\includegraphics[width = 2.1in]{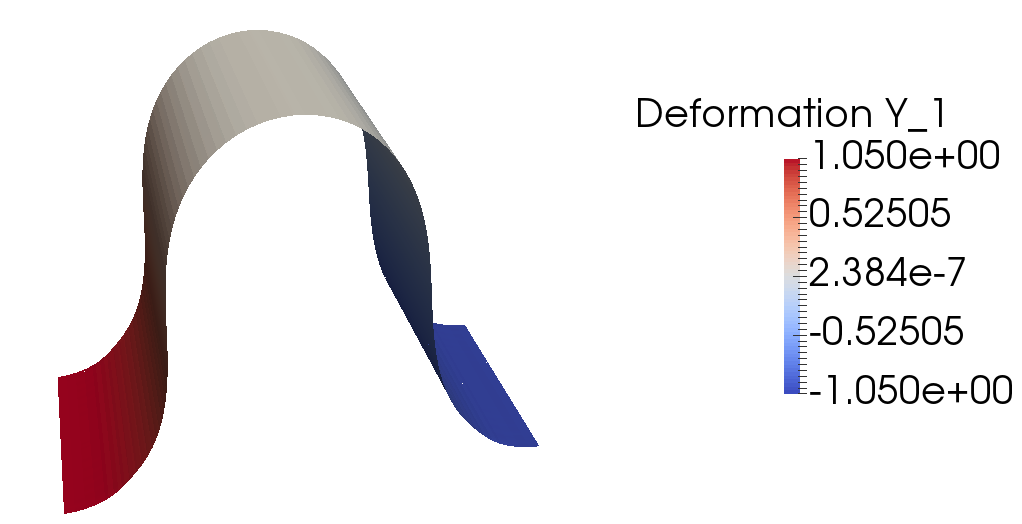}} \quad 
		\subfloat[$\alpha=1$]{\includegraphics[width = 2.1in]{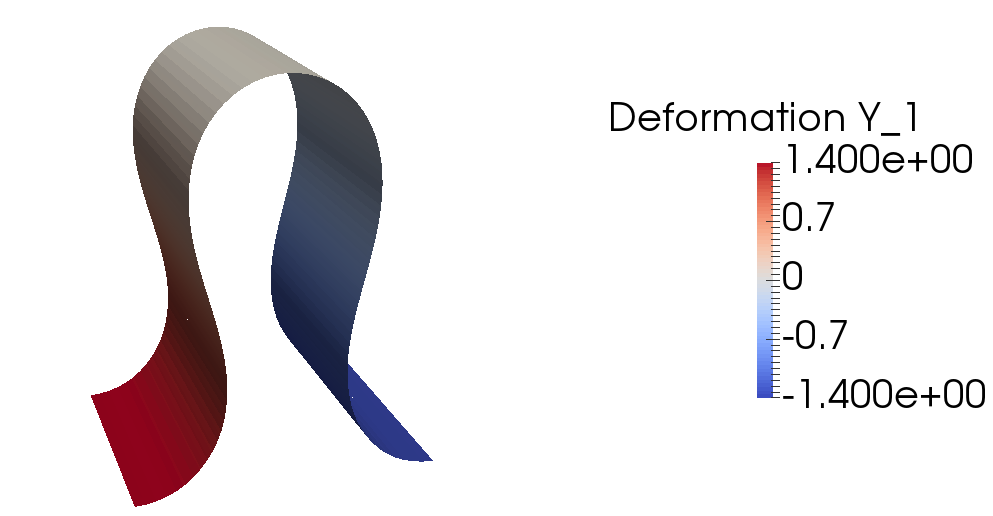}} 
	\end{center}
	\caption{\small \rhn{Example \ref{Eg:compression} (buckling of a rectangular plate at six different stages):} (a)-(c) Initial stages \rhn{with small deformations dictated} mostly by the forcing, (d)-(e) The compressive boundary conditions \rhn{dominate the evolution and induce large bending deformations, (f) Final deformation with attained compressive boundary condition.}}
	\label{F:Compression} 
\end{figure}
\FloatBarrier

        To illustrate the effect of $\alpha$, we depict \rhn{deformations in Figure \ref{F:Compression} that} corresponds to various stages of the gradient flow with $\gamma_0=\gamma_1=10^4$ and time step $\dt = 0.04625$. The parameter  $\alpha$ increases linearly by \rhn{the amount} $5 \times 10^{-5}$ in each iteration. We see that the compressive nature of the boundary conditions becomes more apparent after $\alpha$ \rhn{exceeds the value} $1/2$ and that the boundary conditions are \rhn{attained} at the end of the deformation. The final isometry defect is $D_h[y_h] = 5.04 \ 10^{-3}$, one order of magnitude smaller than observed in \cite{Bartels}.        
        
\end{example}

\section{Implementation} \label{S:Implementation}
%
\rhn{We now make some implementation remarks and connect them with the theory in Sections \ref{S:DGScheme} and \ref{S:GFlow}. If $(\by\npow,\bla\npow)$ are the nodal values of functions $(\delta y_h^{n+1},\lambda_h^{n+1})$, then the matrix form of the saddle point problem \eqref{E:saddle} reads}
\begin{equation} \label{E:Schur}
	\left[ 
	\begin{matrix}
	A & B_n^T \\
	B_n & 0 
	\end{matrix}	
	\right] \ \ 
	\left[ 
	\begin{matrix}
	\by\npow \\
	\bla\npow
	\end{matrix}
	\right]
	= 
	\left[
	\begin{matrix}
	\bF \\
	\boldsymbol{0}
	\end{matrix}
	\right],
	\end{equation}
where $A$ is the matrix corresponding to the \rhn{left-hand side of \eqref{E:GFlow} and $B_n$ is the matrix associated with \eqref{E:ell}, which depends on $y_h^n$.}
Since $A$ does not depend on $y_h^n$ we can assemble it and perform its LU decomposition once and subsequently use a direct solver whenever we need \rhn{the action of} $A^{-1}$. For the full system we use the Schur complement \rhn{approach with a conjugate gradient iterative solver, to first solve for $\bla\npow$ and next recover $\by\npow$. The numerical experiments of Section \ref{S:NumExPl} reveal the existence of solution $(\delta y_h^{n+1},\lambda_h^{n+1})$ of the discrete gradient flow \eqref{E:saddle}, which exhibits small isometry defects asymptotically as $n$ grows and confirm the validity of an inf-sup condition for \eqref{E:saddle}. However, this issue remains open (see Remark \ref{R:saddle}).}

Lastly, it is important to mention that our choice in \Cref{S:DGScheme} of a space $\Vhk$ of discontinuous polynomials is strongly motivated by \rhn{the structure of \eqref{E:Schur}. More precisely, if  $\delta y_h^{n+1}$ were continuous then the inf-sup condition for $\ell_n$ in \eqref{E:ell} would not be local and be harder to achieve.} In fact, computational experiments for continuous functions (not reported here) indicate that the conjugate gradient method becomes significantly slower if the functions of $\Vhk$ are required to be continuous; this \rhn{justifies} our choice of a fully discontinuous space $\Vhk$. We refer to \cite{Ntogkas} for details.

\section{Conclusions}\label{S:conclusions}

In this work we design, \rhn{analyze} and implement a dG approach to construct minimizers for large bending deformations under a nonlinear isometry constraint. \rhn{The problem is nonconvex and falls within the nonlinear Kirchhoff plate theory. We propose a discrete energy functional and provide a flexible approximation of the isometry constraint. We devise a discrete gradient flow for computing discrete minimizers and enforcing a discrete isometry defect. We construct a discrete approximation of the Hessian inspired by \cite{DiPietroErn,Pryer} that turns out to be instrumental to prove $\Gamma-$convergence: convergence of the discrete energy to the continuous one as well as $L^2$-convergence of global minimizers of the discrete energy to} global minimizers of the continuous energy. The existence of the latter is not assumed a-priory, but is rather a consequence of our analysis. Our dG approach simplifies some implementation details and theoretical constructions needed in \cite{Bartels,BaBoNo} for the Kirchhoff elements. \rhn{The dG formulation is also valid for graded isoparametric elements of degree $k\ge2$, which is considerably more general than \cite{Bartels,BaBoNo}.} Moreover, we present numerical experiments that indicate that the dG approach also captures the physics of some problems better than the Kirchhoff approach, while also giving rise to a more accurate approximation of the isometry constraint.

\rhn{ 
\section{Appendix: Estimates for Isoparametric Mappings}\label{S:A-quads}

In this appendix we present inverse and error estimates involving the Hessian
for elements obtained by isoparametric transformations;
this includes quadrilaterals mapped by bilinear maps $\mQ_1$. The issue at stake
is that the isoparametric map $F_T:\hT\to T$ from the reference element $\hT$ to a generic
element $T\in\mathcal{T}_h$ is no longer affine but rather
$F_T$ belongs to $\lbrack \mP_k(\hT)\rbrack^2$ or $\lbrack \mQ_k(\hT) \rbrack^2$,
whence $D^m F_T \ne \bz$ for $2\le m\le k$. Since the considerations below are
local to a single
element $T$, we drop the sub-index $T$ in $F_T$. 
Moreover, we denote by $\widehat{v}:= v \circ F$ the pullback of a function
$v$ defined on $T$. Property $D^m F \ne \bz$
does not affect the first derivatives of a function $v$ but it does influence
its higher derivatives.
To quantify this effect we recall that we assume that $\{\mathcal{T}_h\}_{h>0}$
is shape-regular 
so that $\| DF \|_{L^\infty(\hT)} \lesssim h_T$ and
$\| DF^{-1} \|_{L^\infty(T)}\lesssim h_T^{-1}$, where the hidden constants depend
on shape regularity of $\mathcal{T}_h$. Furthermore, because $ \|D^2 F\|_{L^\infty(\hT)} \lesssim \| DF \|_{L^\infty(\hT)}$, Lemma 13.4 in Ern and Guermond\cite{ErnGNew} (see also Ciarlet and Raviart\cite{CiarletRaviart}) guarantees that
\begin{equation}\label{E:iso-map}
\|D^m F^{-1}\|_{L^\infty(T)} \lesssim \| DF^{-1} \|_{L^\infty(T)}^m \quad 2\le m\le k+1.
\end{equation}
Our estimates below rely on the following key property of isoparametric maps $F$:
\begin{equation}\label{E:key-prop}
  p\in\mP_1(T) \quad\Rightarrow\quad
  \hp = p\circ F \in \Vhk(\hT).
\end{equation}
If $\hI: C^0(\hT)\to\Vhk(\hT)$ is the Lagrange interpolation operator over $\hT$
of degree $k\ge1$ and $I_h: C^0(T)\to\Vhk(T)$ is the corresponding operator induced
by the map $F$, then \eqref{E:key-prop} translates into
\begin{equation}\label{E:P1-invariance}
I_h p(x) = \hI \hp(\hx) = \hp(\hx) = p(x) \quad\forall \, x := F_T(\hx) \in T, \hx\in\hT.
\end{equation}
If $v\in H^2(T)$, $\hv\in H^2(\hT)$, then the chain rule gives
\begin{gather}\label{E:point-second}
\partial_{ij}^2 v(x) = \sum_{m,n=1}^2 \partial_{mn}^2 \hv(\hx) \, \partial_i F^{-1}_m(x) \, \partial_j F^{-1}_n(x) +
\sum_{m=1}^2 \partial_m \hv(\hx) \, \partial_{ij} F^{-1}_m(x),
\\
\label{E:point-second-inv}
\partial_{ij}^2 \hv(\hx) = \sum_{m,n=1}^2 \partial_{mn}^2 v(x) \, \partial_i F_m(\hx) \, \partial_j F_n(\hx) +
\sum_{m=1}^2 \partial_m v(x) \, \partial_{ij} F_m(\hx).
\end{gather}
This, together with \eqref{E:iso-map},
yields
\begin{gather}\label{E:relation-D2v}
  \|D^2 v\|_{L^2(T)} \lesssim h_T \| DF^{-1} \|_{L^\infty(T)}^2 \| \hv \|_{H^2(\hT)}
  \lesssim h_T^{-1} \| \hv \|_{H^2(\hT)},
  \\
  \label{E:relation-D2hv}
\|D^2 \hv\|_{L^2(\hT)} \lesssim h_T \|D^2 v\|_{L^2(T)} + \|\nabla v\|_{L^2(T)}.
\end{gather}
This gives relations between the Hessians of $v$ and $\hv$ that involve lower
order terms. The next two estimates connect higher order derivatives of isoparametric maps.

\begin{Lemma}[inverse estimate]\label{L:inverse}
  Let $T\in\mathcal{T}_h$ and $e\in\calE_h$ be an edge of $T$. Then the following estimates are valid for all $v_h\in\Vhk$
  \begin{equation}\label{E:inverse}
    \|D^2 v_h\|_{L^2(e)} \lesssim h_e^{-1/2} \|D^2 v_h\|_{L^2(T)},
    \quad
    \|D^3 v_h\|_{L^2(e)} \lesssim h_e^{-3/2} \|D^2 v_h\|_{L^2(T)}.
  \end{equation}
\end{Lemma}
\begin{proof}
  Let $p\in\mP_1(T)$ be a linear polynomial to be chosen later. Combining \eqref{E:point-second} and \eqref{E:iso-map} with an inverse estimate for $\hv_h-\hp \in \Vhk(\hT)$, according to \eqref{E:key-prop}, yields
  \[
  \|D^2 v_h\|_{L^2(e)} = \|D^2 (v_h-p)\|_{L^2(e)} \lesssim
  h_e^{-3/2}\|\hv_h-\hp\|_{H^2(\widehat{e})} \lesssim h_e^{-3/2} \|\hv_h-\hp\|_{H^2(\hT)}.
  \]
  We now map back to $T$ using
  \eqref{E:relation-D2hv} for $D^2(\hv_h-\hp)$ to get
  \[
  \|\hv_h-\hp\|_{H^2(\hT)} \lesssim h_T \|D^2 v_h\|_{L^2(T)} + \|\nabla(v_h-p)\|_{L^2(T)}
  + h_T^{-1} \|v_h-p\|_{L^2(T)}.
  \]
  We finally choose $p$ as the best linear approximation of $v_h$ in $T$, whence
  the Bramble-Hilbert lemma yields
  $\|v_h-p\|_{L^2(T)}+ h_T\|\nabla(v_h-p)\|_{L^2(T)}\lesssim h_T^2\|D^2 v_h\|_{L^2(T)}$
  and gives the first assertion.
  The remaining assertion follows along the same lines upon
  differenciating \eqref{E:point-second} once more and using \eqref{E:iso-map} with $m=2$
  together with the inverse estimate
  $\|\hv_h-\hp\|_{H^3(\widehat{e})}\lesssim \|\hv_h-\hp\|_{H^2(\hT)}$. The proof is
  complete.
\end{proof}

\begin{Lemma}[$H^2$-stability]\label{L:H2-stab}
  Let $v\in H^2(T)$ for $T\in\mathcal{T}_h$ and $I_h v \in \mathbb{V}_h^k(T)$ be the Lagrange interpolant of $v$ of degree $k\ge1$. The following bound is then valid
  \begin{equation}\label{E:H2-stab}
   \|D^2 I_h v\|_{L^2(T)} \lesssim \|D^2 v\|_{L^2(T)}.
  \end{equation}
\end{Lemma}
\begin{proof}
  We first note that $I_h v$ is well defined because $H^2(T)\subset C^0(T)$ (recall
  that $T$ is closed). In view of \eqref{E:P1-invariance}, we see that
  $D^2 I_h v = D^2 (I_h v - p) = D^2 I_h (v - p)$ for any $p\in\mP_1(T)$.
  We utilize \eqref{E:relation-D2v} to deduce
  \[
  \| D^2 I_h v\|_{L^2(T)} \lesssim h_T^{-1} \| \hI (\hv - \hp) \|_{H^2(\hT)}.
  \]
  We choose $p$ to equal $v$ at three vertices of $T$, and observe that $\hw = \hv - \hp$
  vanishes at the corresponding three vertices of $\hT$ and the associated
  linear Lagrange interpolant $\hI_1\hw=0$ vanishes as well.
  Concatenating an inverse estimate with the interpolation error estimate
  $\|\hw-\hI_1\hw\|_{L^\infty(\hT)} \lesssim |\hw|_{H^2(\hT)}$, valid because $H^2(T)\subset L^\infty(T)$, yields
  \[
  \| \hI (\hv - \hp) \|_{H^2(\hT)}  \lesssim \|\hI( \hv - \hp) \|_{L^\infty(\hT)}
  \le \| \hv - \hp \|_{L^\infty(\hT)}
  \lesssim \| D^2 (\hv - \hp) \|_{L^2(\hT)},
  \]
  We finally invoke \eqref{E:relation-D2hv} to infer that
  \[
  \| D^2 (\hv - \hp) \|_{L^2(\hT)} \lesssim h_T \|D^2 (v - p) \|_{L^2(T)}
  + \|\nabla (v - p) \|_{L^2(T)} \lesssim  h_T\|D^2 v \|_{L^2(T)}.
  \]
  This leads to the asserted estimate \eqref{E:H2-stab}.
\end{proof}

We point out that the usual interpolation estimate
\begin{equation}\label{E:BH-T}
  \| v - I_h v \|_{L^2(T)} + h_T \|\nabla( v - I_h v) \|_{L^2(T)}
  \lesssim h_T^2 | v |_{H^2(T)},
\end{equation}
with $H^2$-seminorm on the right-hand side is valid for isoparametric elements with
polynomial degree $k=1$.
In fact, transforming to $\hT$ and back to $T$ via
\eqref{E:relation-D2hv} gives
\[
\| v - I_h v \|_{L^2(T)} + h_T \|\nabla( v - I_h v) \|_{L^2(T)}
  \lesssim h_T^2 \| D^2 v \|_{L^2(T)} + h_T |\nabla v|_{L^2(T)}.
\]
Applying this estimate to $v-p$, with $p\in\mP_1(T)$, and recalling
\eqref{E:P1-invariance} leads to \eqref{E:BH-T}. However, this argument fails for 
$k\ge2$. We now state, and prove for completeness, a key error
estimates for quadrilaterals valid for $k\ge1$
due to Ciarlet and Raviart (see Examples 7 and 8 in \cite{CiarletRaviart}).
We also refer to (13.27) in Ern and Guermond \cite{ErnGNew}.

\begin{Lemma}[error estimate for quadrilaterals]\label{L:error-quad}
Let $T\in\mathcal{T}_h$ be so that $T=F(\hT)$ with $F\in\mathbb{Q}_1$ bilinear.
If $v\in H^{k+1}(T)$ and $I_h v\in\Vhk(T)$ is the Lagrange interpolant of $v$ with
$k\ge 1$, then for $0\le m\le k+1$ there holds
\begin{equation}\label{E:error-quad}
| v - I_h v |_{H^m(T)} \lesssim h_T^{k+1-m} |v|_{H^{k+1}(T)}.
\end{equation}
\end{Lemma}
\begin{proof}
Expression \eqref{E:point-second} in conjunction with \eqref{E:iso-map}
reveals that mapping $D^m(v - I_h v)$ from $T$ to
$\hT$ involves computing all derivatives $D^j(\hv-\hI\hv)$ according to
\begin{equation*}
\|D^m(v - I_h v)\|_{L^2(T)} \lesssim h_T^{1-m} \sum_{j=1}^m \|D^j(\hv-\hI\hv)\|_{L^2(\hT)}
\lesssim h_T^{1-m} \|[ D^{k+1} \hv] \|_{L^2(\hT)},
\end{equation*}
where $[ D^{k+1} \hv] = (\partial_i^{k+1} \hv)_{i=1}^2$ stands for all {\it pure} derivatives
of order $k+1$. The latter inequality is a consequence of the Bramble-Hilbert estimate
for $\mathbb{Q}_k$ elements; see Theorem 1 in \cite{Bramble}. We now resort to a variant
of \eqref{E:point-second-inv} involving $k+1$ derivatives but simplified by the fact
that pure derivatives $\partial_i^j F = 0$ if $j\ge2$ because $F$ is bilinear:
\begin{equation*}
  \partial_i^{k+1} \hv(\hx) = \sum_{n_1,\cdots,n_{k+1}=1}^2 \rhn{\partial^{k+1}_{n_1,\cdots,n_{k+1}} v(x)}
  \, \partial_i F_{n_1}(x) \cdots \partial_i F_{n_{k+1}}(x).
\end{equation*}
Since $\|DF\|_{L^\infty(\hT)} \lesssim h_T$, this yields
\begin{equation}\label{E:Dk+1-quads}
\| [D^{k+1} \hv] \|_{L^2(\hT)} \lesssim h_T^k |v|_{H^{k+1}(T)}
\end{equation}
and combined with the previous estimate gives the asserted estimate \eqref{E:error-quad}.
\end{proof}

Lemma \ref{L:error-quad} extends to isoparametric maps $F\in\mathbb{Q}_k$ for $k\ge1$.
We quote here Theorem 6 of Ciarlet and Raviart \cite{CiarletRaviart}; see also
(13.30) in Ern and Guermond \cite{ErnGNew}.

\begin{Lemma}[error estimates for curved quadrilaterals]\label{L:error-curved-quad}
Let $T\in\mathcal{T}_h$ be so that $T=F(\hT)$ with $F\in\mathbb{Q}_k$ and let
$\widetilde{F}\in\mathbb{Q}_1$ be the bilinear function that maps the vertices of
$\hT$ to those of $T$. Let $F$ satisfy
\begin{equation}\label{E:F-tF}
  \| [D^m F] \|_{L^\infty(\hT)} \lesssim \| D\widetilde{F} \|_{L^\infty(\hT)}^m
  \quad\forall \, 2\le m\le k+1.
\end{equation}
If $v\in H^{k+1}(T)$ and $I_h v\in\Vhk(T)$ is the Lagrange interpolant of $v$ with
$k\ge 1$, then for $0\le m\le k+1$ there holds
\begin{equation}\label{E:error-quad-curved}
| v - I_h v |_{H^m(T)} \lesssim h_T^{k+1-m} \|v\|_{H^{k+1}(T)}.
\end{equation}
\end{Lemma}
}

\rhn{
The proof of Lemma \ref{L:error-curved-quad} is similar to that of Lemma
\ref{L:error-quad} except that 
\begin{equation}\label{E:Dk+1-iso}
\| [D^{k+1} \hv] \|_{L^2(\hT)} \lesssim h_T^k \|v\|_{H^{k+1}(T)}
\end{equation}
replaces \eqref{E:Dk+1-quads}.
This is because $\partial_i^j F \ne 0$ for $j\ge2$ and \eqref{E:F-tF} is
used instead.
}

 \bibliographystyle{amsplain}

\begin{thebibliography}{99}
 
 \bibitem{dealii85}
 {\sc D. Arndt, W. Bangerth, D. Davydov, T. Heister, L. Heltai, M. Kronbichler, M. Maier, J.-P. Pelteret, B. Turcksin, D. Wells},
 {\em The \texttt{deal.II} Library, Version 8.5},
 Journal of Numerical Mathematics, 25(3):137--146, 2017.
 
 \bibitem{ABCM}
 {\sc D. N. Arnold, F. Brezzi, B. Cockburn, L. D. Marini},
 {\em Unified analysis of discontinuous Galerkin methods for elliptic problems},
 SIAM J. Numer. Anal., 39(5), 1749--1779, 2002.

\bibitem{BMN:02}
{\sc E. B\"ansch, P. Morin, R.H. Nochetto},
 {\it An adaptive Uzawa FEM for the Stokes problem: Convergence without the
inf-sup condition}, SIAM J. Numer. Anal. 40, 1207--1229, 2002.
 
 \bibitem{dealiipaper}
 {\sc W. Bangerth, R. Hartmann, G. Kanschat},
 {\em {deal.II} -- a general purpose object oriented finite element library}, ACM Trans. Math. Softw., 33(4):24/1--24/27, 2017.
 
 \bibitem{Bartels}
 {\sc S. Bartels}, 
 {\em Finite element approximation of large bending isometries},
  Numer. Math. 124, 3, 415-440, 2013.
 
 
 \bibitem{BartelsBook}
 {\sc S. Bartels},
 {\em Numerical Methods for Nonlinear Partial Differential Equations},
 Springer International Publishing, 2015.
 
 
 \bibitem{BaBoMuNo}
 {\sc S. Bartels, A. Bonito, Anastasia H. Muliana, R. H. Nochetto},
 {\em Modeling and simulation of thermally actuated bilayer plates},
 J. Comp. Phys. 354, 512--528, 2018.
 
 \bibitem{BaBoNo}
 {\sc S. Bartels, A. Bonito, R. H. Nochetto},
 {\em Bilayer plates: Model reduction, $\Gamma$-convergent finite element approximation and discrete gradient flow}, Commun. Pure Appl. Math.  70, 3, 547-–589, 2017.
 
 \bibitem{BasAbLaGr}
{\sc N. Bassik, B. Abebe, K. Laflin, D. Gracias}, {\em Photolithographically patterned smart hydrogel based
	bilayer actuators}, Polymer 51, 6093-6098, 2010.


 \bibitem{Lewicka}
 {\sc K. Bhattacharya, M. Lewicka, M. Schaffner},
 {\em Plates with incompatible prestrain}, Arch. Rational Mech. Anal. 221, 143-–181, 2016.


 \bibitem{BoNo}
 {\sc A. Bonito, R. H. Nochetto},
 {\em Quasi-optimal convergence rate of an adaptive discontinuous Galerkin method},
 SIAM J. Numer. Anal., 48(2), 734-771, 2010. 
 
 
 \bibitem{Bramble}
 {\sc J.H. Bramble, S.R. Hilbert},
 {\em Bounds for a class of linear functionals with applications to
   {H}ermite interpolation},
 Numer. Math., 16, 362--369, 1970/71. 
 
 
 \bibitem{Brenner1}
 {\sc S. C. Brenner},{\em Poincar\'e-Friedrichs inequalities for piecewise
  $H^1$ functions},
 SIAM J. Numer. Anal., 41(1), 306-324, 2003.
 
 \bibitem{BrennerPrec}
 {\sc S. C. Brenner},{\em Two-level additive {S}chwarz preconditioners for
 nonconforming finite elements},
 Math. Comp., 65(215), 897-921, 1996.
 
 
 \bibitem{Brenner}
 {\sc S. C. Brenner, K. Wang, J. Zhao},{\em Poincar\'e-Friedrichs inequalities
 for piecewise $H^2$ functions},
 Numer. Funct. Anal. Optim., 25, 463-478, 2004.

 \bibitem{Brezzi}
 {\sc F. Brezzi, G. Manzini, D. Marini, P. Pietra, A. Russo},
 {\em Discontinuous {G}alerkin approximations for elliptic problems},
 Numer. Methods Partial Differential Equations, 16(4), 365-378, 2000.
 
 \bibitem{Buffa}
 {\sc A. Buffa, C. Ortner},{\em Compact embeddings of broken Sobolev spaces
 and applications},
 IMA Numer. Anal., 29, 827-855, 2009.

 \bibitem{Cockburn}
 \rhn{{\sc P. Castillo, B. Cockburn, I. Perugia, D. Sch\"{o}tzau},{\em An a priori error analysis of the local discontinuous
              {G}alerkin method for elliptic problems},
 SIAM J. Numer. Anal., 38, 1676--1706, 2000.}
	

 \bibitem{CiarletRaviart}
 {\sc P.G. Ciarlet, P.-A. Raviart},{\em Interpolation theory over curved elements,
 with applications to finite element methods},
 Comput. Methods Appl. Mech. Engrg., 1, 217--249, 1972.

 \bibitem{DiPietroErn}
 {\sc D. A. Di Pietro, A. Ern}, {\em Discrete functional fnalysis tools for 
 	discontinuous Galerkin methods with application to the 
 	incompressible Navier-Stokes equations}, 
 Math. Comp.,  79(271), 1303-1330, 2010.

 \bibitem{ErnG}
 {\sc A. Ern, J-L Guermond}, {\em Finite element quasi-interpolation and best
 approximation},
 ESAIM: M2AN 51, 1367-1385, 2017.
 
 \bibitem{ErnGNew}
 {\sc A. Ern, J-L Guermond}, {\em Finite Elements I: Approximation and interpolation},
 to appear.
 
 \bibitem{FJM}
 {\sc G. Friesecke, R.D. James, S. M\"uller},
 {\em A theorem on geometric rigidity and the derivation of nonlinear
plate theory from three-dimensional elasticity},
 Comm. Pure Appl. Math. 55, 11, 1461-1506, 2002.
 
 \bibitem{JaSmIn}
 {\sc E. Jager, E. Smela, O. Ingan\"as},
 {\em Microfabricating conjugated polymer actuators}, 
 Science 290, 1540-1545, 2000.
 
 \bibitem{Hoppe}
 {\sc R.H.W. Hoppe, B. Wohlmuth},
 {\em Element-oriented and edge-oriented local error estimators for nonconforming finite element methods}, 
 RAIRO Mod\'el. Math. Anal. Num\'er., 30(2), 237-263, 1996.
 
 \bibitem{Hornung}
 {\sc P. Hornung},
 {\em Approximation of flat $W^{2,2}$ isometric immersions by smooth ones}, 
 Arch. Ration. Mech. Anal. 199, 1015- 1067, 2011.
 
 \bibitem{KuLPL}
 {\sc J.-N. Kuo, G.-B. Lee, W.-F. Pan, H.-L. Lee},
 {\em Shape and thermal effects of metal films on stress-induced
 	bending of micromachined bilayer cantilever},
 Japanese Journal of Applied Physics 44, 5R, 3180, 2005.
 
 \bibitem{Lewicka2}
 {\sc M. Lewicka, P. Ochoa, M.-R. Pakzad},
 {\em Variational models for prestrained plates with Monge-Ampère constraint},
  Differential Integral Equations 28, no. 9/10, 861--898, 2015.
 
 \bibitem{Ntogkas}
 {\sc D. Ntogkas}, {\em Non-linear geometric PDEs: algorithms, numerical analysis and computation}, PhD Thesis, University of Maryland, College Park, 2018.

 \bibitem{Oswald}
 {\sc P. Oswald}, {\em On a {BPX}-preconditioner for {${\rm P}1$} elements}, Computing, 51(2), 125-133, 1993. 
 
 \bibitem{Pryer}
 {\sc T. Pryer}, {\em Discontinuous Galerkin methods for the $p-$biharmonic equation from a discrete variational perspective}, Electronic Transactions of Numerical Analysis, 2014. 
 
 \bibitem{Riv}
 {\sc B. Rivi\`ere},
 {\em Discontinuous Galerkin Methods for Solving Elliptic and Parabolic Equations:
 Theory and Implementation}, Society for Industrial and Applied Mathematics, 2008.
 
 \bibitem{SchmidtEb}
 {\sc O. Schmidt, K. Eberl},
 {\em Thin solid films roll up into nanotubes},
 Nature 410, 168, 2001.
 
 \bibitem{SmInLu}
 {\sc E. Smela, O. Ingan\"os, I. Lundstr\"om}, 
 {\em Controlled folding of micrometer-size structures}, Science 268,
 5218, 1735–1738, 1995.
 
\end{thebibliography}

\end{document}